\documentclass[11pt]{article}
\usepackage[utf8]{inputenc}
\usepackage{setspace} 
\usepackage[letterpaper, margin=0.8in]{geometry}
\usepackage{amsmath}
\usepackage{array}
\newcolumntype{P}[1]{>{\centering\arraybackslash}p{#1}}
\newcolumntype{M}[1]{>{\centering\arraybackslash}m{#1}}
\usepackage{xcolor}
\usepackage{amsfonts, amssymb}
\usepackage{amsthm}
\usepackage{mathrsfs}  
\usepackage{mathabx}
\usepackage{cases}
\usepackage{esvect}
\usepackage[english]{babel}
\usepackage[ruled,linesnumbered]{algorithm2e}
\usepackage[T1]{fontenc}
\usepackage{multirow}
\usepackage{graphicx}
\usepackage{cite}
\usepackage[numbers,sort]{natbib}
 \usepackage{titlesec}
 \usepackage{caption}
\titleformat{\section}
  {\normalfont\fontfamily{ptm}\fontsize{11}{11}\bfseries}{\thesection}{1em}{}
\titleformat{\subsection}
  {\normalfont\fontfamily{ptm}\fontsize{10}{11}\bfseries}{\thesubsection}{1em}{}
\titleformat{\subsubsection}
  {\normalfont\fontfamily{ptm}\fontsize{10}{11}\selectfont}{\thesubsubsection}{1em}{}
\newtheorem{theorem}{Theorem}[section]
\newtheorem{corollary}[theorem]{Corollary}
\newtheorem{lemma}[theorem]{Lemma}
\newtheorem{proposition}[theorem]{Proposition}
\newtheorem{definition}[theorem]{Definition}
\theoremstyle{definition}
\newtheorem{remark}[theorem]{Remark}
\newtheorem{example}[theorem]{Example}
\newtheorem{observation}[theorem]{Observation}

\usepackage{hyperref}
\hypersetup{
    colorlinks=true,    
    urlcolor=cyan,
    linkcolor=cyan,
    citecolor=cyan
}
\usepackage{mathtools}
\usepackage{amssymb}
\usepackage{dsfont}
\usepackage{algcompatible}

\providecommand{\keywords}[1]
{
  \small	
  \quad \quad \textbf{\textit{Keywords --}} #1
}

\newcommand{\conv}[1]{\text{conv}\left({#1}\right)}
\newcommand{\cone}[1]{\text{cone}\left({#1}\right)}

\title{\large  Convexification of classes of mixed-integer sets with L$^\natural$-convexity}
\author{ \small Qimeng Yu$^1$ \quad Simge K\"u\c{c}\"ukyavuz$^2$\vspace{0.2cm}  \\ \small $^1$ D\'epartement d’informatique et de recherche op\'erationnelle, Universit\'e de Montr\'eal, Montr\'eal, QC, Canada \\ \small $^2$ Department of Industrial Engineering and Management Sciences, Northwestern University, Evanston, IL, USA \\ \small $^1$\{kimyu@iro.umontreal.ca\}, $^2$\{simge@northwestern.edu\}}
\date{\small \today} 
\begin{document}
\maketitle

\begin{abstract}
\noindent L$^\natural$ (natural)-convex functions encompass a large class of nonlinear functions over general integer domains and arise in a wide range of real-world applications. We explore the minimization of L$^\natural$-convex functions, of multiple L$^\natural$-convex functions with common variables, and of a mixed-integer extension of L$^\natural$-convex functions---functions defined over a mixed-integer domain with properties that resemble L$^\natural$-convexity. For each of these families of minimization problems, we propose valid linear inequalities and provide convex hull descriptions for the corresponding epigraphs. For all classes of proposed inequalities, we discuss their facet conditions, develop exact separation methods, and analyze the complexity of the separation problem. We discover hidden L$^\natural$-convexity in well-known mixed-integer structures in the integer programming literature, namely the (general integer) mixing set and the continuous mixing set. We show that our findings subsume the existing polyhedral results for these sets and establish new results for the multi-capacity variant of the continuous mixing set. 
\end{abstract}

\section{Introduction}
\label{sect:intro}
Let $N = \{1,\dots,n\}$ be a finite non-empty ground set. Throughout, we use  
\[\mathcal{X} := \left\{\mathbf{x} \in \mathbb{Z}^n : \ell_i \leq x_i \leq u_i, \forall\, i \in N \right\}\]
to denote any \emph{discrete hyperrectangle}, where $\ell_i, u_i\in\mathbb{Z} \cup \{\pm\infty\}$ for every $i \in N$. 
Introduced by \citet{murota2003dca}, a function $f:\mathcal{X}\rightarrow\mathbb{R}$ is called \emph{L$^\natural$-convex} if 
\begin{equation}
\label{eq:LC_def_mpc}
f(\mathbf{x}) + f(\mathbf{y}) \geq f\left(\left\lceil\frac{\mathbf{x}+\mathbf{y}}{2}\right\rceil\right) + f\left(\left\lfloor\frac{\mathbf{x}+\mathbf{y}}{2}\right\rfloor\right)
\end{equation}
for every $\mathbf{x}, \mathbf{y}\in \mathcal{X}$, where the rounding operators are component-wise. We provide Figure \ref{fig:midpoint} as a visualization of \eqref{eq:LC_def_mpc}. Property \eqref{eq:LC_def_mpc} is referred to as \emph{mid-point convexity} \citep{moriguchi2020discrete} because of its close analogy to the continuous notion of convexity. Thus, L$^\natural$-convexity is considered a discrete counterpart of classical convexity. \\ 
\begin{figure}[h]
    \centering
    \includegraphics[width=0.95\linewidth]{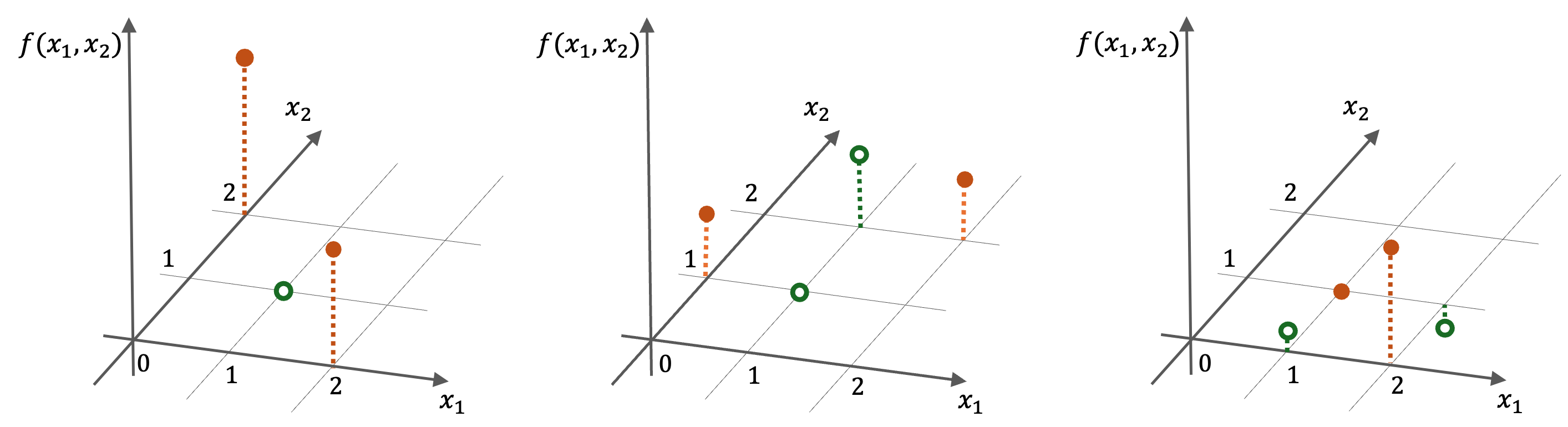}
    \caption{A visual illustration of \eqref{eq:LC_def_mpc}, where $f$ represents an L$^\natural$-convex function, the solid dots represent the function values at arbitrary pairs of elements in the domain, and the hollow dots represent the function values at the discrete midpoint(s).}
    \label{fig:midpoint}
\end{figure}

L$^\natural$-convex functions form a broad class of nonlinear functions. Examples of L$^\natural$-convex functions include, but are not limited to, submodular set functions, quadratic functions with diagonally dominant $M$-matrices, multimodular functions, and max component functions \citep{murota2003dca,murota2005note}. Despite the intuitive analogy with convexity, L$^\natural$-convex functions are \emph{not} equivalent to convex functions restricted to integer domains. In fact, \emph{non-convex} functions over discrete domains can be L$^\natural$-convex. For example, as illustrated in Figure \ref{fig:nonconvex_eg}, $f(\mathbf{x}) = 10x_1^2 - x_2^2$ is nonconvex over $\mathbb{R}^2$ and is L$^\natural$-convex over $\mathcal{X} = \{\mathbf{x}\in\mathbb{Z}^2: \mathbf{0} \leq \mathbf{x} \leq  [2\,\, 1]^\top \}$ \citep{favati1990convexity}.  \\

\begin{figure}[h]
    \centering
    \includegraphics[width=0.8\linewidth]{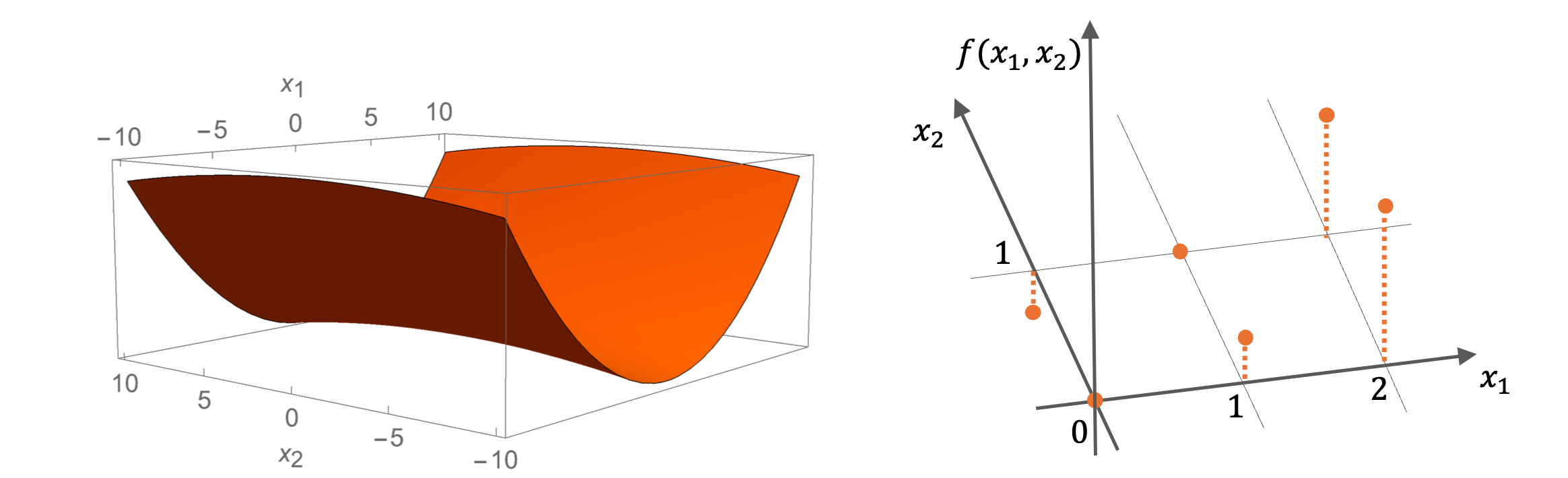}
    \caption{A non-convex function $f(\mathbf{x}) = 10x_1^2 - x_2^2$ (left) and an L$^\natural$-convex function (right) obtained by restricting $f$ to $\mathcal{X} = \{\mathbf{x}\in\mathbb{Z}^2: \mathbf{0} \leq \mathbf{x} \leq  [2\,\, 1]^\top \}$.}
    \label{fig:nonconvex_eg}
\end{figure}

In particular, L$^\natural$-convexity is closely related to submodularity (see Figure \ref{fig:lc_venn}). Submodularity is one of the most important concepts in integer programming and combinatorial optimization. 
A set function $f:2^N\rightarrow\mathbb{R}$ is \emph{submodular} if 
\begin{equation}
\label{eq:submodular_set}
f(S) + f(T) \geq f(S\cup T) + f(S\cap T)
\end{equation}
for all $S, T\subseteq N$. A generalization of submodularity to $f:\mathcal{X} \rightarrow \mathbb{R}$ is the \emph{lattice submodularity}. A function $f:\mathcal{X} \rightarrow \mathbb{R}$ is lattice submodular if 
\begin{equation}
\label{eq:submodular}
f(\mathbf{x}) + f(\mathbf{y}) \geq f\left(\mathbf{x} \vee \mathbf{y}\right) + f\left(\mathbf{x} \wedge \mathbf{y}\right)
\end{equation}
for all $\mathbf{x},\mathbf{y}\in\mathcal{X}$. Here, $\vee$ and $\wedge$ represent the component-wise maximum and minimum between two vectors, respectively. A more precise definition of lattice submodularity is provided in Section \ref{sect:prelim}. Notice that by replacing the subset $S\subseteq N$ by its binary indicator vector, a {submodular set function} is a lattice submodular function with $\mathcal{X} = \{0,1\}^n$. To see the connection with L$^\natural$-convex functions, we give an equivalent definition of 
L$^\natural$-convexity. A function $f:\mathcal{X}\rightarrow\mathbb{R}$ is L$^\natural$-convex {if and only if}
\begin{equation}
\label{eq:LC_def_sub}
f(\mathbf{x}) + f(\mathbf{y}) \geq f\left((\mathbf{x}-\alpha\mathbf{1})\vee \mathbf{y}\right) + f\left(\mathbf{x} \wedge (\mathbf{y}+\alpha\mathbf{1})\right)
\end{equation}
for every $\mathbf{x}, \mathbf{y}\in \mathcal{X}$ and every nonnegative $\alpha\in\mathbb{Z}_+$  such that $(\mathbf{x}-\alpha\mathbf{1})\vee \mathbf{y}, \mathbf{x} \wedge (\mathbf{y}+\alpha\mathbf{1}) \in\mathcal{X}$ \citep{murota2003dca}. 
When $\alpha = 0$, \eqref{eq:LC_def_sub} coincides with \eqref{eq:submodular}, which shows that lattice submodular functions subsume L$^\natural$-convex functions. It is straightforward to see that a {submodular set function} is also an L$^\natural$-convex function over the discrete unit hypercube. Although much progress has been made in optimizing {submodular set functions}, comparable progress in lattice submodular optimization, particularly efficient \emph{exact} solution methods, is lacking. {The Choquet integral (equivalent to the Lov\'asz extension in combinatorial optimization and discrete convex analysis)} is generalized for lattice submodular functions, and an algorithm that relies on additional binary variables to represent general integer variables is given for the minimization problem over box constraints \citep{bach2019submodular}. The convex hull description of the epigraph of lattice submodular functions in the original decision space remains open. Achieving this result for L$^\natural$-convex functions, a broad subclass of lattice submodular functions, will provide invaluable insights into this challenge{; this is exactly one of our objectives in this work.}  \\

\begin{figure}[h]
    \centering
    \includegraphics[width=0.4\linewidth]{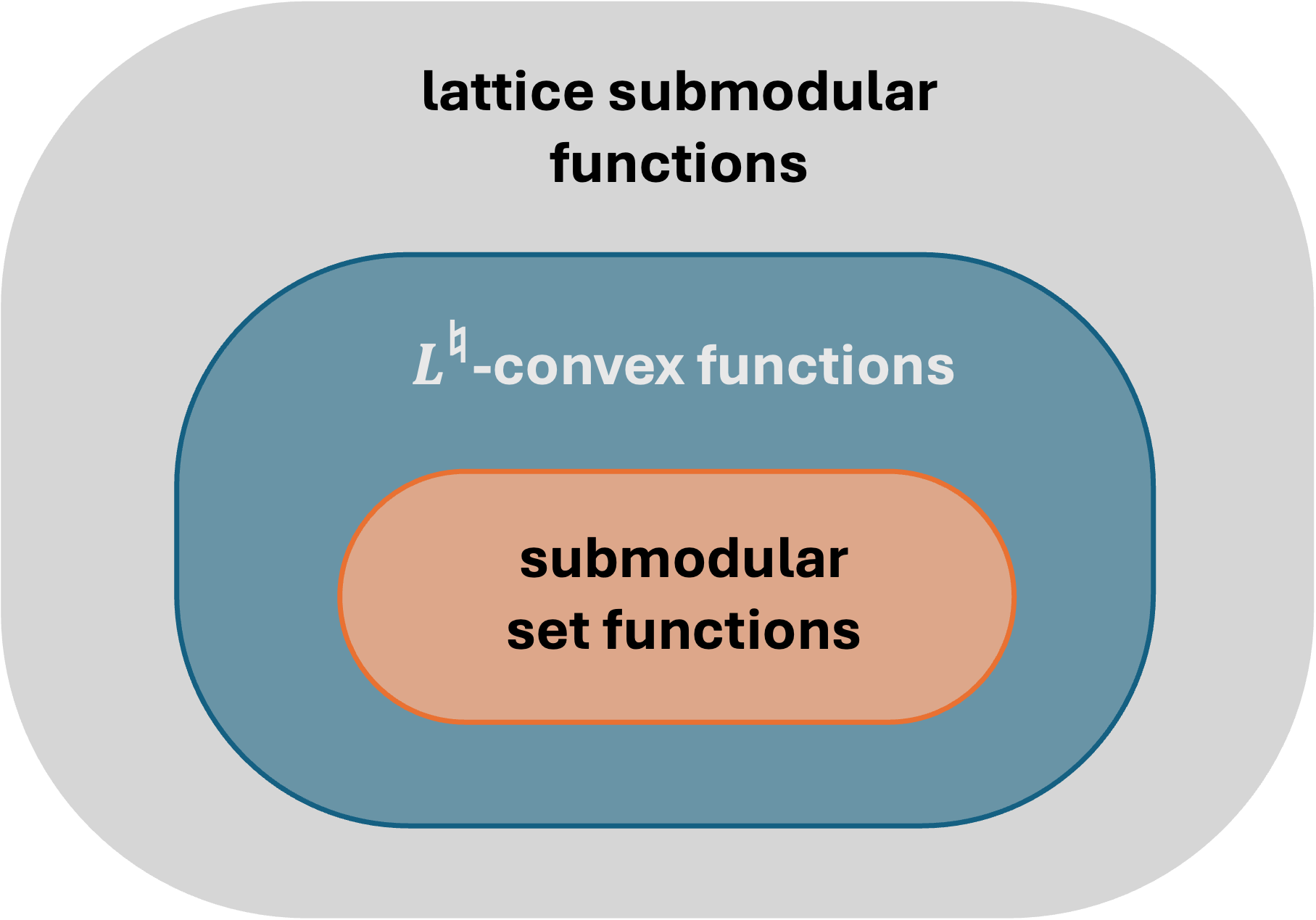}
    \caption{A Venn diagram displaying the relation between L$^\natural$-convexity and submodularity.}
    \label{fig:lc_venn}
\end{figure}

L$^\natural$-convex functions have immense utility in real-world applications, including inventory management \citep{miller1971minimizing,lu2005order,zipkin2008structure,huh2010optimal}, revenue management \citep{chen2017convexity,chen2018preservation},  healthcare \citep{zacharias2020multimodularity}, computer vision \citep{murota2014exact}, {auction theory \citep{murota2016discrete,shioura2017algorithms},} and bike sharing \citep{freund2022minimizing}. It is known that L$^\natural$-convex minimization over discrete hyperrectangles is polynomially solvable. This complexity result follows from the fact that minimizing \emph{integrally convex lattice submodular} functions over discrete hyperrectangles is polynomially solvable \citep{favati1990convexity}, and integrally convex lattice submodular functions are shown to be exactly the L$^\natural$-convex functions \citep{murota2001relationship}. 
Discrete steepest descent algorithms are developed for L$^\natural$-convex function minimization  \citep{murota2004steepest}.
Minimization of a continuous extension of an L$^\natural$-convex function subject to a linear inequality constraint is shown to be polynomially solvable, and an algorithm that relies on binary search for an optimal Lagrangian multiplier is proposed \citep{fujishige2009minimizing}. \\

Nonetheless, these specialized algorithms cannot directly handle arbitrary linearly representable constraints. In contrast to submodular {set} functions, which are naturally extendable to continuous and mixed-integer domains, for which the associated mixed-integer optimization problems have been examined \citep{atamturk2008polymatroids, atamturk2020submodularity, yu2025constrained}, {the study of mixed-integer extensions of L$^\natural$-convex functions is still in its infancy. Notable exceptions include the works on continuous/discrete hybrid convex optimization \citep{takamatsu2004hybrid,moriguchi2007hybrid} and, more recently, on the minimization of symmetric convex functions over a hybrid of continuous and discrete convex sets \citep{kawase2026minimizing}. However, to the best of our knowledge, \emph{polyhedral} studies of the mixed-integer sets associated with L$^\natural$-convex functions and their mixed-integer extensions---which are the key to a versatile cutting-plane methodology---have not been addressed.} To bridge these gaps, we aim to devise a versatile cutting-plane approach to handle arbitrary constraints in L$^\natural$-convex function minimization and to handle both general-integer and continuous decision variables when minimizing extensions of L$^\natural$-convex functions. To achieve these goals, we convexify the underlying mixed-integer sets in these optimization problems through a polyhedral approach. \\

The polyhedral approach has proven highly effective for solving integer programming problems to achieve global optimality. {Submodular set function} optimization is one such example. The pioneering work by 
\citet{edmonds2003submodular} establishes that a specific class of linear inequalities, known as the extremal polymatroid inequalities, completely characterizes the convex hull of the epigraph of any {submodular set function}. Furthermore, 
\citet{wolsey1999integer} introduce the submodular inequalities, which are linear inequalities valid for the hypograph of any {submodular set function}, and they offer a 
mixed 0--1 \emph{linear} program reformulation for the unconstrained submodular maximization. These polyhedral results have subsequently been extended to accommodate 
\emph{constrained} instances of submodular optimization \citep{ahmed2011maximizing,yu2017maximizing,shi2022sequence,yu2017polyhedral,yu2023strong}. 
The polyhedral approach has further demonstrated success in submodular {set function} optimization under stochastic settings \citep{wu2018two, wu2019probabilistic, wu2020exact, kilincc2022joint, xie2021distributionally, zhang2018ambiguous, shen2023chance} as well as minimization of general set functions \citep{atamturk2022submodular}. In addition, a growing body of research investigates how the polyhedral approach can be leveraged for optimizing \emph{extensions} of {submodular set function}s. One natural generalization of submodularity is $k$-submodularity (i.e., functions with $k\geq 2$ set arguments that display diminishing returns). \citet{yu2020polyhedral, yu2021exact} develop efficient exact solution methods for $k$-submodular minimization and maximization by obtaining polyhedral descriptions for the associated epigraph and hypograph. We refer the reader to \citep{kuccukyavuz2023mixed} for a review of generalized submodular optimization. Taken together, these advances motivate exploring the polyhedral approach for L$^\natural$-convex optimization, particularly since L$^\natural$-convex functions extend the class of submodular {set} functions. \\

{A central application of our framework concerns the so-called \emph{mixing sets}. The classical mixing set was introduced by \citet{gunluk2001mixing}, motivated by the \emph{mixing} procedure that combines mixed-integer rounding (MIR) inequalities associated with a common continuous variable; closely related polyhedral results appeared earlier in the study of lot-sizing with Wagner--Whitin costs \citep{pochet1994polyhedra}. Mixing sets and their variants arise as fundamental substructures in production planning and lot-sizing, capacitated facility location, capacitated network design, and joint chance-constrained programming, among others, and an extensive literature is devoted to their polyhedral analysis \citep{gunluk2001mixing,miller2003tight,van2005continuous,zhao2008mixing,constantino2010mixing,Kucukyavuz2012,luedtke2014branch,kilincc2022joint}. In this paper, we show that several of these seemingly disparate polyhedral results can be derived and interpreted in a unified manner by uncovering hidden L$^\natural$-convexity in the underlying sets. In addition, we obtain new convex hull results for the binary multi-capacity continuous mixing set, which has been of interest in MIP literature on capacitated lot-sizing.} {While some classes of mixing sets are not immediately covered by our L$^\natural$-convexity assumptions, our framework provides a basis for their future study. In particular, combining it with polyhedral results for the underlying lattice-submodular functions or with sufficient conditions for simultaneous convexification \citep{tawarmalani2010inclusion} offers natural avenues for extensions (see Sections \ref{sect:mi_extension} and  \ref{sect:conclusion}).}\\

We next summarize our contributions.

 \subsection{Our contributions}
This work presents a comprehensive, in-depth study of the convexification of multiple families of mixed-integer sets associated with L$^\natural$-convex functions. We describe our main contributions as follows. 
\begin{itemize}
    \item[(i)] We establish useful properties of L$^\natural$-convex functions and of closely related classes of functions, namely {submodular set function}s, lattice submodular functions, and continuous submodular functions.
    
    \item[(ii)] We formalize the complete linear description for the epigraph convex hull of any L$^\natural$-convex function $f$ by proposing a class of valid linear inequalities, which we call the shifted extremal polymatroid inequalities (SEPIs). We provide a polynomial-time exact separation algorithm for the SEPIs. We further show that these inequalities are facet-defining and, with trivial bounds when necessary, are sufficient to describe the convex hull of the epigraph of $f$. In particular, we discover the hidden L$^\natural$-convexity in a well-known mixed-integer set in the mixed-integer programming (MIP) literature, called the \emph{mixing set} \citep{gunluk2001mixing}. We show that the existing polyhedral results for the mixing set are direct corollaries of our discussion.
    
    \item[(iii)]  We further consider the joint epigraph of multiple L$^\natural$-convex functions sharing common variables, as well as a variant with additional constraints linking the epigraph variables of these L$^\natural$-convex functions. We provide their complete convex hull descriptions, respectively. We remark that the intersection of the epigraph convex hulls of L$^\natural$-convex functions is exactly the convex hull of the joint epigraph. This appears surprising because the intersection of the convex hulls of multiple mixed-integer sets is generally not the convex hull of the intersections of these mixed-integer sets. This scenario holds for submodular set functions with common binary variables, and our result generalizes it to a much broader class of functions.
    
    \item[(iv)] {L$^\natural$-convex functions are traditionally defined on integer domains, and the study of mixed-integer extensions of L$^\natural$-convex functions, particularly from a polyhedral perspective, remains limited. We thus examine one family of such mixed-integer extensions and provide the convex hull description of its epigraph.} Under certain conditions, such a convex hull is fully described by a class of linear inequalities that we refer to as the mixed-integer SEPIs (MISEPIs). We derive facet-conditions for MISEPIs and show that the separation problem is solvable in polynomial time. This, in turn, proves the complexity of unconstrained minimization of any such mixed-integer function. Lastly, we illustrate our theoretical results using the \emph{continuous mixing set} \citep{van2005continuous} and its capacitated variant. More specifically, we uncover the hidden L$^\natural$-convexity in the general integer continuous mixing set and the binary multi-capacity continuous mixing set, both of which are important mixed-integer structures in the MIP literature. We show that our results subsume the existing polyhedral results for the former. For the latter, to the best of our knowledge, we are the first to present its convex hull description using the MISEPIs and establish the complexity of its unconstrained minimization. 
\end{itemize} 

 \subsection{Outline}
We organize the rest of this paper as follows. 
Section \ref{sect:prelim} provides background information on L$^\natural$-convexity and submodularity. After outlining their known properties, we establish additional characteristics of these functions that are crucial for the later discussions. 
In Section \ref{sect:epi_int_conv}, we examine the epigraph convex hull of any L$^\natural$-convex function. We describe the SEPIs, provide a polynomial-time exact separation algorithm, and show that these inequalities are facet-defining. We then formally argue that the epigraph convex hull of any L$^\natural$-convex function is completely described by SEPIs along with trivial bounds. We elaborate on our discovery of the hidden L$^\natural$-convexity in the \emph{mixing set} and show that its known polyhedral description is a direct corollary of our discussion. 
In Section \ref{sect:multi_LC}, we explore the joint epigraph of multiple L$^\natural$-convex functions sharing common variables and its constrained variant. In Section \ref{sect:mi_extension}, we analyze a class of mixed-integer extensions of L$^\natural$-convexity and give the form of the convex hull description for the epigraph of any such mixed-integer function. We describe the MISEPIs and set forth the conditions under which they fully describe the convex hull. We show in Section \ref{sect:sepa_MISEPIs} that the separation for MISEPIs is solvable in polynomial time, which proves the complexity of minimizing such mixed-integer functions. We further provide the facet conditions for MISEPIs in Section \ref{sect:FD_MISEPI}. Lastly, in Section \ref{sec:structuredH}, we recover the existing polyhedral results for the general integer continuous mixing set and establish new results for the binary multi-capacity continuous mixing set. A closing discussion is included in Section \ref{sect:conclusion}.

\section{Preliminaries} 
\label{sect:prelim}

Recall that $\mathcal{X}$ denotes any bounded or unbounded discrete hyperrectangle. Alternatively, $\mathcal{X} = \prod_{i\in N} \mathcal{X}_i$ where  $\mathcal{X}_i := \{x\in\mathbb{Z}: \ell_i \leq x_i \leq u_i\}$ for $i\in N$. We represent its continuous relaxation by 
\[\overline{\mathcal{X}} := \{\mathbf{x}\in \mathbb{R}^n : \boldsymbol{\ell}\leq \mathbf{x} \leq \boldsymbol{u}\}.\] In addition, we define $\underline{\mathcal{X}} := \{\mathbf{x}\in \mathbb{Z}^n : \ell_i \leq x_i \leq u_i - 1, \forall \, i \in N\}$, and we follow the convention that $\infty -1 = \infty$.
We let 
\[\mathcal{B}_\mathbf{p} := \left\{\mathbf{x} \in \mathbb{Z}^n : p_i \leq x_i \leq p_i+1, \forall\, i \in N \right\}\]
be the \emph{discrete unit hypercube} with a \emph{shifted origin} $\mathbf{p}\in\mathbb{Z}^n$. For example, $\mathcal{B}_\mathbf{0} = \{0,1\}^n$.  
We follow the convention that $\mathbb{Z}_+$ denotes the set of non-negative integers, and $\mathbb{Z}_{++} = \mathbb{Z}_+\setminus\{0\}$.  Let $\mathbf{1}$ be a vector with all ones of appropriate dimension. Given any subset $S\subseteq N$, $\mathbf{1}^S$ is the indicator vector of $S$. If the subset is a singleton, $S=\{i\}$, we use $\mathbf{1}^i$ for ease of notation. \\

In Section \ref{sect:intro}, we have introduced submodular set functions \eqref{eq:submodular_set} and lattice submodular functions \eqref{eq:submodular}. We revisit these functions more rigorously. First, note that set functions are also functions defined over discrete unit hypercubes. Given any submodular set function $g:2^N \rightarrow \mathbb{R}$, we can equivalently state $g$ as $f: \mathcal{B}_\mathbf{0} \rightarrow \mathbb{R}$, where $f(\mathbf{x}) = g(\{i\in N: x_i = 1\})$ for any $\mathbf{x}\in\mathcal{B}_\mathbf{0}$. Conversely, $g(S) = f(\mathbf{1}^S)$ for any $S\subseteq N$. 
More generally, consider $f:\mathcal{D} \subseteq \mathbb{R}^n \rightarrow \mathbb{R}$ where $\mathcal{D}$ is a lattice. Here, by \citet{topkis1978minimizing}, a lattice refers to a partially ordered set (poset) that contains the join ($\vee$) and the meet ($\wedge$) of each pair of its elements. A poset is a set on which there is a binary relation $\leq$ that is reflexive, antisymmetric, and transitive. Moreover, the direct product of lattices is a lattice. Following these definitions, $\overline{\mathcal{X}} = \prod_{i\in N}\overline{\mathcal{X}}_i \subseteq \mathbb{R}^n$, as well as $\mathcal{X} = \prod_{i\in N}\mathcal{X}_i \subseteq \mathbb{Z}^n$, where component-wise natural ordering is the binary relation that defines the poset, are examples of lattices. Below, we provide the formal definition of the extension of submodularity beyond set functions. 

\begin{definition} \citep{topkis1978minimizing}
Consider any function $f:\mathcal{D} \subseteq \mathbb{R}^n \rightarrow \mathbb{R}$, where $\mathcal{D}$ is a lattice. This function is \emph{submodular} if \eqref{eq:submodular} holds 
for all $\mathbf{x},\mathbf{y}\in\mathcal{D}$.
\end{definition}

 The following is an equivalent definition of submodularity that captures an intuition of diminishing returns. 

\begin{definition} \citep{bach2019submodular}
For any $\mathcal{D}$ that is a discrete hyperrectangle $\mathcal{X}$ or a continuous hyperrectangle $\overline{\mathcal{X}}$, 
the function $f:\mathcal{D} \subseteq \mathbb{R}^n \rightarrow \mathbb{R}$ is submodular if 
\begin{equation}
\label{eq:sub_alt_def}
f(\mathbf{x} + a\mathbf{1}^i) - f(\mathbf{x}) \geq f(\mathbf{x} + a\mathbf{1}^i + b\mathbf{1}^j) - f(\mathbf{x}+ b\mathbf{1}^j) 
\end{equation}
for all $\mathbf{x}\in \mathcal{D}$, all $a,b\in\mathbb{R}_+$ and all distinct $i,j\in N$ such that $\mathbf{x} + a\mathbf{1}^i, \mathbf{x}+ b\mathbf{1}^j, \mathbf{x} + a\mathbf{1}^i + b\mathbf{1}^j\in\mathcal{D}$.
\end{definition} 

In the literature, some studies (e.g., \citep{soma2015generalization, soma2018maximizing}) use the term \emph{lattice submodular functions} to refer to submodular functions defined over general integer domains. Other studies simply use the term \emph{submodular functions} regardless of the integrality of the domains (e.g., \citep{bach2019submodular, staib2017robust, bian2017guaranteed, topkis1978minimizing, han2022polynomial}). For clarity, throughout the discussions that follow, we will use {\emph{submodular set functions} to refer to the submodular functions defined over $2^N$ (equivalently, over discrete unit hypercubes)}. We will use \emph{lattice submodular} (lattice submodularity) and \emph{continuous submodular} (continuous submodularity) to distinguish the submodular functions defined over \emph{general integer domains} and \emph{continuous domains}, respectively.  \\

Submodularity is often considered the analog of convexity for functions over discrete unit hypercubes because it is well-known that submodular set functions are efficiently minimizable \citep{grotschel1981ellipsoid, iwata2001combinatorial, lee2015faster, orlin2009faster, cunningham1985submodular, schrijver2000combinatorial, mccormick2005submodular, iwata2009simple, iwata2008submodular}. Moreover, a continuous extension of a set function, called the Lov\'asz extension, is \emph{convex} exactly when the set function is submodular \citep{lovasz1983submodular}. More precisely, let a submodular set function $f:\mathcal{B}_\mathbf{0}\rightarrow\mathbb{R}$ be given. For any $\mathbf{x}\in\overline{\mathcal{B}_\mathbf{0}}=[0,1]^n$, let $\boldsymbol{\delta}\in\mathfrak{S}(N)$ be the permutation such that $x_{\delta(1)} \geq x_{\delta(2)} \geq \dots \geq x_{\delta(n)} \geq x_{\delta(n+1)} \equiv 0$. The Lov\'asz extension of $f$ at $\mathbf{x}$ is 
\begin{equation}
\label{eq:lovasz}
f^L(\mathbf{x}) = \sum_{i=1}^n \left[f\left(\sum_{j=1}^i\mathbf{1}^{\delta(j)}\right) - f\left(\sum_{j=1}^{i-1}\mathbf{1}^{\delta(j)}\right)\right]x_{\delta(i)}.\end{equation}

In fact, the question of whether a function defined over an integral domain has a convex continuous extension has driven \citet{favati1990convexity} to explore the following class of functions. 

\begin{definition} \citep{favati1990convexity}
Given any $f:\mathcal{X}\subseteq \mathbb{Z}^n \rightarrow \mathbb{R}$, consider its continuous extension $\overline{f}:\overline{\mathcal{X}}\rightarrow\mathbb{R}$ defined by 
\[\overline{f}(\mathbf{y}) := \min\left\{\sum_{\mathbf{z}\in\mathcal{N}(\mathbf{y})} \lambda_\mathbf{z} f(\mathbf{z}) :  \sum_{\mathbf{z}\in\mathcal{N}(\mathbf{y})} \lambda_\mathbf{z} \mathbf{z} = \mathbf{y}, \sum_{\mathbf{z}\in\mathcal{N}(\mathbf{y})} \lambda_\mathbf{z} = 1, \lambda_\mathbf{z} \geq 0, \, \forall\, \mathbf{z}\in\mathcal{N}(\mathbf{y})\right\}, \]
for any $\mathbf{y} \in \overline{\mathcal{X}}$, where $\mathcal{N}(\mathbf{y}) := \{\mathbf{z}\in\mathbb{Z}^n : |y_i - z_i| < 1, \forall\, i \in N\}$. 
If $\overline{f}$ is convex, then $f$ is an \emph{integrally convex function} over $\mathcal{X}$. 
\end{definition}

\begin{remark}
\label{re:int_conv_lovasz}
Suppose that $\mathcal{X} = \mathcal{B}_\mathbf{0}$ and $f$ is a {submodular set function}. As suggested in \citep{favati1990convexity}, $\overline{f}$ coincides with the Lov\'asz extension $f^L:\overline{\mathcal{X}} \rightarrow\mathbb{R}$ of $f$, given in \eqref{eq:lovasz} \citep{lovasz1983submodular}. 
\end{remark}

Moreover, L$^\natural$-convex functions are discovered to be closely related to integrally convex functions, as explained in the next theorem. 

\begin{theorem} \citep{fujishige2000notes} 
A function $f:\mathcal{X}\rightarrow\mathbb{R}$ is L$^\natural$-convex if and only if it is integrally convex and lattice submodular. 
\end{theorem}

It follows that an L$^\natural$-convex function $f:\mathcal{B}_\mathbf{0} \rightarrow\mathbb{R}$ must be a submodular {set} function (see Figure \ref{fig:lc_venn}). {Another subclass of L$^\natural$-convex functions that plays an important role in our discussions is the class of \emph{L-convex functions restricted to a hyperrectangle}, which we define next. We remark that, in the standard definition of \citet{murota2003dca}, L-convex functions are defined over $\mathbb{Z}^n$ with linearity in the direction $(1,1,\dots,1)$; the functions in Definition \ref{def:LC} are the restrictions of such functions to a discrete hyperrectangle $\mathcal{X}$. For brevity, and with a slight abuse of terminology, we refer to the functions in Definition \ref{def:LC} simply as L-convex functions in the remainder of the paper.}

\begin{definition} {(L-convex functions restricted to a hyperrectangle, cf.} \citep{murota2003dca}{)}
\label{def:LC}
A function $f:\mathcal{X}\rightarrow\mathbb{R}$ is {\emph{L-convex restricted to a hyperrectangle} (in short, \emph{L-convex})} if \eqref{eq:submodular} holds for all $\mathbf{x},\mathbf{y} \in \mathcal{X}$, and if there exists $r \in \mathbb{R}$ such that $f(\mathbf{x} + \mathbf{1}) = f(\mathbf{x}) + r$ for all $\mathbf{x}\in\mathcal{X}$ with $\mathbf{x} + \mathbf{1}\in \mathcal{X}$.
\end{definition}
The next result establishes the relationship between L-convex functions and L$^\natural$-convex functions. 
\begin{lemma}
\citep{murota2003dca}
\label{thm:LC_LNC}
Any L-convex function $f:\mathcal{X}\rightarrow\mathbb{R}$ satisfies \eqref{eq:LC_def_sub} for every $\mathbf{x}, \mathbf{y}\in \mathcal{X}$ and every $\alpha\in\mathbb{Z}$  such that $(\mathbf{x}-\alpha\mathbf{1})\vee \mathbf{y}, \mathbf{x} \wedge (\mathbf{y}+\alpha\mathbf{1}) \in\mathcal{X}$. 
\end{lemma}
Recall that an L$^\natural$-convex function satisfies \eqref{eq:LC_def_sub} for all $\alpha\in\mathbb{Z}_+$, so every L-convex function is L$^\natural$-convex.

\subsection{Examples and property preserving operations}

In this section, we provide examples of L$^\natural$-convex functions, lattice submodular functions, and continuous submodular functions. We also discuss the transformations of such functions that preserve L$^\natural$-convexity or lattice/continuous submodularity. These properties are crucial for the discussions that follow. We note that all the lemmas in this section are either known or trivially follow from existing results in the literature, so their proofs are omitted or supplemented in the appendix. {To help the reader navigate this section, we highlight where each result is used in the development of our main contributions: Lemmas \ref{lemma:max_LC}--\ref{lemma:partial_min_sub} are building blocks for the new results in Section \ref{sect:new_prelim_results}; Lemma \ref{lemma:G_partialmin} and Proposition \ref{prop:F_lattice_sub} are used in the analysis of the functions $F_{u_0,\mathbf{u}}$ that are introduced to drive the convexification in Section \ref{sect:mi_extension} (see Remark \ref{remark:LC_Fu0u} and Section \ref{sect:F_properties}); Lemma \ref{lemma:gen_int_mix} certifies the hidden L$^\natural$-convexity in the mixing set (Section \ref{sect:epi_int_conv}); and Proposition \ref{prop:max_aff_sub} and Lemma \ref{lemma:max_affine_LC} support the verification of Assumption \hyperlink{A1}{1} in Remark \ref{remark:A1_eg}.}

\begin{lemma}\citep{murota2003dca}
Let $g, h : \mathcal{X}\subseteq \mathbb{Z}^n\rightarrow\mathbb{R}$ be L$^\natural$-convex functions. 
\begin{itemize}
    \item[(1)] For $\alpha\in\mathbb{R}_{++}$, $\alpha g$ is L$^\natural$-convex.
    \item[(2)] For $\mathbf{a}\in\mathbb{Z}^n$ and $\beta \in \mathbb{Z}\setminus \{0\}$, $g(\mathbf{a} + \beta\mathbf{x})$ is L$^\natural$-convex in $\mathbf{x}$. 
    \item[(3)] The sum $g+h$ is also L$^\natural$-convex.
\end{itemize}
\end{lemma}

\begin{lemma} \citep{murota2003dca}
\label{lemma:max_LC}
The maximum-component function $g:\mathbb{Z}^n\rightarrow\mathbb{R}$ defined by $g(\mathbf{x}) := \max_{i\in N}\{x_i\}$ is L-convex. 
\end{lemma}

\begin{lemma} \citep{murota2003dca}
\label{lemma:dc_bivariate}
Suppose that $f:\mathbb{Z}\rightarrow\mathbb{R}$ satisfies $f(x-1) + f(x+1) \geq 2f(x)$. Then $f$ is L$^\natural$-convex. The function $g:\mathbb{Z}^2\rightarrow\mathbb{R}$ defined by $g(\mathbf{x}) = f(x_1 - x_2)$ is L-convex. 
\end{lemma}

\begin{lemma} \citep{bach2019submodular}
\label{lemma:monotone_composite_sub}
Let $F:\overline{\mathcal{X}}\subseteq \mathbb{R}^n \rightarrow\mathbb{R}$  be a continuous submodular function. For every $i\in N$, suppose $h^i:\mathcal{X}_i\rightarrow\mathbb{R}$ is a monotonically decreasing function. Then $G:\mathcal{X} \rightarrow\mathbb{R}$ defined by 
\[G(\mathbf{x}) := F\left(h^1(x_1), h^2(x_2), \dots, h^n(x_n)\right)\]
is lattice submodular. 
\end{lemma}
The same result holds when $h^i$ is a monotonically increasing mapping of $x_i$ for every $i\in N$. \\

Lemmas \ref{lemma:max_LC} and \ref{lemma:dc_bivariate} suggest that the maximum-component functions and the structured bivariate functions are lattice submodular. The next two lemmas formalize the continuous submodularity of these functions. 

\begin{lemma} 
\label{lemma:continuous_max}
The function $g:\overline{\mathcal{X}}\subseteq \mathbb{R}^n\rightarrow\mathbb{R}$  defined by $g(\mathbf{x}) := \max_{i\in N}\{x_i\}$ is continuous submodular. 
\end{lemma}

\begin{lemma}
\label{lemma:uni_conv_sub}
If $f:\mathbb{R}\rightarrow\mathbb{R}$ is convex, then $g:\mathbb{R}^2\rightarrow\mathbb{R}$  where $g(\mathbf{x}) := f(x_1 - x_2)$ is continuous submodular. 
\end{lemma}

The next lemma shows that partial minimization preserves continuous submodularity. 
\begin{lemma} \citep{bach2019submodular}
\label{lemma:partial_min_sub}
Suppose $H: \prod_{i\in N} \overline{\mathcal{X}_i} \subseteq \mathbb{R}^n \rightarrow\mathbb{R}$ is submodular. Let $K = \{1, \dots, k\} \subset N$ for any $k\in N$. The partial minimization function $F: \prod_{i\in K} \overline{\mathcal{X}_i} \rightarrow\mathbb{R}$ defined by 
\[F(x_1, \dots, x_k) := \inf_{\mathbf{z}\in \prod_{i\in N\setminus K} \overline{\mathcal{X}_i}} H(x_1, \dots, x_k, z_1, \dots, z_{n-k})\]
 is submodular. 
\end{lemma}
In fact, $K$ can be any subset of $N$. It is assumed that $K = \{1, \dots, k\}$ without loss of generality up to re-indexing.

\subsection{New results on L$^\natural$-convex and  submodular functions}
\label{sect:new_prelim_results}

In this section, we establish new classes of functions that satisfy submodularity or L$^\natural$-convexity. These functions help our discussion in the later sections. \\

We first consider the function $F_{u_0, \mathbf{u}}:\overline{\mathcal{X}}\subseteq \mathbb{R}^n \rightarrow\mathbb{R}$ for any 
\[(u_0, \mathbf{u}) \in \mathcal{U} := \left\{(u_0, \mathbf{u})\in \mathbb{R}_+^{n+1} : u_0 \leq \sum_{i\in N} u_i\right\},\]
given by 
\begin{equation}
\label{eq:F_u0u_def}
F_{u_0, \mathbf{u}}(\mathbf{x}) := \min_{j\in N} \left\{u_0h^j(\mathbf{x}) + \sum_{i\in N}u_i \max \{0, h^i(\mathbf{x}) - h^j(\mathbf{x})\} \right\},  
\end{equation}
where $h^i :\mathcal{X} \rightarrow\mathbb{R}$ for every $i\in N$. We show that any such function is lattice submodular under certain conditions. Fixing an arbitrary $(u_0, \mathbf{u}) \in \mathcal{U}$, we first discuss the continuous submodularity of a simplified variant of $F_{u_0, \mathbf{u}}$, namely $G_{u_0, \mathbf{u}}:\overline{\mathcal{X}}\subseteq \mathbb{R}^n \rightarrow\mathbb{R}$, 
\begin{equation}
\label{eq:G_u0u_def}
G_{u_0, \mathbf{u}}(\mathbf{x}) := \min_{j\in N} \left\{u_0x_j + \sum_{i\in N}u_i \max \{0, x_i - x_j\} \right\}. 
\end{equation}
We further define a function $H:\overline{\mathcal{X}}\times \mathbb{R} \rightarrow\mathbb{R}$ by 
\begin{equation}
\label{eq:partial_H}
H(\mathbf{x}, t) := u_0 t + \sum_{i\in N}u_i \max \{0, x_i - t\}.
\end{equation}

\begin{lemma}
\label{lemma:partial_H_sub}
For any $(u_0, \mathbf{u}) \in \mathcal{U}$, the function $H$ defined by \eqref{eq:partial_H} is continuous submodular over $\overline{\mathcal{X}}\times \mathbb{R}$. 
\end{lemma}
\begin{proof}
We first notice that $\max\{0, \cdot\}$ is a univariate convex function. Based on Lemma \ref{lemma:uni_conv_sub}, $\max \{0, x_i - t\}$ is continuous submodular for every $i\in N$. The linear function $t$ is trivially continuous submodular. By definition of $\mathcal{U}$, $(u_0, \mathbf{u}) \geq \mathbf{0}$, so $H$ is a non-negative combination of continuous submodular functions, which remains continuous submodular. 
\end{proof}

\begin{lemma}
\label{lemma:G_partialmin}
For any $(u_0, \mathbf{u}) \in \mathcal{U}$ and any $\mathbf{x}\in\overline{\mathcal{X}}$, 
\[G_{u_0, \mathbf{u}}(\mathbf{x}) = \inf_{t\in \mathbb{R}} H(\mathbf{x}, t). \]
\end{lemma} 
\begin{proof}
{For any fixed $\mathbf{x}\in\overline{\mathcal{X}}$, $H(\mathbf{x}, t)$ is a convex, piecewise linear function of $t$, with breakpoints at $x_1, x_2, \dots, x_n$. Its slope in $t$ is $u_0 \geq 0$ for $t > \max_{i\in N}\{x_i\}$ and $u_0 - \sum_{i\in N} u_i \leq 0$ for $t < \min_{i\in N}\{x_i\}$, where the two sign conditions follow from $(u_0,\mathbf{u})\in\mathcal{U}$. Hence, the infimum of $H(\mathbf{x}, t)$ over $t\in\mathbb{R}$ is attained at one of these breakpoints, and
\[\inf_{t\in \mathbb{R}} H(\mathbf{x}, t)  = \min_{t\in \{x_i\}_{i\in N}} H(\mathbf{x}, t)  = \min_{j\in N} \left\{u_0x_j + \sum_{i\in N}u_i \max \{0, x_i - x_j\} \right\}  = G_{u_0, \mathbf{u}}(\mathbf{x}).
\]}
\end{proof}

\begin{proposition}
\label{prop:F_lattice_sub}
The function $F_{u_0, \mathbf{u}}$ given by \eqref{eq:F_u0u_def} is lattice submodular for any $(u_0, \mathbf{u}) \in \mathcal{U}$ when $h_i:\mathcal{X}_i\rightarrow\mathbb{R}$ are monotone for all $i\in N$.
\end{proposition} 
Here, $h_i$ can be monotonically increasing or decreasing, as long as the direction is consistent across $i\in N$. 
\begin{proof}
By Lemma \ref{lemma:partial_H_sub}, $H$ given by \eqref{eq:partial_H} is continuous submodular over $\overline{\mathcal{X}}\times \mathbb{R}$. According to Lemma \ref{lemma:partial_min_sub}, its partial minimization $\inf_{t\in \mathbb{R}} H(\mathbf{x}, t)$ is a continuous submodular function over $\mathbf{x}\in\overline{\mathcal{X}}$. We have shown in Lemma \ref{lemma:G_partialmin} that $\inf_{t\in \mathbb{R}} H(\mathbf{x}, t)$ for any  $\mathbf{x}\in\overline{\mathcal{X}}$ is equal to $G_{u_0, \mathbf{u}}(\mathbf{x})$, so $G_{u_0, \mathbf{u}}$ is continuous submodular. Lastly, observe that $F_{u_0, \mathbf{u}}(\mathbf{x}) = G_{u_0, \mathbf{u}}(h^1(x_1), \dots, h^n(x_n))$ where $h^i$ is monotone for all $i\in N$. We conclude that $F_{u_0, \mathbf{u}}$ is lattice submodular, following from Lemma \ref{lemma:monotone_composite_sub}. 
\end{proof}

\begin{proposition}
\label{prop:F_constant_LNC}
Suppose $F:\mathcal{X}\subseteq \mathbb{Z}^n \rightarrow\mathbb{R}$ satisfies the following conditions. 
\begin{itemize}
\item[(i)] $F$ is decreasing (i.e., $F(\mathbf{x}) \geq F(\mathbf{x}')$ for all $\mathbf{x} \leq \mathbf{x}'$ in $\mathcal{X}$). 
\item[(ii)] There exists $r\in \mathbb{R}$ such that $F(\mathbf{x} + \mathbf{1}) = F(\mathbf{x}) + r$ for all $\mathbf{x}\in\mathcal{X}$. 
 \item[(iii)] $F(\mathbf{x} \wedge \mathbf{y}) = F(\mathbf{x}) \vee F(\mathbf{y})$ for all $\mathbf{x}, \mathbf{y}\in\mathcal{X}$.
\end{itemize}
Then $F$ is L-convex, and the function $F^c:\mathcal{X}\rightarrow\mathbb{R}$,
\[F^c(\mathbf{x}) := F(\mathbf{x}) \vee c\]
is L$^\natural$-convex for any $c\in\mathbb{R}$.
\end{proposition}
{To improve the flow of the presentation, the proof of Proposition \ref{prop:F_constant_LNC} is relegated to the appendix.}

\begin{lemma}{(cf. \citep{murota2009recent})}
\label{lemma:gen_int_mix}
The function $F^0:\mathbb{Z}^n\rightarrow\mathbb{R}$,
\[F^0(\mathbf{x}) := \max_{i\in N} \{q_i - x_i\} \vee 0,\]
where $\mathbf{q}\in\mathbb{R}^n$, is L$^\natural$-convex.
\end{lemma}
{We note that Lemma \ref{lemma:gen_int_mix} follows directly from the known L$^\natural$-convexity of the \emph{maximum-component function} discussed by \citet[p.~233]{murota2009recent}, by setting the vector $\mathbf{p}$ in that discussion equal to $-\mathbf{x}$. We retain a short self-contained proof below, as it illustrates the use of Proposition \ref{prop:F_constant_LNC}.}
\begin{proof}
Let $F(\mathbf{x}) := \max_{i\in N} \{q_i - x_i\}$. We check conditions (i)--(iii) in Proposition \ref{prop:F_constant_LNC}.
\begin{itemize}
\item[(i)] For any $\mathbf{x}, \mathbf{x}'\in\mathbb{Z}^n$ with $\mathbf{x} \leq \mathbf{x}'$, $q_i - x_i \geq q_i - x'_i$ for every $i\in N$. Thus, 
$F(\mathbf{x}) = \max_{i\in N} \{q_i - x_i\} \geq \max_{i\in N} \{q_i - x'_i\} = F(\mathbf{x}')$. 
\item[(ii)] Note that $F(\mathbf{x}+\mathbf{1}) = \max_{i\in N} \{q_i - x_i - 1\} = \max_{i\in N} \{q_i - x_i\} - 1 = F(\mathbf{x}) - 1$ for all $\mathbf{x}\in\mathcal{X}$. 
\item[(iii)] For all $\mathbf{x}, \mathbf{y}\in\mathbb{Z}^n$, 
\begin{align*}
F(\mathbf{x} \wedge \mathbf{y}) & = \max_{i\in N} \{q_i - (x_i \wedge y_i)\} \\
& = \max_{i\in N} \{(q_i - x_i) \vee (q_i - y_i)\} \\
& = \max_{i\in N} \{(q_i - x_i)\} \vee \max_{i\in N} \{(q_i - y_i)\} \\
& = F(\mathbf{x}) \vee F(\mathbf{y}). 
\end{align*}
\end{itemize}
Therefore, it follows from Proposition \ref{prop:F_constant_LNC} that $F^0$ is L$^\natural$-convex. 
\end{proof}

\begin{proposition}
\label{prop:max_aff_sub}
Given $\mathbf{a},\mathbf{b}\in \mathbb R^n$ and $a_0,b_0\in \mathbb R$, the function $F:\mathcal{X}\subseteq \mathbb{Z}^n \rightarrow\mathbb{R}$, 
\[F(\mathbf{x}) := \max \left\{\mathbf{a}^\top\mathbf{x} + a_0, \mathbf{b}^\top\mathbf{x} + b_0 \right\}\]
is lattice submodular if there exist $i, j\in N$ and $\alpha, \beta \geq 0$ such that $\mathbf{a} - \mathbf{b} = \alpha \mathbf{1}^i - \beta\mathbf{1}^j$. 
\end{proposition}
{To improve the flow of the presentation, the proof of Proposition \ref{prop:max_aff_sub} is relegated to the appendix.}

\begin{lemma}
\label{lemma:max_affine_LC}
The function $F:\mathcal{X}\subseteq \mathbb{Z}^n \rightarrow\mathbb{R}$,
\[F(\mathbf{x}) := \max \left\{\mathbf{a}^\top\mathbf{x} + a_0, \mathbf{b}^\top\mathbf{x} + b_0 \right\}\]
is L-convex if there exist $i, j\in N$ and $\alpha \geq 0$ such that $\mathbf{a} - \mathbf{b} = \alpha (\mathbf{1}^i - \mathbf{1}^j)$.
\end{lemma}
\begin{proof}
By Proposition \ref{prop:max_aff_sub}, $F$ is lattice submodular. Let $\mathbf{a}^\top \mathbf{1} = r$. Then $\mathbf{b}^\top \mathbf{1} = r - \alpha (1-1) = r$.  Thus, 
\[F(\mathbf{x} + \mathbf{1}) = \max \left\{\mathbf{a}^\top(\mathbf{x} + \mathbf{1}) + a_0, \mathbf{b}^\top(\mathbf{x} + \mathbf{1}) + b_0 \right\} = \max \left\{\mathbf{a}^\top\mathbf{x} + a_0 + r, \mathbf{b}^\top\mathbf{x} + b_0 + r \right\} = F(\mathbf{x}) + r. \]
We conclude that $F$ is L-convex.
\end{proof}

\begin{remark}
{Proposition \ref{prop:max_aff_sub} and Lemma \ref{lemma:max_affine_LC} are closely related to the maximum-component function discussed in \citet[p.~233]{murota2009recent}; in particular, when $\alpha = \beta$, they can alternatively be derived from the L$^\natural$-convexity of the maximum-component function through suitable affine transformations of the arguments.}
\end{remark}

\section{Epigraphs of L$^\natural$-Convex Functions}
\label{sect:epi_int_conv}

In this section, we formalize the epigraph convex hull description for L$^\natural$-convex functions. Let $f:\mathcal{X}\subseteq \mathbb{Z}^n \rightarrow\mathbb{R}$ be any L$^\natural$-convex function. Recall that $\mathcal{X}$ is a hyperrectangle. We denote the epigraph of $f$ over $\mathcal{X}$ by
\[\mathcal{P}^f_\mathcal{X} := \{{(w, \mathbf{x})\in\mathbb{R}\times\mathcal{X}}: w \geq f(\mathbf{x})\}.\]

{Before presenting the inequalities that describe $\conv{\mathcal{P}^f_\mathcal{X}}$, we position them within the discrete convex analysis literature. It is well known that L-convex and L$^\natural$-convex functions are essentially equivalent concepts, and that they are extensible to convex functions \citep[Section 7.7 and Theorem 7.20]{murota2003dca}. Moreover, the convex extension of an L-convex function---and hence of an L$^\natural$-convex function---is given by a collection of Lov\'asz extensions within each unit hypercube \citep[Theorem 7.19]{murota2003dca}; explicit statements of this formula for L$^\natural$-convex functions appear in Section {4.3} of \citet{murota1998discrete} and Section 16.3 of \citet{fujishige2005submodular}. In a related vein, \citet{he2024discrete} show that the convex envelope, the staircase extension, and the locally Lov\'asz extension of a discrete function over a bounded integer box coincide if and only if the function is L$^\natural$-convex. Consequently, the fact that the linear inequalities \eqref{eq:SEPI} below describe the epigraph convex hull of an L$^\natural$-convex function is essentially a corollary of these known results. Our contributions in this section are complementary to this classical theory: we make the description explicit in the original space of variables in the form of the inequalities \eqref{eq:SEPI}, we provide a polynomial-time exact separation algorithm (Algorithm \ref{alg:frac_greedy}), we show that each such inequality is facet-defining---so the resulting description contains no redundant inequalities (Proposition \ref{prop:SEPI_FD})---and we build the connection between this convex hull description and the general integer mixing set. These developments are also the workhorse of our analysis of the mixed-integer sets in Sections \ref{sect:multi_LC} and \ref{sect:mi_extension}, which constitute the main contributions of this paper.} \\

We {provide} the following class of linear inequalities, {which we refer to} as the \emph{shifted extremal polymatroid inequalities (SEPIs)}. Each SEPI is associated with a point $\mathbf{p}\in \underline{\mathcal{X}}$ and a permutation $\boldsymbol{\delta} = (\delta(1), \delta(2), \dots, \delta(n)) \in \mathfrak{S}(N)$. Such an inequality assumes the following form
\begin{equation}
\label{eq:SEPI}
w \geq f(\mathbf{p}) + \sum_{i=1}^{n} \left[ f\left(\mathbf{p} + \sum_{j=1}^i \mathbf{1}^{\delta(j)} \right) - f\left( \mathbf{p} + \sum_{j=1}^{i-1} \mathbf{1}^{\delta(j)} \right) \right] \left(x_{\delta(i)} - p_{\delta(i)}\right).
\end{equation}
Recall $\mathbf{1}^i\in \{0,1\}^n$ is an indicator vector with one in the $i$th entry and zeros elsewhere. {In \eqref{eq:SEPI}, the summation $\sum_{j=1}^{i-1}\mathbf{1}^{\delta(j)}$ is zero when $i = 1$.}

\begin{remark}
Recall when $\mathcal{X} = \{0,1\}^n$, $f$ is a \emph{submodular set function}. In this case, $\underline{\mathcal{X}} = \{\mathbf{0}\}$, and the SEPIs are exactly the extremal polymatroid inequalities (EPIs), which are facet-defining for the epigraph of $f$ \citep{lovasz1983submodular,atamturk2022submodular,edmonds2003submodular}. Each EPI is associated with a permutation $\boldsymbol{\delta}\in \mathfrak{S}(N)$ and is in the form of
\[w \geq f(\mathbf{0}) + \sum_{i=1}^{n} \left[ f\left(\sum_{j=1}^i \mathbf{1}^{\delta(j)} \right) - f\left(\sum_{j=1}^{i-1} \mathbf{1}^{\delta(j)} \right) \right] x_{\delta(i)}. \]
{This local polyhedral structure over any unit hypercube echoes the locally Lov\'asz construction of the convex extension of an L$^\natural$-convex function. Moreover, L$^\natural$-convexity ensures that these unit-hypercube pieces form a globally convex envelope as discussed in \citet{murota2003dca} and \citet{he2024discrete}.} 
\end{remark}

\begin{proposition} {(cf. \citep{murota2003dca, he2024discrete})}
\label{prop:valid_SEPI}
An SEPI associated with any $\mathbf{p}\in\underline{\mathcal{X}}$ and any $\boldsymbol{\delta}\in\mathfrak{S}(N)$ is valid for $\mathcal{P}^f_\mathcal{X}$. 
\end{proposition}
{To improve the flow of the presentation, the proof of Proposition \ref{prop:valid_SEPI} is relegated to the appendix.} \\

Given any {$(\hat{w}, \hat{\mathbf{x}}) \in (\mathbb{R} \times \overline{\mathcal{X}})\setminus \mathcal{P}^f_\mathcal{X}$}, we propose Algorithm \ref{alg:frac_greedy} to obtain an SEPI that is violated by this infeasible point. 

\vspace{0.2cm}\begin{algorithm}[H]
\caption{\texttt{Fractional\_Greedy}}
\label{alg:frac_greedy}
\begin{algorithmic}[1]
\STATE \textbf{Input} {$(\hat{w}, \hat{\mathbf{x}}) \in \mathbb{R} \times \overline{\mathcal{X}}$}\;
\FOR {$i = 1, 2, \dots, n$}
    \IF {$\hat{x}_i = u_i$}
    \STATE $p_i \leftarrow \hat{x}_i - 1$\;
    \ELSE
    \STATE $p_i \leftarrow \lfloor\hat{x}_i\rfloor$\;
    \ENDIF
\ENDFOR
\STATE Sort the entries in $\mathbf{r} = \hat{\mathbf{x}} - \mathbf{p}$ to obtain a permutation $\boldsymbol{\delta}$ such that $r_{\delta(1)} \geq r_{\delta(2)} \geq \dots \geq r_{\delta(n)}$\;
\STATE  \textbf{Output} An SEPI associated with $\mathbf{p}$ and $\boldsymbol{\delta}$. 
\end{algorithmic}
\end{algorithm}
In fact, Algorithm \ref{alg:frac_greedy} determines the \emph{most} violated SEPI at {$(\hat{w}, \hat{\mathbf{x}})$} as we show in Proposition \ref{prop:exact_sepa_SEPI}. 

\begin{proposition}
\label{prop:exact_sepa_SEPI}
Algorithm \ref{alg:frac_greedy} is an exact separation algorithm for SEPIs with time complexity $\mathcal{O}(n\log n)$.
\end{proposition}
\begin{proof}
Given any infeasible {$(\hat{w}, \hat{\mathbf{x}}) \in \mathbb{R} \times \overline{\mathcal{X}}$}, the separation problem to identify the most violated inequality $\boldsymbol{\pi}^\top \boldsymbol{x}+\pi_0\le w$ is 
\begin{equation}
\label{eq:sepa_SEPI_primal}
\max \{\pi_0 + \hat{\mathbf{x}}^\top \boldsymbol{\pi} : \pi_0 + \mathbf{x}^\top\boldsymbol{\pi}\leq f(\mathbf{x}), \forall\, \mathbf{x}\in\mathcal{X}\}.
\end{equation}
Algorithm \ref{alg:frac_greedy} returns an SEPI associated with $\mathbf{p}$ and $\boldsymbol{\delta}$. Recall for every $i\in N$, 
\[
p_i = 
\begin{cases}
\hat{x}_i - 1, & \text{if } \hat{x}_i = u_i, \\
\lfloor\hat{x}_i\rfloor, & \text{otherwise.}
\end{cases}
\]
Let $(\overline{\pi}_0, \overline{\boldsymbol{\pi}})\in\mathbb{R}^{n+1}$ be the constant term and the coefficients of {the right-hand side of the SEPI \eqref{eq:SEPI} returned by Algorithm \ref{alg:frac_greedy}, viewed as an affine function of $\mathbf{x}$; that is, $\overline{\pi}_{\delta(i)} = f\left(\mathbf{p} + \sum_{j=1}^i \mathbf{1}^{\delta(j)}\right) - f\left(\mathbf{p} + \sum_{j=1}^{i-1} \mathbf{1}^{\delta(j)}\right)$ for $i\in N$, and $\overline{\pi}_0 = f(\mathbf{p}) - \sum_{i=1}^{n}\overline{\pi}_{\delta(i)}\, p_{\delta(i)}$}. Following from Proposition \ref{prop:valid_SEPI}, $(\overline{\pi}_0, \overline{\boldsymbol{\pi}})$ is a feasible solution to problem \eqref{eq:sepa_SEPI_primal}. {We remark that feasibility of $(\overline{\pi}_0, \overline{\boldsymbol{\pi}})$ only certifies validity of the returned SEPI. The purpose of the following duality argument is to certify \emph{exactness} of the separation---that is, among all valid inequalities, and in particular among all SEPIs, the one returned by Algorithm \ref{alg:frac_greedy} attains the maximal violation at {$(\hat{w}, \hat{\mathbf{x}})$}.} The corresponding objective value is
\[\theta_{\text{primal}} = f(\mathbf{p}) + \sum_{i=1}^{n} \left[ f\left(\mathbf{p} + \sum_{j=1}^i \mathbf{1}^{\delta(j)} \right) - f\left( \mathbf{p} + \sum_{j=1}^{i-1} \mathbf{1}^{\delta(j)} \right) \right] \left(\hat{x}_{\delta(i)} - p_{\delta(i)}\right).\]

By assigning the dual variables $\nu_\mathbf{x}$ for all $\mathbf{x}\in\mathcal{X}$, we obtain the following dual problem of \eqref{eq:sepa_SEPI_primal}. 
\begin{equation}
\min \left\{ \sum_{\mathbf{x}\in\mathcal{X}} \nu_\mathbf{x} f(\mathbf{x}) :  \sum_{\mathbf{x}\in\mathcal{X}} \nu_\mathbf{x} \mathbf{x} =  \hat{\mathbf{x}}, \sum_{\mathbf{x}\in\mathcal{X}} \nu_\mathbf{x} = 1, \nu_\mathbf{x} \geq 0, \, \forall\, \mathbf{x} \in \mathcal{X} \right\}. 
\end{equation}
Recall $\mathbf{r} = \hat{\mathbf{x}} - \mathbf{p}$. Next, we construct a dual solution $\overline{\boldsymbol{\nu}}$ by 
\[
\overline{\nu}_\mathbf{x} =
\begin{cases}
1 - r_{\delta(1)}, & \mathbf{x} = \mathbf{p}, \\
r_{\delta(i)} - r_{\delta(i+1)}, & \mathbf{x} = \mathbf{p} + \sum_{j=1}^{i} \mathbf{1}^{\delta(j)}, \forall\, i = 1,2,\dots, n-1, \\
r_{\delta(n)}, & \mathbf{x} = \mathbf{p} + \mathbf{1}, \\
0, & \text{otherwise.}
\end{cases}
\]
By construction, $\sum_{\mathbf{x}\in\mathcal{X}} \overline{\nu}_\mathbf{x} = 1$. For every $k \in \{1,2,\dots,n\}$, 
\begin{align*}
\sum_{\mathbf{x}\in\mathcal{X}} \overline{\nu}_\mathbf{x} x_k & = (1 - r_{\delta(1)}) p_k  + \sum_{i=1}^{\delta^{-1}(k) - 1} (r_{\delta(i)} - r_{\delta(i+1)}) p_k + \sum_{i=\delta^{-1}(k)}^{n-1} (r_{\delta(i)} - r_{\delta(i+1)}) (p_k + 1) \\
& \quad + r_{\delta(n)} (p_k + 1) \\
& = \sum_{\mathbf{x}\in\mathcal{X}} \overline{\nu}_\mathbf{x} p_k + \sum_{i=\delta^{-1}(k)}^{n-1} (r_{\delta(i)} - r_{\delta(i+1)}) + r_{\delta(n)} \\
& = p_k + r_k \\
& = \hat{x}_k.
\end{align*}
Therefore, the proposed dual solution is feasible. The corresponding dual objective value is 
\begin{align*}
\theta_{\text{dual}} & = (1-r_{\delta(1)})f(\mathbf{p}) + \sum_{i=1}^{n-1} \left(r_{\delta(i)} - r_{\delta(i+1)}\right) f\left(\mathbf{p} + \sum_{j=1}^{i} \mathbf{1}^{\delta(j)}\right) + r_{\delta(n)} f(\mathbf{p} + \mathbf{1}) \\
& = f(\mathbf{p}) - r_{\delta(1)} f(\mathbf{p}) + \sum_{i=1}^{n-1} r_{\delta(i)} f\left(\mathbf{p} + \sum_{j=1}^{i} \mathbf{1}^{\delta(j)}\right) - \sum_{i=2}^{n} r_{\delta(i)} f\left(\mathbf{p} + \sum_{j=1}^{i-1} \mathbf{1}^{\delta(j)}\right) \\
& \quad + r_{\delta(n)}f(\mathbf{p}+\mathbf{1}) \\
& =  f(\mathbf{p}) + \sum_{i=1}^{n} r_{\delta(i)} f\left(\mathbf{p} + \sum_{j=1}^{i} \mathbf{1}^{\delta(j)}\right) - \sum_{i=1}^{n} r_{\delta(i)} f\left(\mathbf{p} + \sum_{j=1}^{i-1} \mathbf{1}^{\delta(j)}\right) \\
& = f(\mathbf{p}) + \sum_{i=1}^{n} \left[f\left(\mathbf{p} + \sum_{j=1}^{i} \mathbf{1}^{\delta(j)}\right) - f\left(\mathbf{p} + \sum_{j=1}^{i-1} \mathbf{1}^{\delta(j)}\right)\right] r_{\delta(i)} \\
& = f(\mathbf{p}) + \sum_{i=1}^{n} \left[f\left(\mathbf{p} + \sum_{j=1}^{i} \mathbf{1}^{\delta(j)}\right) - f\left(\mathbf{p} + \sum_{j=1}^{i-1} \mathbf{1}^{\delta(j)}\right)\right] \left(\hat{x}_{\delta(i)} - p_{\delta(i)}\right). 
\end{align*}
Strong duality holds and $(\overline{\pi}_0, \overline{\boldsymbol{\pi}})$ is an optimal solution to problem \eqref{eq:sepa_SEPI_primal}.  Hence, Algorithm \ref{alg:frac_greedy} is exact. The entire algorithm is dominated by the sorting of $\mathbf{r}\in\mathbb{R}^n$, so the time complexity of this separation algorithm is $\mathcal{O}(n\log n)$.
\end{proof}

The following proposition shows that the SEPIs are strong valid inequalities for $\mathcal{P}^f_\mathcal{X}$. 
\begin{proposition}
\label{prop:SEPI_FD}
When $\mathcal{X}\supseteq \mathcal{B}_\mathbf{q}$ for any $\mathbf{q}\in\mathbb{Z}^n$, $\mathcal{P}^f_\mathcal{X}$ is full dimensional, and the SEPI associated with any $\mathbf{p}\in\underline{\mathcal{X}}$ and any $\boldsymbol{\delta}\in\mathfrak{S}(N)$ is facet-defining for $\mathcal{P}^f_\mathcal{X}$. 
\end{proposition}
\begin{proof}
For every $k = 0,1,\dots, n$, let $\mathbf{p}^k = \mathbf{p} + \sum_{j=1}^k\mathbf{1}^{\delta(j)}$. Given that $\mathbf{p}\in\underline{\mathcal{X}}$, $\{\mathbf{p}^k\}_{k\in \{0,1,\dots,n\}} \subseteq \mathcal{X}$. The points { $(f(\mathbf{p})+1, \mathbf{p})$ and $\left( f(\mathbf{p}^k), \mathbf{p}^k\right)$} for $k = 0, 1, \dots, n$ belong to $\mathcal{P}^f_\mathcal{X}$ and are affinely independent. Thus, $\mathcal{P}^f_\mathcal{X}$ is full dimensional. We observe that 
\begin{align*}
f(\mathbf{p}^k) & = f(\mathbf{p})  + \sum_{i=1}^k [f(\mathbf{p}^i) - f(\mathbf{p}^{i-1})] \\
& = f(\mathbf{p}) + \sum_{i=1}^k [f(\mathbf{p}^i) - f(\mathbf{p}^{i-1})](p_{\delta(i)} + 1 - p_{\delta(i)})  + \sum_{i=k+1}^n [f(\mathbf{p}^i) - f(\mathbf{p}^{i-1})](p_{\delta(i)} - p_{\delta(i)}).
\end{align*}
In other words, {$\left(f(\mathbf{p}^k), \mathbf{p}^k\right)$} for all $k = 0,1,\dots, n$ are on the face defined by the arbitrarily given SEPI and are affinely independent. We conclude that the SEPI is facet-defining. 
\end{proof}

In the next theorem, we formalize the complete description of $\conv{\mathcal{P}^f_\mathcal{X}}$. {As discussed at the beginning of this section, the statement of Theorem \ref{thm:conv_SEPI} can alternatively be derived from known results on the convex extensions of L$^\natural$-convex functions \citep{murota1998discrete,murota2003dca,fujishige2005submodular,he2024discrete}. We provide a self-contained proof for completeness, as its constructive nature is used in the proofs of Propositions \ref{prop:inter1} and \ref{prop:inter2}.}
\begin{theorem}
\label{thm:conv_SEPI}
The convex hull of the epigraph, $\mathcal{P}^f_\mathcal{X}$, is fully defined by the trivial inequalities $\boldsymbol{\ell} \leq \mathbf{x} \leq \mathbf{u}$, and all the SEPIs associated with any $\mathbf{p}\in \underline{\mathcal{X}} = \{\mathbf{x}\in \mathbb{Z}^n : \ell_i \leq x_i \leq u_i - 1, \forall \, i = 1,2,\dots,n\}$ and any permutation $\boldsymbol{\delta} \in \mathfrak{S}(N)$.
\end{theorem}
\begin{proof}
We let $\mathcal{C}\subseteq \mathbb{R}^{n+1}$ denote the set constructed by the trivial inequalities and the aforementioned SEPIs. By Proposition \ref{prop:valid_SEPI}, $\mathcal{C} \supseteq \conv{\mathcal{P}^f_\mathcal{X}}$. Next, we establish the reverse containment. Consider arbitrary {$(\eta, \mathbf{y}) \in \mathcal{C}$}. We can apply Algorithm \ref{alg:frac_greedy} to {$(\eta, \mathbf{y})$} because $\mathbf{y}\in \overline{\mathcal{X}}$ due to the trivial inequalities. Suppose Algorithm \ref{alg:frac_greedy} returns an SEPI associated with $\overline{\mathbf{p}}\in\mathcal{X}$ and $\overline{\boldsymbol{\delta}} \in \mathfrak{S}(N)$. In what follows, we let $\lambda_0 = 1- \left(y_{\overline{\delta}(1)} - \overline{p}_{\overline{\delta}(1)}\right)$, $\lambda_i = \left(y_{\overline{\delta}(i)} - \overline{p}_{\overline{\delta}(i)}\right) - \left(y_{\overline{\delta}(i+1)} - \overline{p}_{\overline{\delta}(i+1)}\right)$ for $i \in \{1, 2, \dots, n-1\}$, $\lambda_n = \left(y_{\overline{\delta}(n)} - \overline{p}_{\overline{\delta}(n)}\right)$. We notice that 
\begin{align*}
\sum_{i=0}^n \lambda_i &= 1- \left(y_{\overline{\delta}(1)} - \overline{p}_{\overline{\delta}(1)}\right) + \sum_{i=1}^{n-1} \left[\left(y_{\overline{\delta}(i)} - \overline{p}_{\overline{\delta}(i)}\right) - \left(y_{\overline{\delta}(i+1)} - \overline{p}_{\overline{\delta}(i+1)}\right)\right] + \left(y_{\overline{\delta}(n)} - \overline{p}_{\overline{\delta}(n)}\right)  = 1. 
\end{align*}
Moreover, by line 9 of Algorithm \ref{alg:frac_greedy}, $\lambda_i \geq 0$, for all $i\in\{1,2,\dots,n-1\}$. By lines 2--7 of Algorithm \ref{alg:frac_greedy}, we know that $0\leq y_{\overline{\delta}(1)} - \overline{p}_{\overline{\delta}(1)}, y_{\overline{\delta}(n)} - \overline{p}_{\overline{\delta}(n)} \leq 1$, so $\lambda_0, \lambda_n \geq 0$.  
The point {$(\eta, \mathbf{y})$} satisfies the SEPI returned by Algorithm \ref{alg:frac_greedy}, due to the definition of $\mathcal{C}$. Thus
\begingroup
\allowdisplaybreaks
\begin{align*} 
\eta & \geq f(\overline{\mathbf{p}}) + \sum_{i=1}^{n} \left[f\left(\overline{\mathbf{p}} + \sum_{j=1}^{i} \mathbf{1}^{\overline{\delta}(j)}\right) - f\left(\overline{\mathbf{p}} + \sum_{j=1}^{i-1} \mathbf{1}^{\overline{\delta}(j)}\right)\right] \left(y_{\overline{\delta}(i)} - \overline{p}_{\overline{\delta}(i)}\right) \\
& = f(\overline{\mathbf{p}}) + \sum_{i=1}^{n} f\left(\overline{\mathbf{p}} + \sum_{j=1}^{i} \mathbf{1}^{\overline{\delta}(j)}\right) \left(y_{\overline{\delta}(i)} - \overline{p}_{\overline{\delta}(i)}\right)  - \sum_{i=1}^{n}f\left(\overline{\mathbf{p}} + \sum_{j=1}^{i-1} \mathbf{1}^{\overline{\delta}(j)}\right)\left(y_{\overline{\delta}(i)} - \overline{p}_{\overline{\delta}(i)}\right)  \\
& =  f(\overline{\mathbf{p}}) + f\left(\overline{\mathbf{p}} + \mathbf{1}\right) \left(y_{\overline{\delta}(n)} - \overline{p}_{\overline{\delta}(n)}\right)  + \sum_{i=1}^{n-1} f\left(\overline{\mathbf{p}} + \sum_{j=1}^{i} \mathbf{1}^{\overline{\delta}(j)}\right) \left(y_{\overline{\delta}(i)} - \overline{p}_{\overline{\delta}(i)}\right)  \\
& \quad - f\left(\overline{\mathbf{p}}\right)\left(y_{\overline{\delta}(1)} - \overline{p}_{\overline{\delta}(1)}\right) - \sum_{i=2}^{n} f\left(\overline{\mathbf{p}} + \sum_{j=1}^{i-1} \mathbf{1}^{\overline{\delta}(j)}\right)\left(y_{\overline{\delta}(i)} - \overline{p}_{\overline{\delta}(i)}\right)  \\
& = \left[1- \left(y_{\overline{\delta}(1)} - \overline{p}_{\overline{\delta}(1)}\right)\right] f(\overline{\mathbf{p}}) + \sum_{i=1}^{n-1} \left[\left(y_{\overline{\delta}(i)} - \overline{p}_{\overline{\delta}(i)}\right) - \left(y_{\overline{\delta}(i+1)} - \overline{p}_{\overline{\delta}(i+1)}\right)\right]f\left(\overline{\mathbf{p}} + \sum_{j=1}^{i} \mathbf{1}^{\overline{\delta}(j)}\right) \\
& \quad + \left(y_{\overline{\delta}(n)} - \overline{p}_{\overline{\delta}(n)}\right) f\left(\overline{\mathbf{p}} + \mathbf{1}\right). \\
& = \lambda_0 f(\overline{\mathbf{p}}) + \sum_{i=1}^{n-1} \lambda_i f\left(\overline{\mathbf{p}} + \sum_{j=1}^{i} \mathbf{1}^{\overline{\delta}(j)}\right) + \lambda_n f\left(\overline{\mathbf{p}} + \mathbf{1}\right). 
\end{align*}
\endgroup

We further observe that 
\begin{align*}
\mathbf{y} & = \overline{\mathbf{p}} + \sum_{i=1}^n \left(y_{\overline{\delta}(i)} - \overline{p}_{\overline{\delta}(i)}\right)  \sum_{j=1}^{i} \mathbf{1}^{\overline{\delta}(j)} \\
& = \left[1- \left(y_{\overline{\delta}(1)} - \overline{p}_{\overline{\delta}(1)}\right)\right] \overline{\mathbf{p}} + \sum_{i=1}^{n-1} \left[\left(y_{\overline{\delta}(i)} - \overline{p}_{\overline{\delta}(i)}\right) - \left(y_{\overline{\delta}(i+1)} - \overline{p}_{\overline{\delta}(i+1)}\right)\right]\left(\overline{\mathbf{p}} + \sum_{j=1}^{i} \mathbf{1}^{\overline{\delta}(j)}\right)  \\
& \quad + \left(y_{\overline{\delta}(n)} - \overline{p}_{\overline{\delta}(n)}\right) \left(\overline{\mathbf{p}} + \mathbf{1}\right) \\
& = \lambda_0 \overline{\mathbf{p}} + \sum_{i=1}^{n-1} \lambda_i \left(\overline{\mathbf{p}} + \sum_{j=1}^{i} \mathbf{1}^{\overline{\delta}(j)}\right) + \lambda_n \left(\overline{\mathbf{p}} + \mathbf{1}\right). 
\end{align*}

Following from these observations, an arbitrary {$(\eta, \mathbf{y}) \in \mathcal{C}$} can be written as a convex combination of certain elements in $\mathcal{P}^f_\mathcal{X}$ with a positive multiple of the ray {$(1, \mathbf{0})\in \mathbb{R}^{n+1}$}, namely 
\begin{align*}
(\mathbf{y}, \eta) & = \lambda_0( \overline{\mathbf{p}}, f( \overline{\mathbf{p}})) + \lambda_n ( \overline{\mathbf{p}}+\mathbf{1}, f( \overline{\mathbf{p}}+ \mathbf{1}) ) + \sum_{i=1}^{n-1} \lambda_i\left(\overline{\mathbf{p}} + \sum_{j=1}^{i} \mathbf{1}^{\overline{\delta}(j)}, f\left(\overline{\mathbf{p}} + \sum_{j=1}^{i} \mathbf{1}^{\overline{\delta}(j)}\right)\right) \\
  & \quad + \left(\eta - \lambda_0 f(\overline{\mathbf{p}}) - \sum_{i=1}^{n-1} \lambda_i f\left(\overline{\mathbf{p}} + \sum_{j=1}^{i} \mathbf{1}^{\overline{\delta}(j)}\right) - \lambda_n f\left(\overline{\mathbf{p}} + \mathbf{1}\right)\right)(\mathbf{0}, 1).
 \end{align*}
 Hence, $\mathcal{C}\subseteq \conv{\mathcal{P}^f_\mathcal{X}}$. 
\end{proof}

\begin{remark}
\label{remark:poly_conv}
When finitely many distinct SEPIs are needed for the description of $\conv{\mathcal{P}^f_\mathcal{X}}$, this convex hull is polyhedral. This can happen when $\mathcal{X}$ is bounded; that is $\ell_i > -\infty$ and $u_i < \infty$ for all $i\in N$. Even when $\mathcal{X}$ is unbounded, $\conv{\mathcal{P}^f_\mathcal{X}}$ can still be polyhedral---we will demonstrate this with an example next. 
\end{remark}

\subsection*{An example of $\mathcal{P}^f_\mathcal{X}$ --- the mixing set}
\label{sect:MIX}

The \emph{general integer mixing set} is given by 
\begin{equation}
\mathcal{P}^{\text{MIX}} := \{{(w, \mathbf{x})\in\mathbb{R}_+\times \mathbb{Z}^n} : w + x_i \geq q_i, \forall\, i\in N\},
\end{equation}
where the parameters $\mathbf{q}\in[0,1)^n$ are fixed and sorted such that $1 > q_1 \geq \cdots \geq q_n \geq 0$. This set is an important substructure that arises in lot-sizing, production planning, chance-constrained programming, and vertex packing. \citet{pochet1994polyhedra}, as well as \citet{gunluk2001mixing} establish that the \emph{mixing inequalities}, in addition to the trivial bounds, fully describe $\conv{P^\text{MIX}}$.
Every mixing inequality is associated with a subset $K\subseteq N$ and any of (in case of ties) its corresponding descending order $\boldsymbol{\sigma}\in\mathfrak{S}(K)$, such that $q_{\sigma(1)} \geq \dots \geq q_{\sigma(|K|)}$. For any $K\subseteq N$, the mixing inequalities assume the following two forms: 
\begin{equation}
\label{eq:mix_a}
w \geq \sum_{k=1}^{|K|} (q_{\sigma(k)} - q_{\sigma(k+1)}) ( 1- x_{\sigma(k)}),  
\end{equation}
and 
\begin{equation}
\label{eq:mix_b}
w \geq \sum_{k=1}^{|K|} (q_{\sigma(k)} - q_{\sigma(k+1)}) ( 1- x_{\sigma(k)}) - (1-q_{\sigma(1)})x_{\sigma(|K|)}, 
\end{equation}
where $q_{\sigma(|K|+1)} = 0$. In what follows, we recover the polyhedral result for $\mathcal{P}^{\text{MIX}}$ by exploiting the hidden L$^\natural$-convexity. 

\begin{remark}
The special case of $\mathcal{P}^{\text{MIX}}$ with \emph{binary variables} has been studied extensively as a substructure in problems ranging from chance-constrained programs to vertex packing \citep{atamturk2000mixed, Kucukyavuz2012, luedtke2014branch,Liu2016,Liu2019,kilincc2022joint}.  In particular, \citet{kilincc2022joint} were the first to uncover the hidden submodularity in $\mathcal{P}^{\text{MIX}}$ for binary $\mathbf{x}$ and established the equivalence between the (binary) mixing inequalities and the EPIs. 
\end{remark}

Consider $f:\mathbb{Z}^n\rightarrow\mathbb{R}$ defined by 
\begin{equation}
\label{eq:f_mix}
f(\mathbf{x}) := \max_{i\in N} \{q_i - x_i\} \vee 0, 
\end{equation}
whose epigraph 
\[\mathcal{P}^f_{\mathbb{Z}^n} = \left\{{(w, \mathbf{x})\in\mathbb{R}\times\mathbb{Z}^n} : w \geq f(\mathbf{x})\right\}\]
is exactly $\mathcal{P}^{\text{MIX}}$. 

\begin{lemma}
The function $f$ given by \eqref{eq:f_mix} is L$^\natural$-convex. 
\end{lemma}
\begin{proof}
This lemma follows directly from Lemma \ref{lemma:gen_int_mix}. 
\end{proof}

By Theorem \ref{thm:conv_SEPI}, the SEPIs for $\mathcal{P}^f_{\mathbb{Z}^n}$ associated with all $\mathbf{p}\in\mathbb{Z}^n$ and all $\boldsymbol{\delta}\in\mathfrak{S}(N)$ describe $\conv{\mathcal{P}^{\text{MIX}}}$. We thus obtain the next corollary. 

\begin{corollary}
The mixing inequalities are exactly the SEPIs. 
\end{corollary}

In the next two propositions, we establish the correspondence between the mixing inequalities and the SEPIs.

\begin{proposition}
\label{prop:mix_a}
The mixing inequality of form \eqref{eq:mix_a} with respect to $K\subseteq N$ is identical to the SEPI associated with $\mathbf{p}\in\mathbb{Z}^n$, where 
\[p_i = \begin{cases}
0, & i\in K, \\
1, & \text{otherwise,}
\end{cases}\]
and the natural ordering $\boldsymbol{\delta} = (1,\dots, n)$. 
\end{proposition}
\begin{proof}
Given that $\boldsymbol{\delta}$ is the natural ordering, we drop it from the indices. For any $i\in N$ and $k\in N\cup\{0\}$,
\[\left(\mathbf{p} + \sum_{j=1}^k\mathbf{1}^k\right)_i =\begin{cases}
0, &  \text{if } i\in K\cap \{k+1, \dots, n\}, \\
2, & \text{if } i\in \{1,\dots, k\}\setminus K, \\
1, & \text{otherwise.}
\end{cases}
\]
Thus, for every $k\in N\cup\{0\}$, 
\[f\left(\mathbf{p} + \sum_{j=1}^k\mathbf{1}^j\right) = \begin{cases}
    q_{\min\{j\in K\cap \{k+1,\dots,n\}\}\}}, & \text{if } K\cap \{k+1,\dots,n\} \neq \emptyset, \\
    0, & \text{otherwise.}
\end{cases}\]
Notice further that all $i\in N\setminus K$ are never the maximizer of $f\left(\mathbf{p} + \sum_{j=1}^k\mathbf{1}^j\right)$ for any $k\in N\cup\{0\}$, so all the terms in the SEPI involving $x_i$, $i\in N\setminus K$, are
\[\left[f\left(\mathbf{p} + \sum_{j=1}^{i}\mathbf{1}^j\right) - f\left(\mathbf{p}+ \sum_{j=1}^{i-1}\mathbf{1}^j\right)\right](x_i - 1) = 0(x_i - 1) = 0. \]
Therefore, the SEPI for $f$ associated with $\mathbf{p}$ and $\boldsymbol{\delta}$ is identical to the SEPI for 
\[f'(\mathbf{x}) := \max_{i\in K} \{q_i - x_i\}\vee 0 \]
associated with $\mathbf{p}' = \mathbf{0}$ and $\boldsymbol{\sigma}\in\mathfrak{S}(K)$ such that 
$q_{\sigma(1)} \geq \dots \geq q_{\sigma(|K|)}$. This SEPI is 
\begin{align*}
w & \geq f'\left(\mathbf{0}\right) + \sum_{i=1}^{|K|} \left[f'\left(\sum_{j=1}^{i} \mathbf{1}^{\sigma(j)}\right) - f'\left(\sum_{j=1}^{i-1} \mathbf{1}^{\sigma(j)}\right)\right]x_{\sigma(i)} \\
& = \max_{\ell\in K}\{q_\ell\} + \sum_{i=1}^{|K|} \left[0\vee\max_{j \in\{i+1,\dots,|K|\}}\{q_{\sigma(j)}\} - 0\vee \max_{j \in \{i,\dots,|K|\}}\{q_{\sigma(j)}\}\right]x_{\sigma(i)} \\
& = q_{\sigma(1)} + \sum_{i=1}^{|K|} (q_{\sigma(i+1)} - q_{\sigma(i)})x_{\sigma(i)} \\
 & = \sum_{i=1}^{|K|} (q_{\sigma(i)} - q_{\sigma(i+1)}) + \sum_{i=1}^{|K|} (q_{\sigma(i)} - q_{\sigma(i+1)})(-x_{\sigma(i)}) \\
 & = \sum_{i=1}^{|K|} (q_{\sigma(i)} - q_{\sigma(i+1)})(1-x_{\sigma(i)}). 
\end{align*}
This is exactly \eqref{eq:mix_a}.
\end{proof}

\begin{proposition}
The mixing inequality of form \eqref{eq:mix_b} with respect to $K\subseteq N$ is identical to the SEPI associated with $\mathbf{p}\in\mathbb{Z}^n$, where 
\[p_i = \begin{cases}
-1, & i\in K, \\
1, & \text{otherwise,}
\end{cases}\]
and the natural ordering $\boldsymbol{\delta} = (1,\dots, n)$. 
\end{proposition}
\begin{proof}
With the same argument for Proposition \ref{prop:mix_a}, the SEPI for $f$ associated with $\mathbf{p}$ and $\boldsymbol{\delta}$ is identical to the SEPI for 
\[f'(\mathbf{x}) := \max_{i\in K} \{q_i - x_i\}\vee 0 \]
associated with $\mathbf{p}' = -\mathbf{1}$ and $\boldsymbol{\sigma}\in\mathfrak{S}(K)$ such that 
$q_{\sigma(1)} \geq \dots \geq q_{\sigma(|K|)}$. This SEPI is  
\begingroup
\allowdisplaybreaks
\begin{align*}
w & \geq f'\left(\mathbf{p}'\right) + \sum_{i=1}^{|K|} \left[f'\left(\mathbf{p}' + \sum_{j=1}^{i} \mathbf{1}^{\sigma(j)}\right) - f\left(\mathbf{p}' + \sum_{j=1}^{i-1} \mathbf{1}^{\sigma(j)}\right)\right](x_{\sigma(i)} + 1) \\
& = \max_{\ell\in K}\{1 + q_\ell\} \\
& \quad + \sum_{i=1}^{|K|-1} \left[\max_{j\in \{1,\dots, i\}}\{q_{\sigma(j)}\} \vee \max_{j \in\{i+1,\dots,|K|\}}\{1 + q_{\sigma(j)}\} - \max_{j\in \{1,\dots, i-1\}}\{q_{\sigma(j)}\} \vee \max_{j \in \{i,\dots,|K|\}}\{1 + q_{\sigma(j)}\}\right](x_{\sigma(i)} + 1) \\
& \quad + \left[\max_{\ell\in K}\{q_\ell\} - \max_{j\in \{1,\dots, |K|\}}\{q_{\sigma(j)}\} \vee \{1+ q_{\sigma(|K|)}\}\right](x_{\sigma(|K|)} + 1) \\
& = 1+q_{\sigma(1)} + \sum_{i=1}^{|K|-1} [(1+q_{\sigma(i+1)}) - (1+q_{\sigma(i)})](x_{\sigma(i)} + 1) + [q_{\sigma(1)} - (1+q_{\sigma(|K|)})](x_{\sigma(|K|)} + 1)\\
& = 1+q_{\sigma(1)} + \sum_{i=1}^{|K|-1} (q_{\sigma(i+1)} - q_{\sigma(i)})(x_{\sigma(i)} + 1) + (0 - q_{\sigma(|K|)})(x_{\sigma(|K|)} + 1) + (q_{\sigma(1)} - 1) (x_{\sigma(|K|)} + 1)\\
& = \sum_{i=1}^{|K|} (q_{\sigma(i+1)} - q_{\sigma(i)})(x_{\sigma(i)} + 1) + (q_{\sigma(1)} - 1)x_{\sigma(|K|)} + 2q_{\sigma(1)} \\
& = \sum_{i=1}^{|K|} (q_{\sigma(i+1)} - q_{\sigma(i)})(x_{\sigma(i)} - 1) + 2\sum_{i=1}^{|K|} (q_{\sigma(i+1)} - q_{\sigma(i)}) + (q_{\sigma(1)} - 1)x_{\sigma(|K|)} + 2q_{\sigma(1)} \\
& = \sum_{i=1}^{|K|} (q_{\sigma(i+1)} - q_{\sigma(i)})(x_{\sigma(i)} - 1) - 2q_{\sigma(1)} + (q_{\sigma(1)} - 1)x_{\sigma(|K|)} + 2q_{\sigma(1)} \\
& =  \sum_{i=1}^{|K|} (q_{\sigma(i)} - q_{\sigma(i+1)})(1- x_{\sigma(i)}) - (1- q_{\sigma(1)})x_{\sigma(|K|)}. 
\end{align*}
\endgroup
This is exactly \eqref{eq:mix_b}.
\end{proof}

We include the following example to illustrate the correspondence between mixing inequalities and the SEPIs.

\begin{example}
Consider the instance of $f(\mathbf{x}) = \max\{0, \max\{0.8 - x_1, 0.5 - x_2, 0.2 - x_3\}\}$. Let $K = \{1,2\}$. The corresponding mixing inequalities \eqref{eq:mix_a} and \eqref{eq:mix_b} are 
\[w \geq 0.3(1-x_1) + 0.5(1-x_2) = 0.8 - 0.3x_1 - 0.5x_2,\]
and 
\[w \geq 0.3(1-x_1) + 0.5(1-x_2) - (1-0.8) x_2 = 0.8 - 0.3x_1 - 0.7x_2,\]
respectively. 
Let $\mathbf{p}^1 = (0, 0, 1)$, $\mathbf{p}^2 = (-1, -1, 1)$, and $\boldsymbol{\delta} = (1,2,3)$. The SEPI associated with $\mathbf{p}^1$ and $\boldsymbol{\delta}$ for $\mathcal{P}^f_{{\mathbb{Z}^3}}$ is
\begin{align*}
w & \geq f(0, 0, 1) + [f(1, 0, 1) - f(0,0,1)]x_1 \\
& \quad \quad {+} [f(1,1,1) - f(1,0,1)]x_2 \\
& \quad \quad {+} [f(1,1,2) - f(1,1,1)](x_3-1) \\
& = 0.8 + (0.5 - 0.8)x_1 + (0 - 0.5)x_2 + (0-0)(x_3-1) \\
& = 0.8 - 0.3x_1 - 0.5x_2,
\end{align*}
which is identical to the first mixing inequality. 
The SEPI associated with $\mathbf{p}^2$ and $\boldsymbol{\delta}$ is 
\begin{align*}
w & \geq f(-1, -1, 1) + [f(0, -1, 1) - f(-1, -1, 1)](x_1 + 1) \\
& \quad \quad \quad + [f(0,0,1)-f(0,-1,1)](x_2+1) \\
& \quad \quad \quad +  [f(0,0,2) - f(0,0,1)](x_3-1) \\
& = 1.8 + (1.5-1.8)(x_1 + 1) + (0.8-1.5)(x_2 + 1) + (0.8-0.8)(x_3 - 1) \\
& = 0.8 -0.3x_1 - 0.7x_2.
\end{align*}
This SEPI coincides with the second mixing inequality. 
\end{example}

There are finitely many mixing inequalities, so $\conv{P^{\text{MIX}}}$ is polyhedral. On the other hand, Proposition \ref{prop:SEPI_FD} shows that all SEPIs are facet-defining for $\conv{\mathcal{P}^{\text{MIX}}}$, and there appear to be infinitely many SEPIs due to the unboundedness of $\mathbb{Z}^n$. This seems to be a contradiction. We highlight that, in fact, only finitely many SEPIs are distinct in this case. This is one class of $f$ where $\conv{\mathcal{P}^f_{\mathcal{X}}}$ is polyhedral despite the fact that $\mathcal{X}$ is unbounded, as alluded to in Remark \ref{remark:poly_conv}. We provide an example to illustrate this observation. 

\begin{example}
Consider again $f(\mathbf{x}) = \max\{0, \max\{0.8 - x_1, 0.5 - x_2, 0.2 - x_3\}\}$. Recall that the SEPI associated with $\mathbf{p}^2 = (-1, -1, 1)$ and $\boldsymbol{\delta} = (1,2,3)$ is $w \geq 0.8 -0.3x_1 - 0.7x_2$. 
Let $\mathbf{p}^3 = (-2, -3, -1)$ and $\boldsymbol{\delta}' = (3,2,1)$, the associated SEPI is 
\begin{align*}
w & \geq f(-2,-3,-1) + [f(-2,-3,0) - f(-2,-3,1)](x_3 + 1) \\
& \quad \quad \quad + [f(-2,-2,0)-f(-2,-3,0)](x_2+3) \\
& \quad \quad \quad +  [f(-1,-2,0) - f(-2,-2,0)](x_1+2) \\
& =  3.5 + (3.5 - 3.5) (x_3 + 1) + (2.8 - 3.5)(x_2 + 3) + (2.5 - 2.8)(x_1 + 2) \\
& = 0.8 -0.3x_1 - 0.7x_2.
\end{align*}
Even though $\mathbf{p}^2 \neq \mathbf{p}^3$ and $\boldsymbol{\delta} \neq \boldsymbol{\delta}'$, the corresponding SEPIs are identical. 
\end{example}

We next formally show{, in Proposition \ref{prop:SEPI_mix_correspond},} that only finitely many SEPIs are distinct by relating an arbitrary SEPI to a mixing inequality. Let any $\mathbf{p}\in\mathbb{Z}^n$ and $\boldsymbol{\delta}\in\mathfrak{S}(N)$   be given. We construct $K\subseteq N$ using Algorithm \ref{alg:build_K}. Throughout, we use the notation $\delta^{-1}(i)$, $i\in N$, to denote the ordering of $i$ in $\boldsymbol{\delta}$. In other words, $\delta^{-1}(i) = j$ means that $\delta(j) = i$.

\vspace{0.2cm}\begin{algorithm}[H]
\caption{\texttt{Build\_K}}
\label{alg:build_K}
\begin{algorithmic}[1]
\STATE \textbf{Input} $\mathbf{p}\in\mathbb{Z}^n$,   $\boldsymbol{\delta}\in\mathfrak{S}(N)$ \;
\STATE $p_{\min} \leftarrow \min_{i\in N} \{p_i\}$\;
\IF {$p_{\min} \geq 1$}
    \STATE $K \leftarrow \emptyset$\;
\ELSE
    \STATE $N_{p_{\min}} \leftarrow \{i\in N: p_i = p_{\min}\}$, $k_1 \leftarrow \min N_{p_{\min}}$, $K\leftarrow\{k_1\}$, $\kappa \leftarrow 1$\; 
    \WHILE {$\{i\in N_{p_{\min}}: \delta^{-1}(i) > \delta^{-1}(k_\kappa)\} \neq \emptyset$}
        \STATE $k_{\kappa+1} \leftarrow \min\{i\in N_{p_{\min}}: \delta^{-1}(i) > \delta^{-1}(K_\kappa)\}$, $K\leftarrow K \cup\{k_{\kappa+1}\}$, $\kappa \leftarrow \kappa + 1$\;
    \ENDWHILE
    \STATE $k_L \leftarrow k_{\kappa - 1}$\; 
    \IF {$p_{\min} \leq -1$ }
        \STATE $N_{p_{\min}+1} \leftarrow \{i\in N : p_i = p_{\min} + 1\}$, $\kappa' \leftarrow 0$\;
        \IF {$\{i\in N_{p_{\min}+1} : \delta^{-1}(i) > \delta^{-1}(k_{L}), i < k_1\} \neq \emptyset$}
            \STATE $k'_1 \leftarrow \min\{i\in N_{p_{\min}+1} : \delta^{-1}(i) > \delta^{-1}(k_{L}), i < k_1\}$, $K\leftarrow\{k'_1\}$, $\kappa' \leftarrow 1$\;
            \WHILE {$\{i\in N_{p_{\min}+1} : \delta^{-1}(i) > \delta^{-1}(k'_{\kappa'}), i < k_1\} \neq \emptyset$}
                \STATE $k'_{\kappa'+1} \leftarrow \min\{i\in N_{p_{\min}+1} : \delta^{-1}(i) > \delta^{-1}(k'_{\kappa'}), i < k_1\}$, $K\leftarrow K \cup\{k'_{\kappa'+1}\}$, $\kappa' \leftarrow \kappa' + 1$\;
            \ENDWHILE
           \STATE $k'_{L'} \leftarrow k'_{\kappa'-1}$\; 
        \ENDIF
    \ENDIF
\ENDIF
\STATE  \textbf{Output} $K\subseteq N$. When $K\neq \emptyset$, return the mixing inequality \eqref{eq:mix_a} if $p_{\min} = 0$, \eqref{eq:mix_b} if $p_{\min} \leq -1$. 
\end{algorithmic}
\end{algorithm}

For ease of notation, we let $\mathbf{p}^k = \mathbf{p} + \sum_{j=1}^{k} \mathbf{1}^{\delta(j)}$ for $k\in N\cup\{0\}$. Intuitively, the set $K$ constructed by Algorithm \ref{alg:build_K} is the set of maximizers of $\{f(\mathbf{p}^k)\}_{k=0}^n$. In lines 3--4, $\mathbf{p}^k \geq \mathbf{1}$ and $f(\mathbf{p}^k) = 0$ for all $k\in N\cup\{0\}$. The corresponding SEPI is $w\geq 0$, which maps to the trivial mixing inequality with $K = \emptyset$. Line 6 of Algorithm \ref{alg:build_K} identifies the maximizer $k_1$ of $f(\mathbf{p}) = q_{k_1} - p_{k_1}$. Here, $p_{k_1} = p_{\min} \leq 0$, and $q_{k_1}$ is the largest among all entries $i\in N$ with $p_i = p_{\min}$. In fact, $f(\mathbf{p}^k) = q_{k_1} - p_{k_1}$ for all $k = 0, \dots, \delta^{-1}(k_1)-1$. Lines 7--9 iteratively determine the maximizers $k_\kappa$ of $f(\mathbf{p}^k)$, for $k = \delta^{-1}(k_1), \dots, \delta^{-1}(k_{L})-1$, where $p_{k_\kappa} = p_{\min}$. At $k = \delta^{-1}(k_{L})$, the minimum value in $p^k$ becomes $p_{\min} + 1$. If $p_{\min} = 0$, then $f(\mathbf{p}^k) = 0$ for all $k = \delta^{-1}(k_{L}), \dots, n$. If $p_{\min}\leq -1$ and $N_{p_{\min}+1} = \emptyset$, then $f(\mathbf{p}^k) = q_{k_1} - p_{k_1} - 1$ for all $k = \delta^{-1}(k_{L}), \dots, n$. Otherwise, there could exist $i\in N$ such that $p_{i} = p_{\min} + 1$, that comes after $k_{L}$ according to $\boldsymbol{\delta}$ (i.e., $\delta^{-1}(i) > \delta^{-1}(k_{L})$), and that $q_{i} > q_{k_1}$ (i.e., $i < k_1$). Lines 13--17 iteratively search for such maximizers of $f(\mathbf{p}^k)$, where $k = \delta^{-1}(k_{L}), \dots, n$. We demonstrate how Algorithm \ref{alg:build_K} works in the next example. 

\begin{example}
Consider an instance of $f$, where $\mathbf{q} = (0.9, 0.8, 0.7, 0.6, 0.5,0.4,0.3,0.2,0.1)$. Suppose $\mathbf{p} = (1, 1, 0, 0, 1, -1, 0, 0, -1)$ and $\boldsymbol{\delta}=(1,7,6,2,9,3,8,5,4)$. Here, $p_{\min} = -1$, and $N_{-1} = \{6,9\}$. Thus, $k_1 = 6$. The set $\{i\in N_{-1} : \delta^{-1}(i) > \delta^{-1}(k_1)\} = \{9\}$ because $\delta^{-1}(9) = 5 > 3 = \delta^{-1}(6)$. We obtain $k_{L} = k_2 = 9$.  Note that $N_0 = \{3,4,7,8\}$. Given that $\delta^{-1}(k_{L}) = 5$ and $k_1 = 6$, $k'_1 = \min\{3,4\} = 3$, and $k'_2 = k'_{L'} = 4$. Therefore, $K = \{3,4,6,9\}$. \\

We list the maximizers of $f(\mathbf{p}^k)$, $k = 0, \dots, n$, as follows. 
\begingroup
\allowdisplaybreaks
\begin{align*}
& \mathbf{p}^0 = (1, 1, 0, 0, 1, -1, 0, 0, -1), \quad f(\mathbf{p}^0) = 1.4 = q_6 - p_6, \\
& \mathbf{p}^1 = (2, 1, 0, 0, 1, -1, 0, 0, -1), \quad f(\mathbf{p}^1) = 1.4 = q_6 - p^1_6 = q_6 - p_6, \\
& \mathbf{p}^2 = (2, 1, 0, 0, 1, -1, 1, 0, -1), \quad f(\mathbf{p}^2) = 1.4 = q_6 - p^2_6 = q_6 - p_6, \\
& \mathbf{p}^3 = (2, 1, 0, 0, 1, 0, 1, 0, -1), \quad f(\mathbf{p}^3) = 1.1 = q_9 - p^3_9 = q_9 - p_9,  \quad \delta(3) = 6,\\
& \mathbf{p}^4 = (2, 2, 0, 0, 1, 0, 1, 0, -1), \quad f(\mathbf{p}^4) = 1.1 = q_9 - p^4_9 = q_9 - p_9, \\
& \mathbf{p}^5 = (2, 2, 0, 0, 1, 0, 1, 0, 0), \quad f(\mathbf{p}^5) = 0.7 = q_3 - p^5_3 = q_3 - p_3, \quad \delta(5) = 9,\\
& \mathbf{p}^6 = (2, 2, 1, 0, 1, 0, 1, 0, 0), \quad f(\mathbf{p}^6) = 0.6 = q_4 - p^6_4 = q_4 - p_4, \quad \delta(6) = 3, \\
& \mathbf{p}^7 = (2, 2, 1, 0, 1, 0, 1, 1, 0), \quad f(\mathbf{p}^7) = 0.6 = q_4 - p^7_4 = q_4 - p_4, \\
& \mathbf{p}^8 = (2, 2, 1, 0, 2, 0, 1, 1, 0), \quad f(\mathbf{p}^8) = 0.6 = q_4 - p^8_4 = q_4 - p_4, \\
& \mathbf{p}^9 = (2, 2, 1, 1, 2, 0, 1, 1, 0), \quad f(\mathbf{p}^9) = 0.4 = q_6 - p^9_6 = q_6 - (p_6 + 1), \quad \delta(9) = 4. 
\end{align*}
\endgroup
The maximizers are exactly $K = \{3, 4, 6, 9\}$. The resulting SEPI is 
\begin{align*}
    w & \geq 1.4 + (1.1-1.4)(x_6+1) + (0.7-1.1)(x_9+1) + (0.6-0.7)x_3 + (0.4-0.6)x_4 \\
    & = 0.7 -0.1x_3-0.2x_4-0.3x_6 - 0.4x_9.
\end{align*}
The mixing inequality of form \eqref{eq:mix_b} is also
\begin{align*}
    w & \geq 0.1(1-x_3) + 0.2(1-x_4) + 0.3(1-x_6) + 0.1(1-x_9) - 0.3x_9 \\
    & = 0.7 -0.1x_3 - 0.2x_4 - 0.3x_6 - 0.4x_9. 
\end{align*}
\end{example}

\begin{proposition}
\label{prop:SEPI_mix_correspond}
For any $\mathbf{p}\in\mathbb{Z}^n$ and $\boldsymbol{\delta}\in\mathfrak{S}(N)$, let $K$ be the output of Algorithm \ref{alg:build_K}. The SEPI associated with $\mathbf{p}$ and $\boldsymbol{\delta}$ is
\begin{itemize}
    \item $w \geq 0$, the mixing inequality with respect to $K = \emptyset$, if $p_{\min} \geq 1$; 
    \item \eqref{eq:mix_a} with respect to $K$, if $p_{\min} = 0$; 
    \item \eqref{eq:mix_b} with respect to $K$, otherwise. 
\end{itemize}
\end{proposition}

{To improve the flow of the presentation, the proof of Proposition \ref{prop:SEPI_mix_correspond} is relegated to the appendix.} In sum, we have uncovered the hidden L$^\natural$-convexity in an important mixed-integer structure, namely the general integer mixing set $\mathcal{P}^{\text{MIX}}$. Our results subsume $\conv{\mathcal{P}^{\text{MIX}}}$, and this convex hull serves as an example of $\conv{\mathcal{P}^f_{\mathcal{X}}}$ that is polyhedral (i.e., requires only \emph{finitely} many SEPIs), despite the unbounded domain $\mathcal{X}$.

\section{Mixed-Integer Sets Described by Multiple L$^\natural$-Convex Functions with Common Variables}
\label{sect:multi_LC}

In the previous section, we focused on the mixed-integer set defined by a single L$^\natural$-convex function, namely its epigraph. In what follows, we further explore mixed-integer sets constrained by multiple L$^\natural$-convex functions simultaneously. Throughout, $K = \{1, \dots, k\}$ denotes a non-empty finite index set, and $f_i, g_i:\mathcal{X}\rightarrow\mathbb{R}$ are L$^\natural$-convex functions over a common domain $\mathcal{X}$, for every $i\in K$. The sets that we aim to convexify are as follows: 
\begin{equation}
\label{eq:Q}
\mathcal{Q} := \left\{(\mathbf{w},\mathbf{x})\in \mathbb{R}^k \times \mathcal{X} : w_i \geq f_i(\mathbf{x}), \; \forall\, i\in K\right\}, 
\end{equation}
and 
\begin{equation}
\label{eq:Q'}
\mathcal{Q}' := \left\{(\boldsymbol{\eta}, \mathbf{w},\mathbf{x})\in \mathbb{R}^{2k} \times \mathcal{X} : 
\begin{aligned}
& \eta_i \geq g_i(\mathbf{x}), \; \forall\, i\in K,\\
& w_i \geq f_i(\mathbf{x}), \; \forall\, i\in K, \\
& \eta_i - w_i = \eta_{i'} - w_{i'}, \; \forall\, i, i'\in K
\end{aligned}
\right\}. 
\end{equation}
The first set $\mathcal{Q}$ is defined by $k$ L$^\natural$-convex functions, namely $\{f_i\}_{i\in K}$, which share the same variables $\mathbf{x}\in\mathcal{X}$. The second set $\mathcal{Q}'$ is defined by $k$ \emph{pairs} of L$^\natural$-convex functions, $\{f_i, g_i\}_{i\in K}$, with common $\mathbf{x}\in\mathcal{X}$. This set is further restricted by a set of linear equality constraints, such that the upper bounds $w_i$ and $\eta_i$ on $f_i$ and $g_i$, respectively, have the same difference across all $i\in K$. These two sets play a crucial role in our exploration of a mixed-integer extension of L-$^\natural$-convex functions, detailed in Section \ref{sect:mi_extension}. We present the convex hull results for $\mathcal{Q}$ and $\mathcal{Q}'$ in Proposition \ref{prop:inter1} and Proposition \ref{prop:inter2}, respectively. It is surprising that, under L$^\natural$-convexity, convexifying every individual L$^\natural$-convex function's epigraph and taking their intersection sufficiently defines the convex hull of the mixed-integer sets of interest. \\

\begin{proposition}
\label{prop:inter1}
Recall the set $\mathcal{Q}$ given by \eqref{eq:Q}, where $\{f_i\}_{i\in K}$ are L$^\natural$-convex. Let 
\begin{equation}
\label{eq:CQ}
CQ := \left\{(\mathbf{w}, \mathbf{x}) \in \mathbb{R}^k \times \overline{\mathcal{X}} : (w_i, \mathbf{x})\in \conv{\mathcal{P}^{f_i}_\mathcal{X}}, \forall\, i\in K\right\}. 
\end{equation}
Then $\conv{\mathcal{Q}} = \mathcal{CQ}$. 
\end{proposition}
{To improve the flow of the presentation, the proof of Proposition \ref{prop:inter1} is relegated to the appendix.}

\begin{remark}
Recall the fact that {submodular set function}s form a subclass of L$^\natural$-convex functions. Proposition \ref{prop:inter1} subsumes a result by \citet{baumann2013exact} as a special case, where $\mathcal{X} = \{0,1\}^n$, and $f_i$ is a {submodular set function} for every $i\in K$. {The same result for submodular set functions is also derived by \citet{tawarmalani2010inclusion} following the proposed inclusion certificate for simultaneous convexification. }
\end{remark}

\begin{proposition}
\label{prop:inter2}
Recall the set $\mathcal{Q}'$ defined by \eqref{eq:Q'}, where $\{g_i, f_i\}_{i\in K}$ are L$^\natural$-convex.
Let 
\begin{equation}
\label{eq:CQ'}
CQ' := \left\{(\boldsymbol{\eta}, \mathbf{w},\mathbf{x})\in \mathbb{R}^{2k} \times \overline{\mathcal{X}} : 
\begin{aligned}
& (\eta_i, \mathbf{x}) \in \conv{\mathcal{P}^{g_i}_\mathcal{X}}, \; \forall\, i\in K, \\
& (w_i, \mathbf{x}) \in \conv{\mathcal{P}^{f_i}_\mathcal{X}}, \; \forall\, i\in K, \\
& \eta_i - w_i = \eta_{i'} - w_{i'}, \; \forall\, i, i'\in K
\end{aligned}
\right\}. 
\end{equation}
If there exists $h:\mathcal{X}\rightarrow\mathbb{R}$ with $h(\mathbf{x}) = g_i(\mathbf{x}) - f_i(\mathbf{x})$ for all $\mathbf{x}\in\mathcal{X}$ and $i\in K$, then $\conv{\mathcal{Q}'} = \mathcal{CQ}'$. 
\end{proposition}
\begin{proof}
As a relaxation, $\mathcal{Q}' \subseteq \mathcal{CQ}'$, thus $\conv{\mathcal{Q}'}\subseteq \mathcal{CQ}'$. For the reverse containment, we consider any element $(\overline{\boldsymbol{\eta}}, \overline{\mathbf{w}}, \overline{\mathbf{x}})\in\mathcal{CQ}'$. If $\overline{\mathbf{x}}\in\mathcal{X}$, then $(\overline{\boldsymbol{\eta}}, \overline{\mathbf{w}}, \overline{\mathbf{x}})\in\mathcal{Q}' \subseteq \conv{\mathcal{Q}'}$. Otherwise, $\overline{\mathbf{x}}\in\overline{\mathcal{X}}\setminus \mathcal{X}$. Following the argument for Proposition \ref{prop:inter1}, we obtain $\mathbf{p}\in\underline{\mathcal{X}}$ and $\boldsymbol{\delta}\in\mathfrak{S}(N)$ with respect to $\overline{\mathbf{x}}$. Then we construct 
\[(\boldsymbol{\eta}^j, \mathbf{w}^j, \mathbf{x}^j) = ([g_i(\mathbf{p}^{j})]_{i\in K},\; [f_i(\mathbf{p}^{j})]_{i\in K}, \;\mathbf{p}^{j})\]
 for all $j\in N\cup\{0\}$, and $\boldsymbol{\lambda}\in \mathbb{R}^{n+1}_+$ such that 
\[\boldsymbol{\lambda}^\top\mathbf{1} = 1, \]
\[\overline{\mathbf{x}} = \sum_{j=0}^n \lambda_j \mathbf{x}^j,\]
and 
\[[\overline{\boldsymbol{\eta}} \;\; \overline{\mathbf{w}}] \geq \sum_{j=0}^n \lambda_j [\boldsymbol{\eta}^j \; \mathbf{w}^j].\]
Moreover, for every $j\in N\cup\{0\}$ and any pair of $i, i'\in K$, 
\[\eta^j_i - w^j_i = g_i(\mathbf{p}^{j}) - f_i(\mathbf{p}^{j}) = h(\mathbf{p}^{j}) = g_{i'}(\mathbf{p}^{j}) - f_{i'}(\mathbf{p}^{j}) = \eta^j_{i'} - w^j_{i'},\]
so $(\boldsymbol{\eta}^j, \mathbf{w}^j, \mathbf{x}^j)\in\mathcal{Q}'$. Hence, $\conv{\mathcal{Q}'}\supseteq \mathcal{CQ}'$. With the two-way containment, we conclude that  $\conv{\mathcal{Q}'} = \mathcal{CQ}'$.
\end{proof}

{
\begin{remark}
Propositions \ref{prop:inter1} and \ref{prop:inter2} are closely related to the study of simultaneous convexification via inclusion certificates by \citet{tawarmalani2010inclusion}. Let $P$ be a compact decision set. Corollary 3.9 of \citep{tawarmalani2010inclusion} states that individual convexifications of multiple functions over $P$ are sufficient for simultaneous convexification if, for every point in $\conv{P}$, the convex envelopes of all involved functions are attained by the same inclusion certificate, or intuitively, by the same convex combination of elements in $P$. 
This condition holds for L$^\natural$-convex functions over a common domain: their convex envelopes coincide with their locally Lov\'asz extensions \citep[Corollary 2]{he2024discrete}. In our construction, Algorithm \ref{alg:frac_greedy} produces exactly this common convex combination for any $\overline{\mathbf{x}}\in\overline{\mathcal{X}}$, which is independent of the particular L$^\natural$-convex function. Thus, when $\mathcal{X}$ is bounded, Proposition \ref{prop:inter1} can be recovered from \citep{tawarmalani2010inclusion}. Nonetheless, we provide direct and self-contained proofs of Propositions \ref{prop:inter1} and \ref{prop:inter2} because they rely only on the machinery of Section \ref{sect:epi_int_conv} (in particular, Algorithm \ref{alg:frac_greedy}). They hold when $\mathcal{X}$ is unbounded and cover the important variant $\mathcal{Q}'$, later used in Section \ref{sect:mi_extension}, whose additional linking constraints among the epigraph variables are not covered directly by \citep{tawarmalani2010inclusion}. 
\end{remark}
}

Until now, we have explored the mixed-integer sets described by one or more L$^\natural$-convex functions, which are defined over \emph{pure integer} domains. In what follows, we will examine an extension of L$^\natural$-convexity that concerns the functions defined over \emph{mixed-integer} domains.

\section{A Mixed-Integer Extension of L$^\natural$-Convexity}
\label{sect:mi_extension}

To the best of our knowledge, {the polyhedral aspects of extensions of L$^\natural$-convexity to functions with mixed-integer arguments have received little attention.} Thus, we focus on a structured class of functions with properties that resemble L$^\natural$-convexity and derive characterizations of their epigraphs. In particular, we consider the class of functions $H:\mathcal{X}\times{\mathbb{R}^n_+} \rightarrow\mathbb{R}$ defined by
\begin{equation}
\label{eq:H}
H(\mathbf{x}, \mathbf{y}):= \max_{i\in N}\{h^i(\mathbf{x}) - y_i\}, 
\end{equation}
where $h^i:\mathcal{X}\rightarrow\mathbb{R}$ for all $i\in N$. Here, $\mathbf{x}$ is a vector of pure integer variables.  In contrast to the previous sections, we allow additional variables $\mathbf{y}$ to be continuous. Our goal is to understand the epigraph of $H$, namely 
\[\mathcal{P}^H_{\mathbb{R}_+^n\times\mathcal{X}} := \left\{(w, \mathbf{y}, \mathbf{x})\in\mathbb{R}\times\mathbb{R}_+^n\times\mathcal{X} : w \geq h^i(\mathbf{x}) - y_i, \forall\,i\in N\right\}.\] 
For brevity, we use $\mathcal{P}^H_{\mathbb{R}_+^n\times\mathcal{X}}$ and $\mathcal{P}^H$ interchangeably.\\

{In the rest of this section, we give a three-stage convexification framework to convexify $\mathcal{P}^H$. (i) We represent $\mathcal{P}^H$ as a disjunction of the sets $\mathcal{D}^j$, $j\in N$. (ii) We convexify each $\mathcal{D}^j$ through the convexification result of joint epigraphs in Section \ref{sect:multi_LC}; and (iii) we convexify the disjunction through a polar characterization combined with an aggregation argument. We impose an L$^\natural$-convexity assumption (Assumption \hyperlink{A1}{1}) to make step (ii) tractable; in general, step (ii) may be realized by other sufficient conditions for simultaneous convexification, such as the inclusion-certificate conditions of \citet{tawarmalani2010inclusion}. As concrete consequences, this framework recovers the known polyhedral results for the continuous mixing polyhedron. It further yields new results for the binary multi-capacity continuous mixing set.}

\subsection{Properties of $\mathcal{P}^H$ for general $H$}
\label{sect:properties_PH}

In this section, we explore the properties of the epigraph, $\mathcal{P}^H$, without additional assumptions on $H$. We first consider the following polyhedron: 
\begin{equation}
\label{eq:Dx}
\mathcal{D}(\overline{\mathbf{x}}) := \left\{(w, \mathbf{y}) \in \mathbb{R}\times \mathbb{R}^n_+ : w \geq h^i(\overline{\mathbf{x}}) - y_i, \, \forall\, i\in N \right\}, 
\end{equation}
for every $\overline{\mathbf{x}}\in\mathcal{X}$. Note that 
\[\mathcal{P}^H = \left\{(w, \mathbf{y}, \mathbf{x})\in \mathbb{R}^{n+1}\times\mathcal{X} : (w, \mathbf{y}) \in \mathcal{D}(\mathbf{x})\right\}.\] 
For every $j\in N$, we further define $(w^j(\overline{\mathbf{x}}), \mathbf{y}^j(\overline{\mathbf{x}}))$ where
\begin{equation}
\label{eq:wyj}
w^j(\overline{\mathbf{x}}) := h^j(\overline{\mathbf{x}}), \quad y^j_i(\overline{\mathbf{x}}) := \max\{0, h^i(\overline{\mathbf{x}}) - h^j(\overline{\mathbf{x}})\},\, \forall\, i\in N.
\end{equation}
This arbitrary $\overline{\mathbf{x}}$ is fixed in what follows, unless mentioned otherwise. For brevity, we will use $(w^j(\overline{\mathbf{x}}), \mathbf{y}^j(\overline{\mathbf{x}}))$ and $(w^j, \mathbf{y}^j)$ interchangeably in this section. We observe that $\mathcal{D}(\overline{\mathbf{x}})$ contains the rays $(-1, \mathbf{1})$, $(1, \mathbf{0})$, and $(0, \mathbf{1}^i)$ for every $i\in N$. Moreover, $(w^j, \mathbf{y}^j) \in \mathcal{D}(\overline{\mathbf{x}})$ for all $j\in N$. These observations suggest that all convex combinations of these points and conical combinations of the rays form a subset of $\mathcal{D}(\overline{\mathbf{x}})$. In fact, they are sufficient to generate \emph{all} the elements of $\mathcal{D}(\overline{\mathbf{x}})$. We justify this claim in Lemma \ref{lem:Dx_gen_helper} and Proposition \ref{prop:Dx_gen}. First, in Lemma \ref{lem:Dx_gen_helper}, we prove that a subset of $\mathcal{D}(\overline{\mathbf{x}})$, more precisely, 
\begin{equation*}
\tilde{\mathcal{D}}(\overline{\mathbf{x}}) := \left\{(w, \mathbf{y})\in \mathcal{D}(\overline{\mathbf{x}}) : 
w  = \max_{i\in N} \{h^i(\overline{\mathbf{x}}) - y_i\},  
\{i\in N : y_i = 0\} \neq \emptyset
\right\}
\end{equation*}
can be generated using $\{(w^j, \mathbf{y}^j)\}_{j\in N}$ and the rays $\{(0, \mathbf{1}^i)\}_{i\in N}$.

\begin{lemma}
\label{lem:Dx_gen_helper}
For any $(\overline{w}, \overline{\mathbf{y}})\in \tilde{\mathcal{D}}(\overline{\mathbf{x}})$, there exists $\boldsymbol{\lambda}\in \mathbb{R}^{n}_+$ such that $\boldsymbol{\lambda}^\top\mathbf{1} = 1$, $\overline{w} =  \sum_{j\in N} \lambda_jw^j$, and $\overline{\mathbf{y}} \geq \sum_{j\in N} \lambda_j\mathbf{y}^j$. 
\end{lemma}
\begin{proof}
We sort $\{h^i(\overline{\mathbf{x}})\}_{i\in N}$ in descending order to obtain a permutation $\boldsymbol{\sigma}\in \mathfrak{S}(N)$. In other words, $h^{\sigma(1)}(\overline{\mathbf{x}}) \geq \dots \geq h^{\sigma(n)}(\overline{\mathbf{x}})$. Under the condition that $\overline{w} = \max_{i\in N} \{h^i(\overline{\mathbf{x}}) - \overline{y}_i\}$, $\overline{w} \leq \max_{i\in N} \{h^i(\overline{\mathbf{x}})\} = h^{\sigma(1)}(\overline{\mathbf{x}})$ because $\overline{\mathbf{y}}\geq \mathbf{0}$. In addition, there exists $\hat{\iota}\in N$ with $y_{\hat{\iota}} = 0$, so $\overline{w} \geq h^{\hat{\iota}}(\overline{\mathbf{x}}) - y_{\hat{\iota}} = h^{\hat{\iota}}(\overline{\mathbf{x}})\geq h^{\sigma(n)}(\overline{\mathbf{x}})$. There must exist $\ell \in \{1, \dots, n-1\}$ such that $h^{\sigma(\ell)}(\overline{\mathbf{x}}) \geq \overline{w} \geq h^{\sigma(\ell+1)}(\overline{\mathbf{x}})$. There is an associated pair of convex combination coefficients $(\lambda_{\sigma(\ell)}, \lambda_{\sigma(\ell+1)}) \geq \mathbf{0}$ such that $\lambda_{\sigma(\ell)} + \lambda_{\sigma(\ell+1)} = 1$ and 
\[\overline{w} = \lambda_{\sigma(\ell)}h^{\sigma(\ell)}(\overline{\mathbf{x}}) +\lambda_{\sigma(\ell+1)}h^{\sigma(\ell+1)}(\overline{\mathbf{x}}).\] We complete the vector $\boldsymbol{\lambda}\in \mathbb{R}^{n}_+$ with zeros in all other entries. Observe that for $k \in \{1, \dots, \ell-1\}$, 
\begin{align*}
\lambda_{\sigma(\ell)}y_{\sigma(k)}^{\sigma(\ell)} +\lambda_{\sigma(\ell+1)}y_{\sigma(k)}^{\sigma(\ell+1)} 
& = \lambda_{\sigma(\ell)} \left[h^{\sigma(k)}(\overline{\mathbf{x}}) - h^{\sigma(\ell)}(\overline{\mathbf{x}})\right] +  \lambda_{\sigma(\ell+1)} \left[h^{\sigma(k)}(\overline{\mathbf{x}}) - h^{\sigma(\ell+1)}(\overline{\mathbf{x}}) \right] \\
& = h^{\sigma(k)}(\overline{\mathbf{x}}) - \left[\lambda_{\sigma(\ell)}h^{\sigma(\ell)}(\overline{\mathbf{x}}) +\lambda_{\sigma(\ell+1)}h^{\sigma(\ell+1)}(\overline{\mathbf{x}})\right] \\
& =  h^{\sigma(k)}(\overline{\mathbf{x}})- \overline{w} \\
& = \max \{0, h^{\sigma(k)}(\overline{\mathbf{x}})- \overline{w} \}. 
\end{align*}
When $k = \ell$, 
\begin{align*}
\lambda_{\sigma(\ell)}y_{\sigma(k)}^{\sigma(\ell)} +\lambda_{\sigma(\ell+1)}y_{\sigma(k)}^{\sigma(\ell+1)} & = 0 + \lambda_{\sigma(\ell+1)}\left[h^{\sigma(\ell)}(\overline{\mathbf{x}}) - h^{\sigma(\ell+1)}(\overline{\mathbf{x}}) \right] \\
& = (1 - \lambda_{\sigma(\ell)}) h^{\sigma(\ell)}(\overline{\mathbf{x}}) - \lambda_{\sigma(\ell+1)}h^{\sigma(\ell+1)}(\overline{\mathbf{x}}) \\
 & = h^{\sigma(\ell)}(\overline{\mathbf{x}}) - \left[\lambda_{\sigma(\ell)}h^{\sigma(\ell)}(\overline{\mathbf{x}}) +\lambda_{\sigma(\ell+1)}h^{\sigma(\ell+1)}(\overline{\mathbf{x}})\right] \\
 & = h^{\sigma(\ell)}(\overline{\mathbf{x}}) - \overline{w}\\
 & = \max \{0, h^{\sigma(\ell)}(\overline{\mathbf{x}})- \overline{w} \}. 
\end{align*}
For $k \in \{\ell+1, \dots, n\}$, 
\[\lambda_{\sigma(\ell)}y_{\sigma(k)}^{\sigma(\ell)} +\lambda_{\sigma(\ell+1)}y_{\sigma(k)}^{\sigma(\ell+1)} = 0 = \max \{0, h^{\sigma(k)}(\overline{\mathbf{x}})- \overline{w} \}.\] Therefore, we have constructed $\boldsymbol{\lambda}\in \mathbb{R}^{n}_+$ with $\boldsymbol{\lambda}^\top\mathbf{1} = 1$, such that \[\overline{w} = \sum_{j\in N}\lambda_jw^j,\]
and for every $i\in N$, 
\[\sum_{j\in N}\lambda_jy_i^j = \max \{0, h^i(\overline{\mathbf{x}}) - \overline{w} \} \leq \overline{y}_i.\]
The last inequality follows from the definition of $\mathcal{D}(\overline{\mathbf{x}})$. 
\end{proof}

In the next proposition, we formalize an alternative representation of $\mathcal{D}(\overline{\mathbf{x}})$. 

\begin{proposition}
\label{prop:Dx_gen}
For any $\overline{\mathbf{x}}\in\mathcal{X}$, 
\[\mathcal{D}(\overline{\mathbf{x}}) = \conv{\left\{\left(w^j, \mathbf{y}^j\right)\right\}_{j\in N} } + \cone{\left\{(-1, \mathbf{1}), (1, \mathbf{0})\right\} \cup \left\{(0, \mathbf{1}^i)\right\}_{i\in N}},\]
where $(w^j, \mathbf{y}^j)$ for all $j\in N$ are given by \eqref{eq:wyj}.
\end{proposition}
\begin{proof}
Following from Lemma \ref{lem:Dx_gen_helper}, it suffices to show that any element of $\mathcal{D}(\overline{\mathbf{x}})$ is a conical combination of some element of $\tilde{\mathcal{D}}(\overline{\mathbf{x}})$ with the rays $(-1, \mathbf{1})$ and $(1, \mathbf{0})$. That is, for every $(\overline{w}, \overline{\mathbf{y}})\in \mathcal{D}(\overline{\mathbf{x}})$, we find $(\tilde{w}, \tilde{\mathbf{y}})\in \tilde{\mathcal{D}}(\overline{\mathbf{x}})$ and $r \geq 0$, such that $\overline{w} \geq \tilde{w} - r$, and $\overline{\mathbf{y}} = \tilde{\mathbf{y}} + r\mathbf{1}$. Now, we construct $r = \min_{i\in N} \{\overline{y}_i\}$, $\tilde{\mathbf{y}} = \overline{\mathbf{y}} - r\mathbf{1}$, and $\tilde{w} = \max_{i\in N} \{h^i(\overline{\mathbf{x}}) - \tilde{y}_i\}$. Note that $r \geq 0$, because $\overline{\mathbf{y}}\geq \mathbf{0}$. By construction of $r$, we have  $\tilde{\mathbf{y}} \geq 0$, and $\tilde{y}_{\hat{\iota}} = 0$ where $\hat{\iota} \in \arg\min_{i\in N} \{\overline{y}_i\}$. By definition of $\tilde{w}$, we know that $(\tilde{w}, \tilde{\mathbf{y}})\in \tilde{\mathcal{D}}(\overline{\mathbf{x}})$. Lastly, 
\[
 \overline{w} \geq \max_{i\in N} \{h^i(\overline{\mathbf{x}}) - \overline{y}_i\} 
= \max_{i\in N} \{h^i(\overline{\mathbf{x}}) - \overline{y}_i + r\} - r 
= \max_{i\in N} \{h^i(\overline{\mathbf{x}}) - \tilde{y}_i\} - r 
= \tilde{w} - r. \]
 This completes the proof. 
\end{proof}

Recall \[\mathcal{P}^H = \left\{(w, \mathbf{y}, \mathbf{x})\in \mathbb{R}^{n+1}\times\mathcal{X} : (w, \mathbf{y}) \in \mathcal{D}(\mathbf{x})\right\}.\] 
It follows from Proposition \ref{prop:Dx_gen} that $\conv{\mathcal{P}^H}$ is generated by the points $(w^j(\mathbf{x}), \mathbf{y}^j(\mathbf{x}), \mathbf{x})$ for all $\mathbf{x}\in\mathcal{X}$ and all $j\in N$, along with the rays $(-1, \mathbf{1}, \mathbf{0})$, $(1, \mathbf{0}, \mathbf{0})$ and $\{(0, \mathbf{1}^i, \mathbf{0})\}_{i\in N}$. We formalize this result in the next corollary. 
\begin{corollary}
\label{coro:cv_alt}
For $j\in N$, let 
\begin{equation}
\label{eq:Dj}
\mathcal{D}^j := \bigcup_{\mathbf{x}\in\mathcal{X}} \left\{ (w^j(\mathbf{x}), \mathbf{y}^j(\mathbf{x}), \mathbf{x}) +\cone{\left\{(-1, \mathbf{1}, \mathbf{0}), (1, \mathbf{0}, \mathbf{0})\right\} \cup \left\{(0, \mathbf{1}^i, \mathbf{0})\right\}_{i\in N}}   \right\}.
\end{equation}
Then $\conv{\mathcal{P}^H} = \conv{\bigcup_{j\in N} \mathcal{D}^j }$. 
\end{corollary}
We remark that the generating points are grouped and combined with the cone in a special way, which enables us to exploit certain properties for convexification. \\

We now explore the structure of $\mathcal{D}^j$ for any $j\in N$, by first understanding the set 
\[\mathcal{D}^j(\overline{\mathbf{x}}) := (w^j(\overline{\mathbf{x}}), \mathbf{y}^j(\overline{\mathbf{x}})) +\cone{\left\{(-1, \mathbf{1}), (1, \mathbf{0})\right\} \cup \left\{(0, \mathbf{1}^i)\right\}_{i\in N}} \]
for every fixed $\overline{\mathbf{x}}\in\mathcal{X}$. In what follows, we will use $h^j(\overline{\mathbf{x}})$ and $\max\{0, h^i(\overline{\mathbf{x}}) - h^j(\overline{\mathbf{x}})\}$ to replace $w^j(\overline{\mathbf{x}})$ and $y^j_i(\overline{\mathbf{x}})$, $i\in N$, respectively, to avoid confusion between the variables $(w, \mathbf{y})$ and these parameters. We can equivalently state $\mathcal{D}^j(\overline{\mathbf{x}})$ as 
\begin{equation}
\label{eq:Djx}
\mathcal{D}^j(\overline{\mathbf{x}}) = \left\{ (w, \mathbf{y}) \in \mathbb{R}^{n+1} : \exists \, r \geq 0, \text{ s.t. } 
\begin{aligned}
w & \geq h^j(\overline{\mathbf{x}}) - r, \\
y_i & \geq \max\{0, h^i(\overline{\mathbf{x}}) - h^j(\overline{\mathbf{x}})\} + r,  \, \forall\, i\in N 
\end{aligned}
\right\},
\end{equation}
where $r$ represents the coefficient of the ray $(-1, \mathbf{1})$, while the inequalities capture the rest of the rays. In the next lemma, we eliminate the auxiliary variable $r$. 

\begin{lemma}
\label{lem:Djx_alt}
For any $j\in N$ and any $\overline{\mathbf{x}}\in\mathcal{X}$, 
\begin{equation}
\label{eq:Djx_alt}
\mathcal{D}^j(\overline{\mathbf{x}}) = \left\{(w,\mathbf{y})\in\mathbb{R}^{n+1} : 
\begin{aligned}
& w + y_i \geq \max \{h^j(\overline{\mathbf{x}}), h^i(\overline{\mathbf{x}})\},  \\
& y_i \geq \max \{0, h^i(\overline{\mathbf{x}}) - h^j(\overline{\mathbf{x}})\}, 
\end{aligned}
\; \; \forall\, i\in N
 \right\}.
 \end{equation}
\end{lemma}
\begin{proof}
{The inequality conditions in \eqref{eq:Djx}, including $r\geq 0$, can be written as
\[\min_{i\in N} \left( y_i - \max\{0, h^i(\overline{\mathbf{x}}) - h^j(\overline{\mathbf{x}})\} \right) \geq r \geq \max\{0, h^j(\overline{\mathbf{x}}) - w\}.\]
Such $r$ exists if and only if
\[\min_{i\in N} \left( y_i - \max\{0, h^i(\overline{\mathbf{x}}) - h^j(\overline{\mathbf{x}})\} \right) \geq \max\{0, h^j(\overline{\mathbf{x}}) - w\},\]
which is equivalent to
\[y_i - \max\{0, h^i(\overline{\mathbf{x}}) - h^j(\overline{\mathbf{x}})\} \geq h^j(\overline{\mathbf{x}}) - w, \quad y_i - \max\{0, h^i(\overline{\mathbf{x}}) - h^j(\overline{\mathbf{x}})\} \geq 0, \quad \forall\, i\in N,\]
that is,
\[w + y_i \geq \max \{h^j(\overline{\mathbf{x}}), h^i(\overline{\mathbf{x}})\}, \quad y_i \geq \max \{0, h^i(\overline{\mathbf{x}}) - h^j(\overline{\mathbf{x}})\}, \quad \forall\, i\in N.\]}
\end{proof}

Following from Lemma \ref{lem:Djx_alt}, we observe that the set $\mathcal{D}^j$ given by \eqref{eq:Dj} is, in fact, equivalent to 
\begin{equation}
\mathcal{D}^j = \left\{(w,\mathbf{y}, \mathbf{x})\in\mathbb{R}^{n+1} \times\mathcal{X} : 
\begin{aligned}
 w +  y_i &\geq \max \{h^j(\mathbf{x}), h^i(\mathbf{x})\},  \\
 y_i &\geq \max \{0, h^i(\mathbf{x}) - h^j(\mathbf{x})\}, 
\end{aligned}
\; \; \forall\, i\in N.
 \right\}
\end{equation}

Based on Corollary \ref{coro:cv_alt}, our goal of obtaining $\conv{\mathcal{P}^H}$ can be achieved by convexifying $\bigcup_{j\in N} \mathcal{D}^j$. However, this task is challenging in general because $\mathcal{D}^j$ is mixed-integer for every $j\in N$ and the terms involving $\mathbf{x}$ are nonlinear. Therefore, in the next section, we explore $\conv{\mathcal{P}^H}$ when certain relevant functions are L$^\natural$-convex. In particular, we will first convexify $\mathcal{D}^j$ for every $j\in N$ and then discuss the convex hull of their union.

\subsection{Convexify $\mathcal{P}^H$ with L$^\natural$-convexity}
For ease of notation, we define $g^j_i:\mathcal{X}\rightarrow\mathbb{R}$ by 
\[g^j_i (\mathbf{x}) := \max\{0, h^i(\mathbf{x})-h^j(\mathbf{x})\}, \]
for $i,j\in N$. 
With this new notation, we have 
\begin{equation}
\label{eq:Dj_alt}
\mathcal{D}^j = \left\{(w,\mathbf{y}, \mathbf{x})\in\mathbb{R}^{n+1} \times\mathcal{X} : 
\begin{aligned}
 w +  y_i &\geq g^j_i(\mathbf{x}) + h^j(\mathbf{x}),  \\
 y_i &\geq g^j_i(\mathbf{x}), 
\end{aligned}
\; \; \forall\, i\in N,
 \right\}
\end{equation}
for every $j\in N$. We impose Assumption \hyperlink{A1}{1} in this section. \\

\textbf{Assumption \hypertarget{A1}{1}.} The functions $g^j_i$ and $h^j$ are L$^\natural$-convex for $i, j \in N$. \\

{The appeal of L$^\natural$-convexity in our framework is twofold: the convex hull of each individual epigraph is fully understood, and multiple such epigraphs admit tractable simultaneous convexification. We note that the framework is not inherently limited to L$^\natural$-convex functions and extends to other function classes possessing these two properties. The inclusion-certificate approach of \citep{tawarmalani2010inclusion} provides a natural avenue for identifying sufficient conditions under which such extensions are possible.}

\begin{remark}
\label{remark:A1_eg}
We include examples where Assumption \hyperlink{A1}{1} is satisfied. When $h^j$ is a monotone univariate function for every $j\in N$, $h^j$ and $-h^j$ are trivially lattice submodular. By Lemmas \ref{lemma:monotone_composite_sub}--\ref{lemma:continuous_max}, $g^j_i(\mathbf{x}) = \max(h^j(\mathbf{x}), h^i(\mathbf{x})) - h^j(\mathbf{x})$, for $i, j\in N$, is also lattice submodular. If $\mathcal{X} = \mathcal{B}_\mathbf{0}$, $g^j_i$ and $h^j$ are submodular, and thus L$^\natural$-convex, for all $i, j \in N$. For any $\mathcal{X}\subseteq \mathbb{Z}^n$, if, additionally, $h^j$ for all $j\in N$ are univariate linear functions with a common slope, then $g^j_i$ and $h^j$ are L-convex, and thus L$^\natural$-convex, for all $i, j \in N$.
\end{remark}

We characterize $\conv{\mathcal{D}^j}$ in the next proposition. 

\begin{proposition}
\label{prop:conv_dj}
Let any $j\in N$ be fixed. Under Assumption \hyperlink{A1}{1}, $\conv{\mathcal{D}^j} = \mathcal{CD}^j$, where 
\begin{equation}
\label{eq:CDj_def}
\mathcal{CD}^j := \left\{(w,\mathbf{y}, \mathbf{x})\in\mathbb{R}^{n+1}\times\overline{\mathcal{X}} : \;
\begin{aligned}
(w+y_i, \mathbf{x}) & \in \conv{\mathcal{P}^{g^j_i+h^j}_\mathcal{X}}, \\
(y_i, \mathbf{x}) & \in \conv{\mathcal{P}^{g^j_i}_\mathcal{X}},
\end{aligned}
\, \forall\, i\in N.
\right\}
\end{equation}
\end{proposition}
Recall that $\overline{\mathcal{X}}$ represents the continuous relaxation of the hyperrectangle $\mathcal{X}$, and $\mathcal{P}^\cdot_\mathcal{X}$ represents the epigraph of the given function over $\mathcal{X}$. The same result holds as long as $g^j_i+h^j$ and $g^j_i$ are L$^\natural$-convex for all $i\in N$. 
\begin{proof}
We first transform $\mathcal{D}^j$ by introducing auxiliary variables $\boldsymbol{\eta} \in \mathbb{R}^n$ such that $\eta_i = w + y_i$, for all $i\in N$, and we obtain
\begin{equation}
\label{eq:Dj_aux}
\left\{(\boldsymbol{\eta}, \mathbf{y}, \mathbf{x})\in\mathbb{R}^{2n} \times\mathcal{X} : 
\begin{aligned}
& \eta_i \geq g^j_i(\mathbf{x})+h^j(\mathbf{x}), \; \; \forall\, i\in N, \\
 & y_i \geq g^j_i(\mathbf{x}), \; \; \forall\, i\in N,\\
& \eta_i - y_i = \eta_{i'} - y_{i'}, \;\; \forall\, i, i'\in N.
\end{aligned}
 \right\}
\end{equation}
Notice that this new representation is in the form of \eqref{eq:CQ'}. Additionally, with L$^\natural$-convexity of $\{g^j_i+h^j, g^j_i\}_{i\in N}$, the convex hull of \eqref{eq:Dj_aux} is 
\begin{equation}
\left\{(\boldsymbol{\eta}, \mathbf{y}, \mathbf{x})\in\mathbb{R}^{2n} \times\overline{\mathcal{X}} : 
\begin{aligned}
& (\eta_i, \mathbf{x}) \in \conv{\mathcal{P}^{g^j_i+h^j}_\mathcal{X}}, \; \; \forall\, i\in N, \\
 & (y_i, \mathbf{x}) \in \conv{\mathcal{P}^{g^j_i}_\mathcal{X}} \; \; \forall\, i\in N,\\
& \eta_i - y_i = \eta_{i'} - y_{i'}, \;\; \forall\, i, i'\in N,
\end{aligned}
 \right\}
\end{equation}
following from Proposition \ref{prop:inter2}. Replacing $\eta_i$ by $w + y_i$ for all $i\in N$ completes the proof.
\end{proof}

From the previous sections, we know that $\mathcal{CD}^j$ for every $j\in N$ is described by the trivial inequalities and the SEPIs associated with all points $\mathbf{p}\in\underline{\mathcal{X}}$ and permutations $\boldsymbol{\delta}\in\mathfrak{S}(N)$ with respect to $g^j_i+h^j$ and $g^j_i$ for all $i\in N$. Let $(s_0^{j, i, \mathbf{p}, \boldsymbol{\delta}}, \mathbf{s}^{j, i, \mathbf{p}, \boldsymbol{\delta}}) \in \mathbb{R}^{n+1}$ denote the parameters of the right-hand side of the SEPI associated with $\mathbf{p}$ and $\boldsymbol{\delta}$ for $g^j_i$, where $i,j\in N$. Similarly, let $(t_0^{j, \mathbf{p}, \boldsymbol{\delta}}, \mathbf{t}^{j, \mathbf{p}, \boldsymbol{\delta}}) \in \mathbb{R}^{n+1}$ denote the parameters of the right-hand side of the SEPI associated with $\mathbf{p}$ and $\boldsymbol{\delta}$ for $h^j$, $j\in N$. We remark that every SEPI, associated with $\mathbf{p}$ and $\boldsymbol{\delta}$, for the epigraph $\{(\eta, \mathbf{x}) \in \mathbb{R}\times\mathcal{X} : \eta \geq g^j_i(\mathbf{x})+h^j(\mathbf{x})\}$, for any $i,j\in N$, is exactly $\eta \geq s_0^{j, i, \mathbf{p}, \boldsymbol{\delta}} + t_0^{j, \mathbf{p}, \boldsymbol{\delta}} + \left( \mathbf{s}^{j, i, \mathbf{p}, \boldsymbol{\delta}} + \mathbf{t}^{j, \mathbf{p}, \boldsymbol{\delta}}\right) ^\top\mathbf{x}$. 

Then 
\begin{equation}
\label{eq:CDj_form}
\mathcal{CD}^j = \left\{(w,\mathbf{y}, \mathbf{x})\in\mathbb{R}^{n+1}\times\overline{\mathcal{X}} : \;
\begin{aligned}
w+y_i & \geq \left( \mathbf{s}^{j, i, \mathbf{p}, \boldsymbol{\delta}} + \mathbf{t}^{j, \mathbf{p}, \boldsymbol{\delta}}\right) ^\top\mathbf{x} +  s_0^{j, i, \mathbf{p}, \boldsymbol{\delta}} + t_0^{j, \mathbf{p}, \boldsymbol{\delta}}, \\ 
y_i & \geq \mathbf{s}^{{j, i, \mathbf{p}, \boldsymbol{\delta}}^\top}\mathbf{x} + s_0^{j, i, \mathbf{p}, \boldsymbol{\delta}},  \\
&  \forall\, i\in N, \, \mathbf{p}\in\underline{\mathcal{X}}, \,\boldsymbol{\delta}\in\mathfrak{S}(N).
\end{aligned}
\right\}
\end{equation}
We note that all inequalities that define $\mathcal{CD}^j$ are linear. When there are finitely many distinct SEPIs, the convex hull, $\mathcal{CD}^j$, is polyhedral.  \\

In what follows, we will convexify the union $\bigcup_{j\in N} \mathcal{D}^j$ by characterizing the intersection of the polars of $\mathcal{CD}^j$ for $j\in N$. {We caution that our use of the term \emph{polar}, as defined in polyhedral combinatorics \cite{wolsey1999integer}, differs from the standard polar set in convex analysis: here, the polar of $\mathcal{CD}^j$ refers to the set of coefficient vectors $(u_0, \mathbf{u}, \boldsymbol{\pi}, \pi_0)$ of all linear inequalities of the form $u_0 w + \mathbf{u}^\top\mathbf{y} \geq \boldsymbol{\pi}^\top\mathbf{x} + \pi_0$ that are valid for $\mathcal{CD}^j$.} For any $j\in N$, we denote the polar of $\mathcal{CD}^j$ by
\begin{equation}
\Pi^{\mathcal{CD}^j} := \left\{ (u_0, \mathbf{u}, \boldsymbol{\pi}, \pi_0) \in \mathbb{R}^{2n+2} : u_0 w + \mathbf{u}^\top\mathbf{y} \geq \boldsymbol{\pi}^\top\mathbf{x} + \pi_0, \, \forall\, (w,\mathbf{y}, \mathbf{x})\in \mathcal{CD}^j \right\}.
\end{equation}
Following from \eqref{eq:CDj_form}, for every $j\in N$, 
\begin{equation}
\Pi^{\mathcal{CD}^j} = \left\{ (u_0, \mathbf{u}, \boldsymbol{\pi}, \pi_0) \in \mathbb{R}^{2n+2} : \, \exists \, \mathbf{a}^j, \mathbf{b}^j \geq \mathbf{0} \text{ s.t. } \eqref{eq:pi_dj_constr1}\text{--}\eqref{eq:pi_dj_constr4}\right\},
\end{equation}
where 
\begin{subequations}
\begin{align}
u_0 & = \sum_{i\in N} \sum_{\mathbf{p}\in\underline{\mathcal{X}}}\sum_{\boldsymbol{\delta}\in\mathfrak{S}(N)} a^{j,i,\mathbf{p},\boldsymbol{\delta}},  \label{eq:pi_dj_constr1}\\
u_i & = \sum_{\mathbf{p}\in\underline{\mathcal{X}}}\sum_{\boldsymbol{\delta}\in\mathfrak{S}(N)} \left(a^{j,i,\mathbf{p},\boldsymbol{\delta}} + b^{j,i,\mathbf{p},\boldsymbol{\delta}}\right), \quad \forall\, i\in N, \\
\pi_k & = \sum_{i\in N} \sum_{\mathbf{p}\in\underline{\mathcal{X}}}\sum_{\boldsymbol{\delta}\in\mathfrak{S}(N)} s^{j,i,\mathbf{p},\boldsymbol{\delta}}_k \left(a^{j,i,\mathbf{p},\boldsymbol{\delta}} + b^{j,i,\mathbf{p},\boldsymbol{\delta}}\right) + \sum_{\mathbf{p}\in\underline{\mathcal{X}}}\sum_{\boldsymbol{\delta}\in\mathfrak{S}(N)}  t^{j,\mathbf{p},\boldsymbol{\delta}}_k \sum_{i\in N}  a^{j,i,\mathbf{p},\boldsymbol{\delta}}, \quad \forall\, k\in N,   \\ 
\pi_0 & \leq \sum_{i\in N} \sum_{\mathbf{p}\in\underline{\mathcal{X}}}\sum_{\boldsymbol{\delta}\in\mathfrak{S}(N)} s^{j,i,\mathbf{p},\boldsymbol{\delta}}_0 \left(a^{j,i,\mathbf{p},\boldsymbol{\delta}} + b^{j,i,\mathbf{p},\boldsymbol{\delta}}\right) + \sum_{\mathbf{p}\in\underline{\mathcal{X}}}\sum_{\boldsymbol{\delta}\in\mathfrak{S}(N)}  t^{j,\mathbf{p},\boldsymbol{\delta}}_0 \sum_{i\in N}  a^{j,i,\mathbf{p},\boldsymbol{\delta}}.\label{eq:pi_dj_constr4}
\end{align}
\end{subequations}

That is, $\Pi^{\mathcal{CD}^j}$ contains all the conical combinations of the SEPI parameters given in \eqref{eq:CDj_form} and the ray $(0, \mathbf{0}, \mathbf{0}, -1)$. We next discuss a few properties of $\Pi^{\mathcal{CD}^j}$ for every $j\in N$. 

\begin{observation}
\label{obs:u0_u_prop}
We notice that $(u_0, \mathbf{u})\geq \mathbf{0}$ by non-negativity of $\mathbf{a}^j$ and $\mathbf{b}^j$. Furthermore, 
\[\sum_{i\in N} u_i = \sum_{i\in N} \sum_{\mathbf{p}\in\underline{\mathcal{X}}}\sum_{\boldsymbol{\delta}\in\mathfrak{S}(N)} \left(a^{j,i,\mathbf{p},\boldsymbol{\delta}} + b^{j,i,\mathbf{p},\boldsymbol{\delta}}\right) = u_0 + \sum_{i\in N} \sum_{\mathbf{p}\in\underline{\mathcal{X}}}\sum_{\boldsymbol{\delta}\in\mathfrak{S}(N)} b^{j,i,\mathbf{p},\boldsymbol{\delta}}, \; \forall\, j\in N.\] 
Thus, $u_0 \leq \sum_{i\in N} u_i$. 
\end{observation}

For ease of notation, we define 
\begin{equation}
\label{eq:U}
 \mathcal{U} := \left\{ (u_0, \mathbf{u})\in\mathbb{R}_+^{n+1} : u_0 \leq \sum_{i\in N} u_i\right\}. 
 \end{equation}
In addition, for any $j\in N$ and any $(u_0, \mathbf{u}) \in \mathcal{U}$, we define $f^j_{u_0, \mathbf{u}}: \mathcal{X}\rightarrow\mathbb{R}$ by 
\[f^j_{u_0, \mathbf{u}}(\mathbf{x}) := u_0h^j(\mathbf{x}) + \sum_{i\in N}u_i g^j_i(\mathbf{x}).  \]
Given that $f^j_{u_0, \mathbf{u}}$ is a nonnegative combination of L$^\natural$-convex functions, $f^j_{u_0, \mathbf{u}}$ itself is L$^\natural$-convex. Thus, the convex hull of its epigraph, $\mathcal{P}^{f^j_{u_0, \mathbf{u}}}$, is described fully by SEPIs. We let 
\[\Pi^{\conv{\mathcal{P}^{f^j_{u_0, \mathbf{u}}}}} := \left\{(\boldsymbol{\xi}, \xi_0) \in \mathbb{R}^{n+1} : \eta \geq \boldsymbol{\xi}^\top\mathbf{x} + \xi_0, \, \forall\, (\eta,\mathbf{x})\in\conv{\mathcal{P}^{f^j_{u_0, \mathbf{u}}}}\right\} \]
represent the polar of $\conv{\mathcal{P}^{f^j_{u_0, \mathbf{u}}}}$ for all $(u_0, \mathbf{u}) \in \mathcal{U}$. 

\begin{lemma}
\label{lem:fju_prop}
Given any $j\in N$ and any $(u_0, \mathbf{u}) \in \mathcal{U}$, 
\[\left(u_0 w + \sum_{i\in N} u_i y_i, \mathbf{x}\right) \in \conv{\mathcal{P}^{f^j_{u_0, \mathbf{u}}}},\]
for all $(w, \mathbf{y}, \mathbf{x})\in\mathcal{CD}^j$. 
\end{lemma}
{To improve the flow of the presentation, the proof of Lemma \ref{lem:fju_prop} is relegated to the appendix.} This lemma implies that, for any $j\in N$ and $(u_0, \mathbf{u}) \in \mathcal{U}$, the inequalities of the form 
\[u_0 w + \mathbf{u}^\top\mathbf{y} \geq \boldsymbol{\xi}^\top\mathbf{x} + \xi_0\]
for every $(\boldsymbol{\xi}, \xi_0) \in \Pi^{\conv{\mathcal{P}^{f^j_{u_0, \mathbf{u}}}}}$, are valid for $\mathcal{CD}^j$.

\begin{lemma}
For any $j\in N$, 
\[\Pi^{\mathcal{CD}^j} = \left\{(u_0, \mathbf{u}, \boldsymbol{\pi}, \pi_0) \in \mathbb{R}^{2n+2} : (u_0, \mathbf{u}) \in \mathcal{U}, (\pi_0, \boldsymbol{\pi})\in \Pi^{\conv{\mathcal{P}^{f^j_{u_0, \mathbf{u}}}}}\right\}.\] 
\end{lemma}
\begin{proof}
First, consider any $(u_0, \mathbf{u}, \boldsymbol{\pi}, \pi_0) \in \Pi^{\mathcal{CD}^j}$. By Observation \ref{obs:u0_u_prop}, $(u_0, \mathbf{u})\in\mathcal{U}$. Note that $(h^j(\mathbf{x}),$  $[g^j_i(\mathbf{x})]_{i\in N}, \mathbf{x})\in\mathcal{CD}^j$ for every $\mathbf{x}\in\mathcal{X}$. Thus, 
\begin{equation*}
\boldsymbol{\pi}^\top\mathbf{x}+ \pi_0  \leq u_0h^j(\mathbf{x}) + \sum_{i\in N} u_ig^j_i(\mathbf{x})
 = f^j_{u_0, \mathbf{u}}(\mathbf{x}) \leq \eta, 
\end{equation*}
for all $(\eta, \mathbf{x})\in\mathcal{P}^{f^j_{u_0, \mathbf{u}}}$. This means that $(\boldsymbol{\pi}, \pi_0)\in \Pi^{\conv{\mathcal{P}^{f^j_{u_0, \mathbf{u}}}}}$. Next, consider any $(u_0, \mathbf{u})\in\mathcal{U}$ and any $(\boldsymbol{\xi}, \xi_0) \in \Pi^{\conv{\mathcal{P}^{f^j_{u_0, \mathbf{u}}}}}$. The reverse containment follows from Lemma \ref{lem:fju_prop}. 
 \end{proof}

Now, we explore the intersection of the polars $\bigcap_{j\in N} \Pi^{\mathcal{CD}^j}$ to obtain all the inequalities that define $\bigcup_{i\in N} \mathcal{CD}^j$. For any $(u_0, \mathbf{u}) \in \mathcal{U}$, we define a function $F_{u_0, \mathbf{u}}:\mathcal{X}\rightarrow\mathbb{R}$ by 
\begin{equation}
\label{eq:F}
F_{u_0, \mathbf{u}}(\mathbf{x}) := \min_{j\in N} \left\{f^j_{u_0, \mathbf{u}} (\mathbf{x})\right\}. 
\end{equation}
We represent the polar of the convex hull of its epigraph by 
\[\Pi^{\conv{\mathcal{P}^{F_{u_0, \mathbf{u}}}}} := \left\{(\boldsymbol{\xi}, \xi_0) \in \mathbb{R}^{n+1} : \eta \geq \boldsymbol{\xi}^\top\mathbf{x} + \xi_0, \, \forall\, (\eta,\mathbf{x})\in\conv{\mathcal{P}^{F_{u_0, \mathbf{u}}}}\right\}.  \]

\begin{lemma}
\label{lem:polar_inter}
The intersection of the polars of $\{\mathcal{CD}^j\}_{j\in N}$ is
\[\bigcap_{j\in N} \Pi^{\mathcal{CD}^j} = \left\{(u_0, \mathbf{u}, \boldsymbol{\pi}, \pi_0) \in \mathbb{R}^{2n+2} : (u_0, \mathbf{u}) \in \mathcal{U}, (\boldsymbol{\pi}, \pi_0)\in \Pi^{\conv{\mathcal{P}^{F_{u_0, \mathbf{u}}}}}\right\}.\] 
\end{lemma}
\begin{proof} We observe that 
\begin{align*}
\bigcap_{j\in N} \Pi^{\mathcal{CD}^j} & = \bigcap_{j\in N} \left[\bigcup_{(u_0, \mathbf{u}) \in \mathcal{U}}\{(u_0, \mathbf{u})\} \times \Pi^{\conv{\mathcal{P}^{f^j_{u_0, \mathbf{u}}}}} \right] \\
& = \bigcup_{(u_0, \mathbf{u}) \in \mathcal{U}}\{(u_0, \mathbf{u})\} \times \left[ \bigcap_{j\in N} \Pi^{\conv{\mathcal{P}^{f^j_{u_0, \mathbf{u}}}}} \right] \\
& = \bigcup_{(u_0, \mathbf{u}) \in \mathcal{U}}\{(u_0, \mathbf{u})\} \times  \Pi^{\bigcup_{j\in N} \conv{\mathcal{P}^{f^j_{u_0, \mathbf{u}}}}}  \\
& = \bigcup_{(u_0, \mathbf{u}) \in \mathcal{U}}\{(u_0, \mathbf{u})\} \times \Pi^{\conv{\mathcal{P}^{F_{u_0, \mathbf{u}}}}}.
\end{align*}
\end{proof}

\begin{theorem}
\label{thm:conv_H}
Under Assumption \hyperlink{A1}{1}, the convex hull of the epigraph of $H${, where $H(\mathbf{x},\mathbf{y}) = \max_{i\in N}\{h^i(\mathbf{x}) - y_i\}$ is given by \eqref{eq:H},} is
\begin{equation*}
\conv{\mathcal{P}^H} = \left\{(w, \mathbf{y}, \mathbf{x})\in \mathbb{R}^{n+1}\times\overline{\mathcal{X}} :  (u_0 w + \mathbf{u}^\top\mathbf{y}, \mathbf{x}) \in \conv{\mathcal{P}^{F_{u_0, \mathbf{u}}}}, \forall\, (u_0, \mathbf{u})\in\mathcal{U} \right\},
\end{equation*}
{where $F_{u_0,\mathbf{u}}$ is defined in \eqref{eq:F} and $\mathcal{U}$ is given by \eqref{eq:U}.}
\end{theorem}
\begin{proof}
We have noted {in Corollary \ref{coro:cv_alt} that $\conv{\mathcal{P}^H} = \conv{\bigcup_{j\in N} \mathcal{D}^j}$, and thus, by Proposition \ref{prop:conv_dj},} that $\conv{\mathcal{P}^H} = \conv{\bigcup_{j\in N} \mathcal{CD}^j}$. The latter is described by the inequalities $u_0 w + \mathbf{u}^\top\mathbf{y} \geq \boldsymbol{\pi}^\top\mathbf{x} + \pi_0$ for all $(u_0, \mathbf{u},  \boldsymbol{\pi}, \pi_0)\in \bigcap_{j\in N} \Pi^{\mathcal{CD}^j}$, which are given by Lemma \ref{lem:polar_inter}. This completes the proof. 
\end{proof}

\begin{corollary}
Under Assumption \hyperlink{A1}{1}, if $F_{u_0, \mathbf{u}}$ are L$^\natural$-convex for all $(u_0, \mathbf{u})\in\mathcal{U}$, then $\conv{\mathcal{P}^H}$ is completely described by the \emph{mixed-integer SEPIs (MISEPIs)}
\begin{equation}
\label{eq:MISEPI}
u_0 w + \mathbf{u}^\top\mathbf{y} \geq \mathbf{s}^{\mathbf{p},\boldsymbol{\delta}, u_0, \mathbf{u}^\top}\mathbf{x} + s_0^{\mathbf{p},\boldsymbol{\delta}, u_0, \mathbf{u}}, \; \forall\, (u_0,\mathbf{u})\in\mathcal{U}. 
\end{equation}
Here, $\eta \geq \mathbf{s}^{\mathbf{p},\boldsymbol{\delta}, u_0, \mathbf{u}^\top}\mathbf{x} + s_0^{\mathbf{p},\boldsymbol{\delta}, u_0, \mathbf{u}}$ is the SEPI associated with $\mathbf{p}\in\underline{\mathcal{X}}$ and $\boldsymbol{\delta}\in\mathfrak{S}(N)$ for the epigraph of $F_{u_0, \mathbf{u}}$.
\end{corollary}

\begin{remark}
\label{remark:LC_Fu0u}
Based on Proposition \ref{prop:F_lattice_sub}, if  $h^i:\mathcal{X}_i\rightarrow\mathbb{R}$ are monotonically decreasing functions for all $i\in N$ (or increasing, as long as consistent for all $i\in N$), then $F_{u_0, \mathbf{u}}$ is lattice submodular for all $(u_0, \mathbf{u})\in\mathcal{U}$. 
\begin{itemize}
\item If $\mathcal{X} = \mathcal{B}_\mathbf{0}$, then $F_{u_0, \mathbf{u}}$ for all $(u_0, \mathbf{u})\in\mathcal{U}$ are {submodular set function}s, which are also L$^\natural$-convex. 
\item Let any $\mathbf{q}\in\mathbb{R}^n$ and $\alpha\in\mathbb{R}_+$ be given. 
If $h^i(x_i) = q_i - \alpha x_i$ for every $i\in N$, then $F_{u_0, \mathbf{u}}$ is L-convex for all $(u_0, \mathbf{u})\in\mathcal{U}$. To see this, 
\begin{align*}
F_{u_0, \mathbf{u}}(\mathbf{x}+ \mathbf{1}) & = \min_{j\in N} \left\{u_0h^j(x_j+1) + \sum_{i\in N}u_i \max \{0, h^i(x_i+1) - h^j(x_j+1)\} \right\} \\
& = \min_{j\in N} \left\{u_0(q_j - \alpha x_j - \alpha) + \sum_{i\in N}u_i \max \{0, q_i - \alpha x_i -\alpha - q_j + \alpha x_j + \alpha\} \right\} \\
& = \min_{j\in N} \left\{u_0(q_j - \alpha x_j) + \sum_{i\in N}u_i \max \{0, q_i - \alpha x_i - q_j + \alpha x_j\} \right\} - u_0\alpha\\
& = F_{u_0, \mathbf{u}}(\mathbf{x}) - u_0\alpha. 
\end{align*}
\end{itemize}
It follows that, in these examples, $\conv{\mathcal{P}^H}$ given in Theorem \ref{thm:conv_H} can be stated as a linear program with MISEPIs \eqref{eq:MISEPI}.
\end{remark}

\subsection{Properties of $F_{u_0, \mathbf{u}}$}
\label{sect:F_properties}
In this section, we explore the properties of the functions $F_{u_0, \mathbf{u}}$ for $(u_0, \mathbf{u})\in\mathcal{U}$. {These properties are the basis of the separation algorithm for the MISEPIs in Section \ref{sect:sepa_MISEPIs} and of the facet analysis in Section \ref{sect:FD_MISEPI}.} Let any $\overline{\mathbf{x}}\in\mathcal{X}$ be fixed. We use $h_i$ to represent $h^i(\overline{\mathbf{x}})$ for brevity, and we use $(\cdot)$ to denote the relabeled indices following any (in case of ties) descending order of $\{h_i\}_{i\in N}$; that is,
\[h_{(1)} \geq h_{(2)} \geq \dots \geq h_{(n)}.\]
{By definition, $F_{u_0, \mathbf{u}}(\overline{\mathbf{x}}) = \min_{j\in N}\{u_0 h_j + \sum_{i\in N} u_i\max\{0, h_i - h_j\}\}$, which} is equal to the optimal objective value of the following linear program (LP){, whose minimum is attained at $t\in\{h_i\}_{i\in N}$ (cf. Lemma \ref{lemma:G_partialmin})}:
\begin{equation*}
\begin{aligned}
\min_{t, \mathbf{r}} \quad &  u_0t + \sum_{i\in N} u_i r_i \\
\text{s.t.} \quad & r_i + t \geq h_i,\quad \forall\, i\in N,\\
& \mathbf{r} \geq \mathbf{0}.
\end{aligned}
\end{equation*}
The dual problem is
\begin{subequations}
\label{eq:F_u0_u_dual}
\begin{align}
\max_{\boldsymbol{\nu}} \quad &  \sum_{i\in N} h_i \nu_i  \\
\text{s.t.} \quad & \sum_{i\in N} \nu_i = u_0, \\
& \nu_i \leq u_i, \quad \forall\, i\in N, \\
& \boldsymbol{\nu} \geq \mathbf{0}. 
\end{align}
\end{subequations}

These LPs give important insights into the properties of $F_{u_0, \mathbf{u}}$, which we summarize in the lemma below. 

\begin{lemma}
\label{lemma:F_eval}
For any $(u_0, \mathbf{u})\in\mathcal{U}$ with $u_0 > 0$ and any $\overline{\mathbf{x}}\in\mathcal{X}$, 
\[F_{u_0, \mathbf{u}}(\overline{\mathbf{x}}) = f^{(j^*)}_{u_0, \mathbf{u}}(\overline{\mathbf{x}}) = u_0h_{(j^*)} + \sum_{j=1}^{j^* -1} u_{(j)} (h_{(j)}- h_{(j^*)})\]
where $j^*\in N$ satisfies 
\[\sum_{i=1}^{j^*-1} u_{(i)} < u_0, \quad \text{ and } \quad \sum_{i=1}^{j^*} u_{(i)} \geq u_0. \]
On the other hand, if $u_0 = 0$, then $F_{u_0, \mathbf{u}}(\overline{\mathbf{x}}) = 0$.
\end{lemma} 
\begin{proof}
When $u_0 > 0$, such $j^*$ must exist because $u_0 \leq \sum_{i\in N} u_i$. The dual problem \eqref{eq:F_u0_u_dual} can be solved greedily. We let 
\[\nu_{(i)} = \begin{cases}
u_{(i)}, & i < j^*, \\
u_0 - \sum_{i=1}^{j^*-1} u_{(i)}, & i = j^*, \\
0, & \text{otherwise.}
\end{cases}\]
The dual objective at $\boldsymbol{\nu}$ is 
\[
\sum_{i=1}^{j^*-1} u_{(i)} h_{(i)} + \left(u_0 - \sum_{i=1}^{j^*-1} u_{(i)}\right)h_{(j^*)} = u_0h_{(j^*)} +  \sum_{i=1}^{j^*-1} u_{(i)} (h_{(i)} - h_{(j^*)}),
\]
which is exactly the primal objective at $t = h_{(j^*)}$ and $r_i = \max\{0,h_{(i)} - h_{(j^*)}\}$ for $i\in N$. By strong duality, this objective value is optimal and is equal to $F_{u_0, \mathbf{u}}(\overline{\mathbf{x}})$. When $u_0 = 0$, both the primal and the dual objectives trivially attain zero. 
\end{proof}

\begin{remark}
\label{remark:u0_0}
Note that for $(u_0,\mathbf{u})\in\mathcal{U}$ with $u_0=0$, the constraints related to $F_{0, \mathbf{u}}$ in the convex hull description in Theorem \ref{thm:conv_H} are equivalent 
to $\mathbf{u}^\top\mathbf{y} \geq 0$. With $(u_0,\mathbf{u})=(0, \mathbf{1}^i)$, we obtain
$y_i \geq 0$ for $i\in N$. This explains why $\mathbf{y}\geq \mathbf{0}$ is ensured in Theorem \ref{thm:conv_H} without explicitly stating. In addition, the inequalities for any  $(0,\mathbf{u})\in\mathcal{U}$ with $\mathbf{u}\ne \mathbf{1}^i$ for any $i\in N$, are dominated by $y_i \geq 0$ for $i\in N$. 
\end{remark}

In the next lemma, we provide an intuition behind another special case of $F_{u_0, \mathbf{u}}$ where $\mathbf{u}\in\mathbb{B}^n$ and $u_0\in \{1,\dots, \mathbf{u}^\top\mathbf{1}\}$. This lemma will be useful in Section \ref{sect:eg_CMIX}. 

\begin{lemma}
\label{lemma:F_u0_binary_u}
Suppose $\mathbf{u} = \mathbf{1}^{\Gamma}$ for some non-empty $\Gamma\subseteq N$, and $u_0 \in \{1,\dots, |\Gamma|\}$. Then 
\[F_{u_0,\mathbf{u}}(\mathbf{x}) = \max_{S\subseteq \Gamma, |S| = u_0}\left\{\sum_{i\in S} h^i(\mathbf{x})\right\}\]
for all $\mathbf{x}\in\mathcal{X}$. 
\end{lemma}
\begin{proof}
We sort $\{h_i\}_{i\in\Gamma}$ such that 
\[h_{(1)^\Gamma} \geq \dots \geq h_{(|\Gamma|)^\Gamma}.\]
With the special form of $u_0$ and $\mathbf{u}$, $u_{(i)^\Gamma} = 1$ for all $i\in \{1,\dots, |\Gamma|\}$, zero otherwise. Note that 
\[\sum_{i=1}^{u_0 - 1} u_{(i)^\Gamma} < u_0, \quad \sum_{i=1}^{u_0} u_{(i)^\Gamma} = u_0. \]
By Lemma \ref{lemma:F_eval}, 
\begin{align*}
F_{u_0,\mathbf{u}}(\mathbf{x})& = f^{(u_0)^\Gamma}_{u_0,\mathbf{u}}(\mathbf{x}) \\
& = u_0h_{(u_0)^\Gamma} + \sum_{j=1}^{u_0 -1} (h_{(j)^\Gamma}- h_{(u_0)^\Gamma}) \\
& = h_{(u_0)^\Gamma} + \sum_{j=1}^{u_0 -1} (h_{(j)^\Gamma}- h_{(u_0)^\Gamma} + h_{(u_0)^\Gamma}) \\
& = \sum_{j=1}^{u_0} h_{(j)^\Gamma} \\
& = \max_{S\subseteq \Gamma, |S| = u_0}\left\{\sum_{i\in S} h_i \right\},
\end{align*}
for all $\mathbf{x}\in\mathcal{X}$. 
Intuitively, $F_{u_0, \mathbf{u}}(\mathbf{x})$ evaluates the sum of 
the $u_0$-th highest entries among $\{h^i(\mathbf{x})\}_{i\in\Gamma}$ for this choice of $(u_0,\mathbf{u})$. 
\end{proof}

\subsection{Separation of MISEPIs}
\label{sect:sepa_MISEPIs}
Recall that when Assumption \hyperlink{A1}{1} holds and when $F_{u_0, \mathbf{u}}$ are L$^\natural$-convex, $\conv{\mathcal{P}^H}$ is given by infinitely many MISEPIs \eqref{eq:MISEPI}. We address the challenge of handling so many constraints by discussing the separation problem of MISEPIs. 
Let any $(\overline{w}, \overline{\mathbf{y}}, \overline{\mathbf{x}})\in \mathbb{R}\times \mathbb{R}^{n}_+\times\overline{\mathcal{X}}$ be given. 
Here, we include the trivial constraints $\mathbf{y}\geq \mathbf{0}$ in the relaxation due to Remark \ref{remark:u0_0}. This way, it suffices to only consider $(u_0, \mathbf{u})\in\mathcal{U}$ with $u_0 > 0$. Fix any such $(u_0, \mathbf{u})\in\mathcal{U}$. We let $\overline{\eta} =  u_0 \overline{w} + \mathbf{u}^\top\overline{\mathbf{y}}$. With Algorithm \ref{alg:frac_greedy}, we can exactly separate an SEPI for $\mathcal{P}^{F_{u_0, \mathbf{u}}}$ associated with $\mathbf{p}\in\underline{\mathcal{X}}$ and $\boldsymbol{\delta}\in\mathfrak{S}(N)$ that attains maximal violation at $(\overline{\eta},  \overline{\mathbf{x}})$. Thanks to the exact separation algorithm (Algorithm \ref{alg:frac_greedy}) and the structure of SEPIs, we have   $\boldsymbol{\lambda} \geq \mathbf{0}$ with $\sum_{k=0}^n \lambda_k = 1$, such that $\overline{\mathbf{x}} = \sum_{k=0}^n\lambda_k \mathbf{p}^k$ and 
\[\mathbf{s}^{\mathbf{p},\boldsymbol{\delta}, u_0, \mathbf{u}^\top}\overline{\mathbf{x}} + s_0^{\mathbf{p},\boldsymbol{\delta}, u_0, \mathbf{u}} = \sum_{k=0}^n\lambda_kF_{u_0, \mathbf{u}}(\mathbf{p}^k). \]
Recall $\mathbf{p}^k = \mathbf{p} + \sum_{j=1}^k \mathbf{1}^{\delta(j)}$, and $\mathbf{p}^{0} = \mathbf{p}$. We highlight that Algorithm \ref{alg:frac_greedy} only depends on $\overline{\mathbf{x}}$, so no matter how we vary $(u_0, \mathbf{u})\in\mathcal{U}$, we obtain the same $\mathbf{p}$ and $\boldsymbol{\delta}$.\\

Next, we explore how to determine $(u_0, \mathbf{u})\in\mathcal{U}$ so that the MISEPI \eqref{eq:MISEPI} with respect to $u_0, \mathbf{u}, \mathbf{p}$ and $\boldsymbol{\delta}$ leads to the maximal violation at $(\overline{w}, \overline{\mathbf{y}}, \overline{\mathbf{x}})$. For ease of notation, let $h^k_i = h^i\left(\mathbf{p}^{k}\right)$ and $a^k_i = \lambda_k h^k_i$ for $k\in N\cup\{0\}$ and $i\in N$. The violation of \eqref{eq:MISEPI} at $(\overline{w}, \overline{\mathbf{y}}, \overline{\mathbf{x}})$ is $\sum_{k=0}^n\lambda_kF_{u_0, \mathbf{u}}(\mathbf{p}^k) - u_0 \overline{w} + \mathbf{u}^\top\overline{\mathbf{y}}$.
The separation problem can thus be stated as 
\begin{equation*}
\begin{aligned}
\max_{u_0,\mathbf{u}} & \quad \sum_{k=0}^n\lambda_k \min_{t^k\in\mathbb{R}} \left\{ u_0t^k + \sum_{i\in N} u_i \max\left\{ 0, h^k_i - t^k\right\} \right\} - \overline{w}u_0 - \overline{\mathbf{y}}^\top\mathbf{u} \\
\text{s.t.} & \quad \sum_{i\in N} u_i \geq u_0, \\
& \quad (u_0, \mathbf{u}) \geq \mathbf{0}. 
\end{aligned}
\end{equation*}
Recall that we only need to consider $u_0 > 0$, so we scale $u_0$ such that it is one. Thanks to \eqref{eq:F_u0_u_dual}, the separation problem becomes 
\begin{subequations}
\label{eq:sepa_prob_single_max}
\begin{align}
\max_{\boldsymbol{\nu}, \mathbf{u}} \quad &\sum_{k=0}^n \sum_{i\in N} a^k_i \nu^k_i   - \sum_{i\in N}\overline{y}_i u_i - \overline{w}\\
\text{s.t.} & \sum_{i\in N} \nu^k_i = 1, \quad \forall\, k\in N\cup\{0\}, \label{eq:sepa_prob_single_max2}  \\
& \nu^k_i \leq u_i, \quad \forall\, i\in N, k\in N\cup\{0\}, \label{eq:sepa_prob_single_max3} \\
& \boldsymbol{\nu}^k \geq \mathbf{0}, \quad \forall\, k\in N\cup\{0\}, \label{eq:sepa_prob_single_max4} \\
& \mathbf{u} \geq \mathbf{0}. \label{eq:sepa_prob_single_max5}
\end{align}
\end{subequations}

We drop the constraint $\sum_{i\in N} u_i \geq 1$ because it is implied by constraints \eqref{eq:sepa_prob_single_max2}--\eqref{eq:sepa_prob_single_max3}. 
If the optimal objective value of the separation LP \eqref{eq:sepa_prob_single_max}  is strictly positive, then a violation occurs. The MISEPI associated with $\mathbf{p}$, $\boldsymbol{\delta}$ and $(1, \mathbf{u}^*)$, where $\mathbf{u}^*$ is an optimal solution to \eqref{eq:sepa_prob_single_max}, is most violated by $(\overline{w}, \overline{\mathbf{y}}, \overline{\mathbf{x}})$. 

\begin{proposition}
\label{prop:complexity}
Given any function $H$ defined in \eqref{eq:H}, if Assumption \hyperlink{A1}{1} is satisfied and $F_{u_0, \mathbf{u}}$ are L$^\natural$-convex for all $(u_0, \mathbf{u})\in\mathcal{U}$, then the mixed-integer problem $\min_{(\mathbf{x},\mathbf{y})\in\mathcal{X}\times\mathbb{R}^n_+} H(\mathbf{x}, \mathbf{y})$ is polynomially solvable. 
\end{proposition}
\begin{proof}
The separation LP \eqref{eq:sepa_prob_single_max} is polynomially solvable. The statement follows from the equivalence of optimization and separation for LPs \citep{grotschel1981ellipsoid}. 
\end{proof}

\begin{remark}
Intuitively, the separation problem for MISEPI is easy due to L$^\natural$-convexity of $F_{u_0, \mathbf{u}}$ for any $(u_0, \mathbf{u})\in\mathcal{U}$. This special property ensures that the separation for $\mathbf{p}$ and $\boldsymbol{\delta}$  relies only on $\overline{\mathbf{x}}$, and the returned parameters are the same for all $(u_0, \mathbf{u})\in\mathcal{U}$. Thus, the separation for MISEPI can be decomposed into two steps, namely first identifying $\mathbf{p}$ and $\boldsymbol{\delta}$ with Algorithm \ref{alg:frac_greedy}, and then obtaining $(1, \mathbf{u}^*)\in\mathcal{U}$ with the LP \eqref{eq:sepa_prob_single_max}. 
\end{remark}

\begin{remark}
We remark that it is generally hard to obtain the optimal $\mathbf{u}^*$ for \eqref{eq:sepa_prob_single_max} in closed form. With any feasible $\boldsymbol{\nu}$, we can construct the best feasible $\mathbf{u}$ with $u_i = \max_{k\in N\cup\{0\}}\{\nu^k_i\}$. In addition, with any feasible $\mathbf{u}$, the best assignment of $\boldsymbol{\nu}$ can easily be obtained greedily. However, simultaneously optimizing $\mathbf{u}$ and $\boldsymbol{\nu}$ is difficult because they are highly interdependent.
\end{remark}

Nonetheless, we can still exploit the structure of \eqref{eq:sepa_prob_single_max} to characterize the optimal solutions. In what follows, we use the notation $GL(m,2)$ to represent the set of all invertible $m\times m$ binary matrices, where $m\in \mathbb{Z}_+$. For any non-empty $M\subseteq N$, let $\mathbf{u}^M = [u_i]_{i\in M} \in \mathbb{R}^{|M|}$ be the corresponding sub-vector of $\mathbf{u}$. We construct a collection of vectors $\mathbf{u}\in\mathbb{R}^n_+$ in the following way:

\begin{equation}
\label{eq:U_prime}
\mathcal{U}' := \bigcup_{M\in 2^N\setminus \{\emptyset\}}\left\{\mathbf{u}\in\mathbb{R}^n_+ : \; 
  \mathbf{u}^M \in \bigcup_{B\in GL(|M|,2)} \left\{B^{-1}\mathbf{1}\right\}, \; 
 \mathbf{u}^{N\setminus M} = \mathbf{0}
\right\}.
\end{equation}
The cardinality of $\mathcal{U}'$ is upper bounded by $\sum_{m=1}^n {n\choose m} 2^{m^2}$; even though large, $\mathcal{U}'$ is finite.  

\begin{lemma}
\label{lemma:legal_u}
Any $\mathbf{u}\in\mathcal{U}'$ defined in \eqref{eq:U_prime} satisfies $(1,\mathbf{u})\in\mathcal{U}$. 
\end{lemma}
\begin{proof}
Consider an arbitrary $\mathbf{u}\in\mathcal{U}'$. There exists a non-empty $M\subseteq N$ and $B\in GL(|M|,2)$ such that $B\mathbf{u}^M = \mathbf{1}$. Given that $B$ is binary, $B_{i\cdot} \leq \mathbf{1}^\top$ for every row $i\in \{1,\dots,|M|\}$ in $B$. Thus, $\mathbf{1}^\top \mathbf{u}^M \geq B_{i\cdot}\mathbf{u}^M = 1$ for $i\in \{1,\dots,|M|\}$, because $\mathbf{u}^M \geq \mathbf{0}$. Hence, $(1,\mathbf{u})\in\mathcal{U}$.
\end{proof}

\begin{lemma}
\label{lemma:U_prime}
Given the separation LP \eqref{eq:sepa_prob_single_max} for any $(\overline{w}, \overline{\mathbf{y}}, \overline{\mathbf{x}})\in\mathbb{R}\times\mathbb{R}^n_+\times\overline{\mathcal{X}}$, every optimal $\mathbf{u}^*$ belongs to $\mathcal{U}'$ defined in \eqref{eq:U_prime}. 
\end{lemma}
\begin{proof}
Intuitively, an optimal $\boldsymbol{\nu}^*$ with respect to the optimal $\mathbf{u}^*$ is constructed in a greedy fashion that prioritizes the highest $a^k_i$ with $\nu^{*k}_i = u^*_i$ until \eqref{eq:sepa_prob_single_max2} is satisfied. 
For every $k\in N\cup\{0\}$, sort $\{a^k_i\}_{i\in N}$ in any (in case of ties)  descending order such that  
\[a^k_{(1)^k} \geq a^k_{(2)^k} \geq  \dots \geq a^k_{(n)^k}.\]
Then $\boldsymbol{\nu}^{*k}$ assumes the form 
\[\nu^{*k}_{(i)^k} = \begin{cases}
u^*_{(i)^k}, & i \in \{1, \dots, \ell_k - 1\}, \\
1-\sum_{j=1}^{i - 1} u^*_{(j)^k}, & i = \ell_k, \\
0, & i \in \{\ell_k + 1, \dots, n\},
\end{cases}\]
for some $\ell_k\in N$, where $k\in N\cup\{0\}$. More precisely, $\ell_k$ satisfies 
\[\sum_{i=1}^{\ell_k-1} u^*_{(i)^k} < 1, \quad \text{and} \quad \sum_{i=1}^{\ell_k} u^*_{(i)^k} \geq 1. \]
In addition to the $n+1$ equality constraints \eqref{eq:sepa_prob_single_max2}, the form of $\boldsymbol{\nu}^*$ results in $n^2-1$ binding constraints \eqref{eq:sepa_prob_single_max3}--\eqref{eq:sepa_prob_single_max4}. 
Let $M := \bigcup_{k\in N\cup\{0\}} \{(1)^k, \dots, (\ell_k)^k\}$ and $m = |M|$. We know that $u^*_i = 0$ for every $i\in N\setminus M$ because $u^*_i = \max_{k\in N\cup\{0\}} \{\nu^{*k}_i\}$. These are $n-m$ additional binding constraints \eqref{eq:sepa_prob_single_max5}. 
These $n^2 + 2n - m$ constraints are linearly independent because the only possible redundancy would occur when $u^*_{(i)^k} = \nu^{*k}_{(i)^k} = 0$; however, when $u^*_{(i)^k} = 0$, exactly one of $\nu^{*k}_{(i)^k} = u^*_{(i)^k}$ and $\nu^{*k}_{(i)^k} = 0$ is included in the set of binding constraints, so the redundancy is avoided.  
Given that $(\boldsymbol{\nu}^*, \mathbf{u}^*)$ is a basic solution, there must exist another $m$ constraints in the optimal basis. The only possibility is that there exists $L\subset N\cup\{0\}$ with $|L| = m$, such that $\nu^k_{(\ell_k)^k} =  u_{(\ell_k)^k}$ for every $k \in L$. Note that $u_{(\ell_k)^k} > 0$ for all $k$ by definition of $\ell_k$. We remark that adding these binding constraints to the aforementioned set of tight constraints is equivalent to adding $\sum_{i = 1}^{\ell_k} u^*_{(i)^k} = 1$ for every $k\in L$. 
Therefore, the optimal basis is the full-rank constraint matrix of the following system: 
\begin{subequations}
\begin{align*}
 \sum_{i\in N} \nu^k_i & = 1, \quad \forall\, k\in N\cup\{0\}, \\
 u_{(i)^k} - \nu^{k}_{(i)^k} & = 0, \quad \forall\, i \in \{1, \dots, \ell_k - 1\}, \; k\in N\cup\{0\},\\
 u_{(\ell_k)^k} - \nu^{k}_{(\ell_k)^k} & = 0, \quad \forall\, k\in L,\\
 \nu^{k}_{(i)^k} & = 0, \quad \forall\, i \in \{\ell_k + 1, \dots, n\}, \; k\in N\cup\{0\},\\
 u_i & = 0, \quad \forall\,i\in N\setminus M. 
\end{align*}
\end{subequations}
By elementary row operations, the system transforms into 
\begin{subequations}
\begin{align}
\sum_{i = 1}^{\ell_k} \nu^k_{(i)^k} & = 1, \quad \forall\, k\in N\cup\{0\}, \\
 u_{(i)^k} - \nu^{k}_{(i)^k} & = 0, \quad \forall\, i \in \{1, \dots, \ell_k - 1\}, \; k\in N\cup\{0\},\\
  \nu^{k}_{(i)^k} & = 0, \quad \forall\, i \in \{\ell_k + 1, \dots, n\}, \; k\in N\cup\{0\},\\
 \sum_{i = 1}^{\ell_k} u_{(i)^k} &= 1 \quad \forall\, k\in L,\label{eq:basis_contr}\\
 u_i & = 0, \quad \forall\,i\in N\setminus M. 
\end{align}
\end{subequations}

Now, construct a matrix $B\in\{0,1\}^{L\times M}$ where 
\[B_{ki} = \begin{cases}
1, & i\in \{(1)^k , \dots, (\ell_k)^k\}, \\
0, & i\in M\setminus \{(1)^k , \dots, (\ell_k)^k\}
\end{cases}\]
for all $k\in L$. In matrix form, constraints \eqref{eq:basis_contr} are exactly 
\[B\mathbf{u}^M = \mathbf{1}. \]
Matrix $B\in\mathbb{B}^{m\times m}$ contributes full row rank to the system of equality constraints, so $B\in GL(m,2)$. The optimal $\mathbf{u}^*\in\mathbb{R}^n_+$ satisfies $B\mathbf{u}^{*M} = \mathbf{1}$, and $\mathbf{u}^{*N\setminus M} = \mathbf{0}$.
Hence, $\mathbf{u}^*\in\mathcal{U}'$.
\end{proof}

\begin{remark}
We provide more insights into why the optimal solutions $\mathbf{u}^*$ to the separation problems \eqref{eq:sepa_prob_single_max} with respect to all $(\overline{w}, \overline{\mathbf{y}}, \overline{\mathbf{x}})\in\mathbb{R}\times\mathbb{R}^n_+\times\overline{\mathcal{X}}$ form a finite set $\mathcal{U}'$. We emphasize that constructing the optimal solution to \eqref{eq:sepa_prob_single_max} in Lemma \ref{lemma:U_prime} does not directly use the actual values of $\{\mathbf{a}^k\}_{k\in N\cup\{0\}}$. Rather, it suffices to know the set of $n+1$ descending orders of $\{a^k_i\}_{i\in N}$, which are exactly the descending orders of $\{h^k_i\}_{i\in N}$, for $k\in N\cup \{0\}$. Recall that there are as many $\{\{h^k_i\}_{i\in N}\}_{k\in N\cup\{0\}}$ as the pairs of $\mathbf{p}\in\underline{\mathcal{X}}$ and $\boldsymbol{\delta}\in\mathfrak{S}(N)$. Even though $|\underline{\mathcal{X}}|$ could be infinite, there are finitely many (at most $n!^{n+1}$) possible sets of descending orders of $\{\{h^k_i\}_{i\in N}\}_{k\in N\cup\{0\}}$. 
\end{remark}

Following the previous lemma, we can replace the infinite set $\mathcal{U}$ in the convex hull description from Theorem \ref{thm:conv_SEPI} by its finite subset $\mathcal{U}'$. We formalize this observation in the next proposition.

\begin{proposition}
\label{prop:finite_CV}
Under Assumption \hyperlink{A1}{1}, suppose $F_{u_0, \mathbf{u}}$ are L$^\natural$-convex for all $(u_0, \mathbf{u})\in\mathcal{U}$. Let 
\[\mathcal{CP}' := \left\{(w, \mathbf{y}, \mathbf{x})\in \mathbb{R}\times \mathbb{R}^{n}_+\times\overline{\mathcal{X}} :  (w + \mathbf{u}^\top\mathbf{y}, \mathbf{x}) \in \conv{\mathcal{P}^{F_{1, \mathbf{u}}}}, \forall\, \mathbf{u}\in\mathcal{U}' \right\}.\] Then 
$\conv{\mathcal{P}^H} = \mathcal{CP}'$. 
\end{proposition}
\begin{proof}
By Lemma \ref{lemma:legal_u}, $\{1\}\times \mathcal{U}'\subset \mathcal{U}$, so $\mathcal{CP}' \supseteq\conv{\mathcal{P}^H}$. We show the reverse containment by contradiction. Suppose that there exists $(\overline{w}, \overline{\mathbf{y}}, \overline{\mathbf{x}}) \in\mathcal{CP}' \setminus\conv{\mathcal{P}^H}$. Then $\overline{\mathbf{x}}$ gives rise to $\mathbf{p}\in\underline{\mathcal{X}}$, $\boldsymbol{\delta}\in\mathfrak{S}(N)$, and the corresponding separation LP \eqref{eq:sepa_prob_single_max} with an optimal solution $\mathbf{u}^*$. Given that $(\overline{w}, \overline{\mathbf{y}}, \overline{\mathbf{x}}) \notin\conv{\mathcal{P}^H}$, this point incurs a positive violation of the MISEPI associated with $\mathbf{p}$, $\boldsymbol{\delta}$, and $(1,\mathbf{u}^*)$. However, by Lemma \ref{lemma:U_prime}, $\mathbf{u}^*\in\mathcal{U}'$, which suggests that this most violated MISEPI is already valid for $\mathcal{CP}'$. Therefore, $\mathcal{CP}' \setminus\conv{\mathcal{P}^H} = \emptyset$.  
\end{proof}

It follows that if $\conv{\mathcal{P}^{F_{1, \mathbf{u}}}}$ is polyhedral for every $\mathbf{u}\in\mathcal{U}'$, then $\conv{\mathcal{P}^H}$ is polyhedral. This is trivially true when $\mathcal{X}$ is finite. Moreover, the set $\mathcal{U}'$ can be further refined, but to do so, it is crucial to exploit the structure of the functions $\{h^i\}_{i\in N}$. We will address this in Section \ref{sec:structuredH}.

\subsection{Facet-defining MISEPIs}
\label{sect:FD_MISEPI}

In this section, we explore the conditions under which an MISEPI \eqref{eq:MISEPI} is facet-defining for $\conv{\mathcal{P}^H}$. Our discussion is relevant when Assumption \hyperlink{A1}{1} is satisfied and $F_{u_0, \mathbf{u}}$ are L$^\natural$-convex for all $(u_0, \mathbf{u})\in\mathcal{U}$. We focus on the instances where $u_0 > 0$ because otherwise the MISEPIs are dominated by the trivial inequalities $y_i \geq 0$ for $i\in N$ according to Remark \ref{remark:u0_0}. Additionally, we will only consider $n\geq 2$. When $n=1$, $H(x,y) = h(x) - y$ where $x\in\mathcal{X}\subseteq \mathbb{Z}$ and $y \geq 0$. By Assumption \hyperlink{A1}{1}, $h$ is L$^\natural$-convex, so this case is trivial.\\

Consider any $\mathbf{p}^0\in\underline{\mathcal{X}}$ and any $\boldsymbol{\delta}\in\mathfrak{S}(N)$. Recall that 
\[\mathbf{p}^k = \mathbf{p}^0 + \sum_{j=1}^k \mathbf{1}^{\delta(j)}, \; \forall\, k \in N,\]  
\[h^k_i = h^i\left(\mathbf{p}^{k}\right), \; \forall\, k \in N \cup \{0\}, i\in N, \]
and 
\[h^k_{(1)^k} \geq h^k_{(2)^k} \geq \dots \geq h^k_{(n)^k}.\] 

\begin{proposition}
\label{prop:full_dim_pH}
If $\mathcal{X} \supseteq \mathcal{B}_{\mathbf{p}}$ for any $\mathbf{p}\in\mathbb{Z}^n$, then $\conv{\mathcal{P}^H}$ is full-dimensional. 
\end{proposition}
\begin{proof}
We construct $2n+2$ affinely independent points in $\mathcal{P}^H$. The points $(h^k_{(1)^k}, \mathbf{0}, \mathbf{p}^k) \in \mathcal{P}^H$ for $k\in N\cup\{0\}$ because $h^k_{(1)^k} \geq h^k_{(j)^k} - 0$ for all $j\in N$. Moreover, $(h^0_{(1)^0} + 1, \mathbf{0}, \mathbf{p}^0) \in \mathcal{P}^H$, and $(h^0_{(1)^0}, \mathbf{1}^j, \mathbf{p}^0) \in \mathcal{P}^H$ for all $j\in N$. Subtracting $(h^0_{(1)^0}, \mathbf{0}, \mathbf{p}^0)$ from all other $2n + 1$ points, we obtain 
\begin{align*}
 \left(h^k_{(1)^k} - h^0_{(1)^0}, \mathbf{0}, \sum_{j=1}^k \mathbf{1}^{\delta(j)}\right), & \, \forall\, k\in N, \\
 (1, \mathbf{0}, \mathbf{0}),& \\
 (0, \mathbf{1}^j, \mathbf{0}),& \, \forall\, j\in N, 
\end{align*}
which are linearly independent. Therefore, $\conv{\mathcal{P}^H}$ is full-dimensional. 
\end{proof}

Consider a given $\mathbf{p}^0\in\underline{\mathcal{X}}$ and $\boldsymbol{\delta}\in\mathfrak{S}(N)$. For any $(u_0,\mathbf{u})\in\mathcal{U}$, the face of the MISEPI \eqref{eq:MISEPI} contains $(w, \mathbf{y}, \mathbf{x})\in\mathcal{P}^H$ such that 
\[u_0 w + \mathbf{u}^\top\mathbf{y} = \mathbf{s}^{\mathbf{p}^0,\boldsymbol{\delta}, u_0, \mathbf{u}^\top}\mathbf{x} + s_0^{\mathbf{p}^0,\boldsymbol{\delta}, u_0, \mathbf{u}}.\]
By Proposition \ref{prop:SEPI_FD}, we know that 
\[ F_{u_0,\mathbf{u}}(\mathbf{p}^k) = \mathbf{s}^{\mathbf{p}^0,\boldsymbol{\delta}, u_0, \mathbf{u}^\top}\mathbf{p}^k + s_0^{\mathbf{p}^0,\boldsymbol{\delta}, u_0, \mathbf{u}}, \, \forall\, k\in N\cup\{0\}.\]
Thus, we next explore $(w, \mathbf{y})$ such that $u_0 w + \mathbf{u}^\top\mathbf{y} = F_{u_0,\mathbf{u}}(\mathbf{p}^k)$. Consider any $(u_0,\mathbf{u})\in\mathcal{U}$. For every $k\in N\cup\{0\}$, let $j^*_k\in N$ be the index such that 
\begin{equation}
\label{eq:def_j_star}
\sum_{j=1}^{j^*_k-1} u_{(j)^k} < u_0, \quad \text{and} \quad \sum_{j=1}^{j^*_k} u_{(j)^k} \geq u_0,
\end{equation}
as in Lemma \ref{lemma:F_eval}. Recall that $(\cdot)^k$ are permutations governed by $\{h^k_i\}_{i\in N}$ for $k\in N\cup\{0\}$. Each $(\cdot)^k$ gives rise to a permuted version of $\mathbf{u}$, and $j^*_k$ must exist according to the definition of $\mathcal{U}$. These permutations can coincide, but are generally different for $k\in N\cup\{0\}$. Thus, for any $k,k'\in N\cup\{0\}$, we have to include the subscripts $k, k'$ to distinguish $j^*_k$ and $j^*_{k'}$. We include an example to illustrate our notation. 

\begin{example}
\label{eg:j_star_k}
Suppose $n=3$, $u_0 = 1$, $\mathbf{u} = (0.3,1.2,0.7)$, and 
\begin{equation*}
h^1  = (0.8, 0.5, 0.1), \; h^2  = (-0.2, 0.5, 0.1), \; h^3  = (-0.2, -0.5, 0.1), \; h^4  = (-0.2, -0.5, -0.9).\\
\end{equation*}
We obtain 
\begin{align*}
h^1_1 > h^1_2 > h^1_3, \quad \longrightarrow\quad (1)^1 = 1, (2)^1 = 2, (3)^1 = 3, \\
h^2_2 > h^2_3 > h^2_1, \quad \longrightarrow\quad (1)^2 = 2, (2)^2 = 3, (3)^2 = 1, \\
h^3_3 > h^3_1 > h^3_2, \quad \longrightarrow\quad (1)^3 = 3, (2)^3 = 1, (3)^3 = 2, \\
h^4_1 > h^4_2 > h^4_3, \quad \longrightarrow\quad (1)^4 = 1, (2)^4 = 2, (3)^4 = 3.
\end{align*}
When $k=1$, 
\[u_{(1)^1} = 0.3 < 1 = u_0, \quad u_{(1)^1} + u_{(2)^1} = 0.3 + 1.2 > u_0,\]
so $j^*_1 = 2$. When $k=2$, $j^*_2 = 0$ because 
\[0 < u_0, \quad u_{(1)^2} = 1.2 > u_0.\]
Similarly, when $k=3$, $j^*_3 = 2$ because 
\[u_{(1)^3} = 0.7 < u_0, \quad u_{(1)^3} + u_{(2)^3} = 0.7 + 0.3 = 1 = u_0. \]
Note that the second inequality holds at equality in this case. 
\end{example}

In Observations \ref{obs:tight_pt1}--\ref{obs:tight_pt3} and Lemma \ref{lemma:tight_pt2}, we explore elements of $\mathcal{P}^H$ that belong to the face of the MISEPI associated with $(u_0, \mathbf{u})$, $\mathbf{p}^0$, and $\boldsymbol{\delta}$.

\begin{observation}
\label{obs:tight_pt1}
For any $(u_0,\mathbf{u})\in\mathcal{U}$ and every $k\in N\cup\{0\}$,
\[u_0 w^k + \mathbf{u}^\top\mathbf{y}^k = F_{u_0,\mathbf{u}}(\mathbf{p}^k),\]
where 
\begin{equation}
\label{eq:construct_wy_k}
\begin{aligned}
w^k & = h^k_{(j^*_k)^k}, \\
y^k_i & = \max\left\{0, h^k_i - h^k_{(j^*_k)^k}\right\}, \;\forall\, i\in N.
\end{aligned}
\end{equation}  This follows from Lemma \ref{lemma:F_eval}. Notice that $(w^k, \mathbf{y}^k, \mathbf{p}^k)\in\mathcal{P}^H$.
\end{observation}

The second inequality in \eqref{eq:def_j_star} may hold at equality for some $k\in N\cup\{0\}$; that is, 
$j^*_k\in N$ further satisfies 
\[\sum_{j=1}^{j^*_k} u_{(j)^k} = u_0.\]
We have seen an instance in Example \ref{eg:j_star_k} with $k=3$.
In this case, we can identify additional points where the MISEPI is tight. We explain what these points are in the next lemma. 

\begin{lemma}
\label{lemma:tight_pt2}
Suppose $k\in N\cup\{0\}$ further satisfies that $\sum_{i=1}^{j^*_k} u_{(i)^k} = u_0$. Then we have 
\[u_0 \tilde{w}^k + \mathbf{u}^\top\tilde{\mathbf{y}}^k = F_{u_0,\mathbf{u}}(\mathbf{p}^k),\]
where 
\begin{equation}
\label{eq:construct_wy'_k}
\begin{aligned}
\tilde{w}^k & = h^k_{(j^*_k+1)^k}, \\
\tilde{y}^k_i & = \max\left\{0, h^k_i - h^k_{(j^*_k+1)^k}\right\}, \;\forall\, i\in N.
\end{aligned}
\end{equation}
\end{lemma}
\begin{proof}
By construction, $(\tilde{w}^k, \tilde{\mathbf{y}}^k, \mathbf{p}^k)\in\mathcal{P}^H$. We observe that 
\begin{align*}
u_0 \tilde{w}^k + \sum_{i\in N} u_i \tilde{y}^k_i & = u_0 h^k_{(j^*_k + 1)^k} + \sum_{j =1}^{j^*_k} u_{(j)^k}(h^k_{(j)^k} - h^k_{(j^*_k + 1)^k}) \\
& = u_0 h^k_{(j^*_k)^k} + \sum_{j =1}^{j^*_k} u_{(j)^k}(h^k_{(j)^k} - h^k_{(j^*_k)^k}) + \left(\sum_{j=1}^{j^*_k} u_{(j)^k} - u_0\right) (h^k_{(j^*_k)^k} - h^k_{(j^*_k+1)^k}) \\
& = u_0 h^k_{(j^*_k)^k} + \sum_{j =1}^{j^*_k - 1} u_{(j)^k}(h^k_{(j)^k} - h^k_{(j^*_k)^k}) + 0\\
& = F_{u_0, \mathbf{u}}(\mathbf{p}^k). 
\end{align*}
\end{proof}

\begin{observation}
\label{obs:tight_pt3}
For any $i\in N$ with $u_i = 0$, 
\[u_0 w^k + \mathbf{u}^\top(\mathbf{y}^k + \mathbf{1}^i) = F_{u_0,\mathbf{u}}(\mathbf{p}^k),\]
where $(w^k, \mathbf{y}^k, \mathbf{p}^k)$ is given by \eqref{eq:construct_wy_k}. Furthermore, the point $(w^k, \mathbf{y}^k + \mathbf{1}^i, \mathbf{p}^k)$ remains in $\mathcal{P}^H$.
\end{observation}

So far, we have found points in $\mathcal{P}^H$ that are also on the face of an MISEPI. When $2n+1$ of these points are affinely independent, we can conclude that this MISEPI is facet-defining. \\

We now discuss how to construct $(u_0, \mathbf{u})\in\mathcal{U}$ for the given $\mathbf{p}^0\in\underline{\mathcal{X}}$ and $\boldsymbol{\delta}\in\mathfrak{S}(N)$, such that the resulting MISEPI is facet-defining. 
For every $k\in N\cup\{0\}$, pick any $\ell_k\in N$, and construct a binary incidence matrix $A(\boldsymbol{\ell})\in\{0,1\}^{(n+1)\times n}$ such that 
\[A(\boldsymbol{\ell})_{k(j)^k} = \begin{cases}
1, & \text{if } j \leq \ell_k, \\
0, & \text{otherwise.}
\end{cases}\]
In other words, the $k$-th row of $A(\boldsymbol{\ell})$ is the incidence vector for the support of the top $\ell_k$ elements in $\{h^k_i\}_{i\in N}$. Every row of $A(\boldsymbol{\ell})$ contains at least one entry of one. Let $M$ denote all the non-zero columns in $A(\boldsymbol{\ell})$, and let $|M| = m$.

\begin{proposition}
\label{prop:MISEPI_FD}
Fix any $\mathbf{p}\in\underline{\mathcal{X}}$ and $\boldsymbol{\delta}\in\mathfrak{S}(N)$. We assume that $\mathcal{B}_\mathbf{p} \subseteq \mathcal{X}$. Suppose that $\{h^k_i\}_{i\in N}$ are all distinct for every $k\in N\cup\{0\}$, and that $A(\boldsymbol{\ell})$ contains a submatrix $B\in GL(m, 2)$ such that $B^{-1}\mathbf{1} > \mathbf{0}$ entry-wise. 
Then the MISEPI associated with $\mathbf{p}, \boldsymbol{\delta}$ and $(1, \mathbf{u})$ is facet-defining for $\conv{\mathcal{P}^H}$, where $\mathbf{u}^M = B^{-1}\mathbf{1}$ and $\mathbf{u}^{N\setminus M} = \mathbf{0}$. 
\end{proposition}
\begin{proof}
First, we notice that $\mathbf{u}\in\mathcal{U}'$ given in \eqref{eq:U_prime} by construction, and $(1, \mathbf{u})\in\mathcal{U}$ from Lemma \ref{lemma:legal_u}.  

\hspace{12pt} By Proposition \ref{prop:full_dim_pH}, $\conv{\mathcal{P}^H}$ is full dimensional.
Thus, we need to construct $2n+1$ affinely independent points on the face of the MISEPI. We construct $\{(w^k, \mathbf{y}^k, \mathbf{p}^k)\}_{k\in N\cup\{0\}}$ and $\{(w^0, \mathbf{y}^0 + \mathbf{1}^i, \mathbf{p}^0)\}_{i\in N\setminus M}$ as described in 
Observations \ref{obs:tight_pt1} and \ref{obs:tight_pt3}. These $2n+1-m$ points lie on the face of the proposed MISEPI. Let $L\subset N\cup\{0\}$ be the rows in $A(\boldsymbol{\ell})$ that overlap with $B$, then $|L| = m$ because $B$ is a square matrix. As described in Lemma \ref{lemma:tight_pt2}, we construct $m$ points, namely $(\tilde{w}^k, \tilde{\mathbf{y}}^k, \mathbf{p}^k)$ for $k\in L$. Note that $j^*_k = \ell_k$ for $k\in L$. This is because $B$ is a sub-matrix of the non-zero columns of $A(\boldsymbol{\ell})$, and $B_{k\cdot}\mathbf{u}^M =\sum_{i=1}^{\ell_k}u_{(i)^k} = 1 = u_0$ for every $k\in L$. 
For the construction to work, $(\ell_k+1)^k$ must exist for $k\in L$. Recall that $n \geq 2$. If, for a contradiction, there exists $k'\in L$ such that $\ell_{k'} = n$, then $B_{k'\cdot} = \mathbf{1}$. For $B\mathbf{u}^M = \mathbf{1}$ to hold with $\mathbf{u}^M > \mathbf{0}$ entry-wise, every row in $B$ has to be $\mathbf{1}$, which contradicts the nonsingularity of $B$. Thus, $\ell_k < n$, and $(\ell_k+1)^k$ always exists for $k\in L$. These $m$ points also belong to the face of the MISEPI by Lemma \ref{lemma:tight_pt2}. 

\hspace{12pt} So far, we have obtained $2n+1 - m + m = 2n + 1$ points on the face. We next verify that they are affinely independent. We subtract $(w^0, \mathbf{y}^0, \mathbf{p}^0)$ from all points, and form a $2n\times (2n+1)$ matrix with  
\[\{(w^k, \mathbf{y}^k, \mathbf{p}^k) - (w^0, \mathbf{y}^0, \mathbf{p}^0)\}_{k\in N}, \]
\[\{(w^0, \mathbf{y}^0 + \mathbf{1}^i, \mathbf{p}^0) - (w^0, \mathbf{y}^0, \mathbf{p}^0)\}_{i\in N\setminus M}, \]
\[\{(\tilde{w}^k, \tilde{\mathbf{y}}^k, \mathbf{p}^k) - (w^0, \mathbf{y}^0, \mathbf{p}^0)\}_{k\in L}\]
as its rows. The block sub-matrix $[\mathbf{p}^k - \mathbf{p}^0]_{k\in N}$ contributes $n$ to the total rank because $\mathbf{p}^k - \mathbf{p}^0 = \sum_{i=1}^k \mathbf{1}^{\delta(i)}$ for all $k\in N$. The block sub-matrix $[\mathbf{y}^0 + \mathbf{1}^i - \mathbf{y}^0]_{i\in N\setminus M}$ contributes $n-m$ to the total rank. We perform the elementary row operation of adding the negated row $-(w^k, \mathbf{y}^k, \mathbf{p}^k) +(w^0, \mathbf{y}^0, \mathbf{p}^0)$ to $(\tilde{w}^k, \tilde{\mathbf{y}}^k, \mathbf{p}^k)-(w^0, \mathbf{y}^0, \mathbf{p}^0)$ for every $k\in L$. The resulting row is $(h^k_{(\ell_k)^k} - h^k_{(\ell_k+1)^k}) (-1, \mathbf{1}^{\{(1)^k, \dots, (\ell_k)^k\}}, \mathbf{0})$. Given that $\{h^k_i\}_{i\in N}$ are all distinct for every $k\in N\cup\{0\}$, $h^k_{(\ell_k)^k} - h^k_{(\ell_k+1)^k} \neq 0$. 
We observe that $B$ coincides with $A(\boldsymbol{\ell})$ in the columns $i\in M$; otherwise, $B$ contains all-zero columns and cannot be nonsingular. Therefore, $\mathbf{1}^{\{(1)^k, \dots, (\ell_k)^k\}} = B_{k\cdot}$ for all $k\in L$. 
Matrix $B$ is nonsingular, so it contributes $m$ to the total rank. Thus, the matrix representation of these $2n$ points has a total rank of $n + (n - m) + m = 2n$. We conclude that the MISEPI is facet-defining. 
\end{proof}

We can derive facet-defining inequalities by enumerating $A(\boldsymbol{\ell})$ and characterizing its nonsingular submatrices. We include the following corollary as one example. 

\begin{corollary}
Fix any $\mathbf{p}\in\underline{\mathcal{X}}$ and $\boldsymbol{\delta}\in\mathfrak{S}(N)$. Let $M = \{(1)^k\}_{k\in N\cup\{0\}}\subseteq N$. Then the MISEPI associated with $\mathbf{p}, \boldsymbol{\delta}$ and $(1, \mathbf{1}^M)$ is facet-defining for $\conv{\mathcal{P}^H}$. 
\end{corollary}
\begin{proof}
In this case, let $\ell_k = 1$ for all $k\in N\cup\{0\}$. We construct the incidence matrix $A(\boldsymbol{\ell})$ accordingly. This incidence matrix has zero columns for $i\in N\setminus M$. Let $L\subseteq N\cup\{0\}$ be any set of indices such that $\{(1)^k\}_{k\in L}$ are distinct. We know that $|L| = |M| = m$. Consider the sub-matrix $B$ of $A(\boldsymbol{\ell})$ with rows in $L$ and columns in $M$. This sub-matrix is exactly an $m\times m$ permutation matrix, which satisfies $B^{-1}\mathbf{1} = B^\top\mathbf{1} = \mathbf{1}$. Thus, $\mathbf{u} = \mathbf{1}^M$, which is the indicator vector of $M$.
\end{proof}

\subsection{Examples of structured $H$ and $\conv{\mathcal{P}^H}$} \label{sec:structuredH}
In this section, we illustrate the results from previous discussions using two examples: the general integer continuous mixing set and the binary multi-capacity continuous mixing set. {We conclude this section by applying our framework to two more classes of mixing sets, which are common substructures in capacitated lot-sizing problems. With our framework, we recover the known convex-hull description of the general-integer continuous mixing set in Section \ref{sect:eg_CMIX}. Moreover, we give new polyhedral results for the binary multi-capacity continuous mixing set in Section \ref{sect:eg_bicmix}.}

\subsubsection{General integer continuous mixing set}
\label{sect:eg_CMIX}
The \emph{continuous mixing set} closely resembles the general integer mixing set (see Section \ref{sect:MIX}), but with additional non-negative continuous variables. We denote a continuous mixing set by 
\[\mathcal{P}^{\text{CMIX}} := \{(w, \mathbf{y}, \mathbf{x}) \in \mathbb{R}\times \mathbb{R}^n_+ \times\mathbb{Z}^n: w + x_i + y_i \geq q_i, \forall \, i\in N\},\]
where $\mathbf{q}\in\mathbb{R}^n_+$ such that $1 > q_1 \geq q_2 \geq \dots \geq q_n$. The set $\mathcal{P}^{\text{CMIX}}$ plays a crucial role in the relaxations of lot-sizing, capacitated facility location, and capacitated network design problems. \citet{miller2003tight} extend the mixing inequalities \eqref{eq:mix_a}--\eqref{eq:mix_b} to propose a class of valid inequalities for $\mathcal{P}^{\text{CMIX}} \cap \{w \geq 0\}$. The full convex hull description of $\mathcal{P}^{\text{CMIX}}$ is provided by \citet{van2005continuous} with the \emph{cycle inequalities}.  We next describe how to construct these inequalities. Throughout this section, $G = (N, A)$ is a directed graph, where $A = \{(j,k) \in N\times N : q_j \neq q_k\} \cup\{(j,j) : j\in N\}$ represents the arcs in $G$. Each arc $(j,k)\in A$ with $q_j \neq q_k$ is associated with the following expression: 
\[\phi_{jk}(w, \mathbf{y}, \mathbf{x}) := 
\begin{cases}
w + y_j - [q_k + (q_k - q_j - 1)x_j],  & \text{if } q_j < q_k , \\ 
y_j - (q_k - q_j)x_j,  & \text{otherwise.} 
\end{cases}
\]
Each self-loop $(j,j)\in A$ is associated with 
\[ \phi_{jj}(\mathbf{x},\mathbf{y},w) :=  w + y_j - (q_j - x_j).  \]
Every elementary cycle $C \subseteq A$ in $G$ corresponds to a cycle inequality 
\begin{equation}
\label{eq:cycle_ineq}
\sum_{(j,k)\in C} \phi_{jk}(w, \mathbf{y}, \mathbf{x}) \geq 0. 
\end{equation}
The cycle inequalities associated with elementary cycles are facet-defining for $\conv{\mathcal{P}^{\text{CMIX}}}$, and they fully describe $\conv{\mathcal{P}^{\text{CMIX}}}$ \citep{van2005continuous}. In what follows, we exploit the hidden L$^\natural$-convexity in $\mathcal{P}^{\text{CMIX}}$ and show that our result subsumes $\conv{\mathcal{P}^{\text{CMIX}}}$ as a special case. 
Consider $H:\mathbb{Z}^n \times \mathbb{R}^n_+ \rightarrow \mathbb{R}$ given by 
\begin{equation}
\label{eq:H_CMIX}
H(\mathbf{x},\mathbf{y}) := \max_{i\in N}\{q_i - x_i - y_i\}.
\end{equation}
Notice that $\mathcal{P}^H = \mathcal{P}^{\text{CMIX}}$. In this case, $h^i(x_i) = q_i - x_i$ is a univariate monotone function for every $i\in N$. Thus, Assumption \hyperlink{A1}{1} is satisfied by Remark \ref{remark:A1_eg}, and $F_{u_0, \mathbf{u}}$ is L-convex for every $(u_0, \mathbf{u})\in\mathcal{U}$ according to Remark \ref{remark:LC_Fu0u}.  By Theorem \ref{thm:conv_SEPI}, $\conv{\mathcal{P}^{\text{CMIX}}}$ is fully described by MISEPIs. We remark that $\mathcal{P}^{\text{CMIX}}$ also satisfies all the conditions in Proposition \ref{prop:finite_CV}, which means that we only need the MISEPIs corresponding to the finitely many {$\mathbf{u}\in\mathcal{U}'$} to describe this convex hull. 

\begin{corollary}
Any facet-defining cycle inequality is an MISEPI. 
\end{corollary}

In what follows, we address the question  `which MISEPI corresponds to the facet-defining cycle inequality associated with $C$?' Consider any elementary cycle $C\subseteq A$. We construct $\mathbf{p}\in\underline{\mathcal{X}} = \mathbb{Z}^n$, $\boldsymbol{\delta}\in\mathfrak{S}(N)$, and $(u_0,\mathbf{u})\in\mathcal{U}$ such that the associated MISEPI is identical to the cycle inequality corresponding to $C$. The next observation addresses the case where $C$ is a self-loop. 

\begin{observation}
\label{obs:self_loop}
For each self-loop $C = \{(j,j)\}$ where $j\in N$, the corresponding cycle inequality is $w + y_j \geq q_j - x_j$. Note that $(1, \mathbf{1}^j)\in\mathcal{U}$, and $F_{1, \mathbf{1}^j}(\mathbf{x}) = q_i - x_j$. For the epigraph of $F_{1, \mathbf{1}^j}$, $\eta \geq q_i - x_i$ is the only SEPI associated with all $\mathbf{p}\in \underline{\mathcal{X}}$ and all $\boldsymbol{\delta}\in\mathfrak{S}(N)$. 
The resulting MISEPI, given $(1, \mathbf{1}^j)$ along with any arbitrary pair of $\mathbf{p}$ and $\boldsymbol{\delta}$, is exactly $w + y_j \geq q_j - x_j$. 
\end{observation}

In what follows, we focus on the case where $C$ is not a self-loop. Let the set of nodes involved in $C$ be $N(C)$. We know that $|N(C)| \geq 2$ and that $q_i$ are distinct for all $i\in N(C)$ because the original graph $G$ does not contain arcs that connect distinct $j,k\in N$ with $q_j = q_k$. We define $L(C)\subset N(C)$ to be the set of indices in the cycle from which the arc points backward; that is, $L(C) := \{j\in N(C): (j,k)\in C, q_j < q_k\}$. We let 
\[u_0 = |L(C)|, \quad \text{ and } \quad \mathbf{u} = \mathbf{1}^{N(C)}.\] 
Observe that $|N(C)|-1 \geq u_0 \geq 1$ because $C$ is a cycle with at least two arcs---one of them must point backward, but not all vertices have backward arcs. Thus, $(|L(C)|,  \mathbf{1}^{N(C)})\in\mathcal{U}$. 
In addition, we construct $\mathbf{p}\in\mathbb{Z}^n$ by 
\[p_j = \begin{cases}
-1, & \text{ if } j \in L(C), \\
0, & \text{else if } j\in N(C)\setminus L(C), \\
2, & \text{otherwise, i.e., } j\in N\setminus N(C). 
\end{cases}\]
The permutation $\boldsymbol{\delta}\in\mathfrak{S}(N)$ is obtained as follows. 
We first construct a permutation of $N(C)$. For ease of exposition, we label the elements in $N(C)$ by $\{N(C)_1,\dots,N(C)_{|N(C)|}\}$ such that $q_{N(C)_1} > \dots > q_{N(C)_{|N(C)|}}$. 
We use $N^-_C(i)$ and $N^+_C(i)$ to denote the in-neighbor and out-neighbor of any $i\in N(C)$ in the cycle $C$, respectively. That is, $(N^-_C(i), i), (i, N^+_C(i))\in C$. 
We construct 
\[\delta(i) = N_C^-(N(C)_i), \; \forall\, i \in \{1, \dots, |N(C)|\}. \]
If $N(C) \subsetneq N$, then we complete $\boldsymbol{\delta}$ with an arbitrary permutation of the remaining elements $N\setminus N(C)$. We illustrate the aforementioned construction using an example.

\begin{example}
\label{eg:cycle_MISEPI_eg}
Consider the example of $H(\mathbf{x}, \mathbf{y})$ given in \eqref{eq:H_CMIX} with $\mathbf{q} = (0.8, 0.5, 0.2, 0.1)$. The cycle $C = \{(1,4), (4,3), (3,1)\}$ (see Figure \ref{fig:cycle_eg}) gives rise to the cycle inequality: 
\[2w + y_1 + y_3 + y_4 \geq 1 - 0.7x_1 - 0.4x_3 - 0.9x_4. \]
Here, $N(C) = \{1,3,4\}$, so $\mathbf{u} = \mathbf{1}^{\{1,3,4\}}$. We obtain $L(C) = \{3,4\}$ and $u_0 = 2$. It follows that $\mathbf{p} = (0, 2, -1, -1)$. 
\begin{figure}
    \centering
    \includegraphics[width=0.25\linewidth]{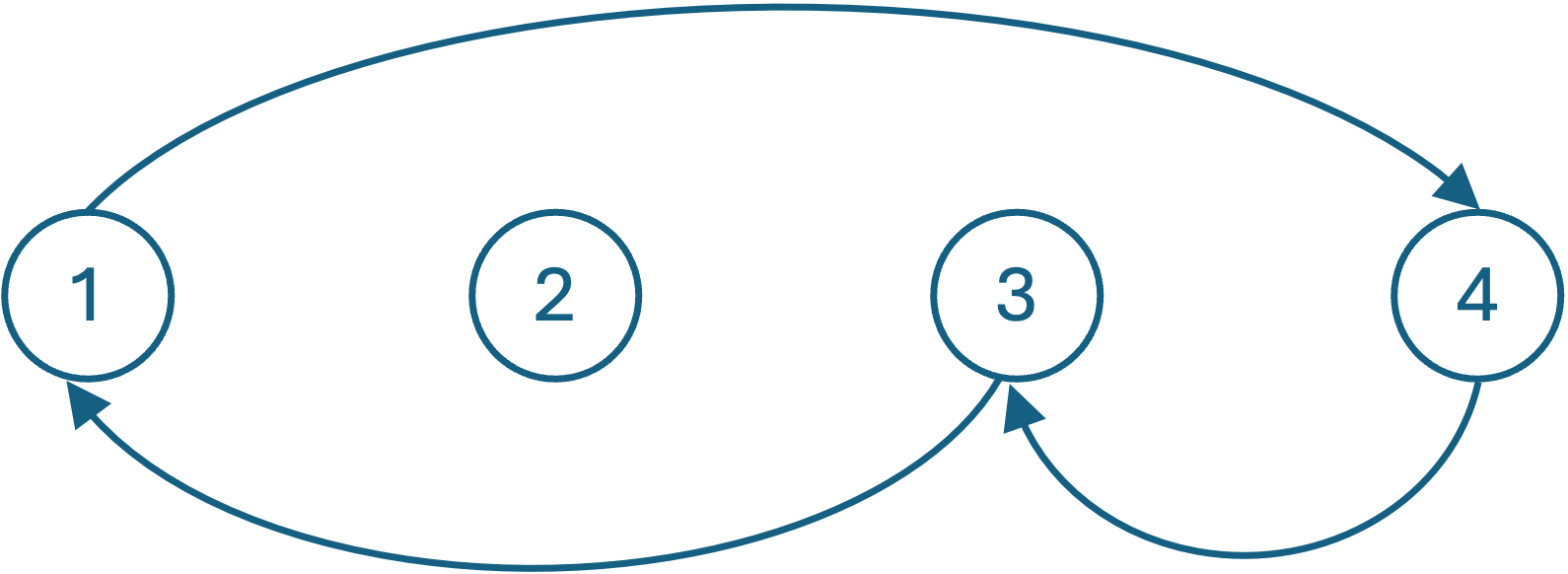}
    \caption{A visual illustration of the cycle $C=\{(1,4), (4,3), (3,1)\}$, where $N(C) = \{1,3,4\}$ and $L(C) = \{3,4\}$ because $(4,3), (3,1)$ are backward arcs.}
    \label{fig:cycle_eg}
\end{figure}
In addition, the elements in $N(C)$ satisfies $q_1 > q_3 > q_4$, so $N(C)_1 = 1, N(C)_2 = 3$, and $N(C)_3 = 4$. We obtain 
\[\delta(1) = N_C^-(N(C)_1) = N_C^-(1) = 3, \quad \delta(2) = N_C^-(N(C)_2) = N_C^-(3) = 4, \quad \delta(3) = N_C^-(N(C)_3) = N_C^-(4) = 1,\]
and we complete the permutation by appending $\delta(4) = 2$. Thus, $\boldsymbol{\delta} = (3,4,1,2)$. The associated MISEPI is 
\begin{align*}
2w + y_1 + y_3 + y_4 \geq \;& F_{2, \mathbf{1}^{\{1,3,4\}}}(0, 2, -1, -1) + [F_{2, \mathbf{1}^{\{1,3,4\}}}(0, 2, 0, -1) - F_{2, \mathbf{1}^{\{1,3,4\}}}(0, 2, -1, -1)](x_3 + 1) \\
\;& + [F_{2, \mathbf{1}^{\{1,3,4\}}}(0, 2, 0, 0) - F_{2, \mathbf{1}^{\{1,3,4\}}}(0, 2, 0, -1)](x_4 + 1) \\
\;& + [F_{2, \mathbf{1}^{\{1,3,4\}}}(1, 2, 0, 0) - F_{2, \mathbf{1}^{\{1,3,4\}}}(0, 2, 0, 0)]x_1 \\
\;& + [F_{2, \mathbf{1}^{\{1,3,4\}}}(1, 3, 0, 0) - F_{2, \mathbf{1}^{\{1,3,4\}}}(1, 2, 0, 0)](x_2 - 2)  \\
= \;& 2.3 + (1.9 - 2.3)(x_3 + 1) + (1-1.9)(x_4 + 1) + (0.3 - 1)x_1 + 0  \\
= \;& 1 - 0.4x_3 -0.9x_4 -0.7x_1. 
\end{align*}
This matches the given cycle inequality. 
\end{example}

Before establishing the equivalence between a cycle inequality and the constructed MISEPI, we highlight an important property of the function $F_{|L(C)|, \mathbf{1}^{N(C)}}$. Recall that Lemma \ref{lemma:F_u0_binary_u} shows an alternative interpretation of $F_{|L(C)|, \mathbf{1}^{N(C)}}$---this function evaluates the highest $|L(C)|$ values (including ties) among $\{h^i(\mathbf{x})\}_{i\in N(C)}$ for any $\mathbf{x}\in\mathcal{X}$. We provide an example to illustrate how the supports of $F_{|L(C)|, \mathbf{1}^{N(C)}}(\mathbf{p}^k)$ and of $F_{|L(C)|, \mathbf{1}^{N(C)}}(\mathbf{p}^{k-1})$ are related. 

\begin{example}
We assume the same setup in Example \ref{eg:cycle_MISEPI_eg}. 
Observe that 
\begin{align*}
\mathbf{p}^0 = (0, 2, -1, -1), \quad & F_{2, \mathbf{1}^{\{1,3,4\}}}(\mathbf{p}^0) = 1.2 + 1.1 = \sum_{i\in \{3,4\}} (q_i - p^0_i), \\
& \{3,4\} = L(C), \\
\mathbf{p}^1 = (0, 2, 0, -1), \quad & F_{2, \mathbf{1}^{\{1,3,4\}}}(\mathbf{p}^1) = 0.8 + 1.1 = \sum_{i\in \{1,4\}} (q_i - p^1_i), \\
& \{1,4\} = \{3,4\}\cup\{N(C)_1\} \setminus \{N^-_C(N(C)_1)\}, \\
\mathbf{p}^2 = (0, 2, 0, 0), \quad & F_{2, \mathbf{1}^{\{1,3,4\}}}(\mathbf{p}^2) = 0.8 + 0.2 = \sum_{i\in \{1,3\}} (q_i - p^2_i),\\
& \{1,3\} = \{1,4\}\cup\{N(C)_2\} \setminus \{N^-_C(N(C)_2)\}, \\
\mathbf{p}^3 = (1, 2, 0, 0), \quad & F_{2, \mathbf{1}^{\{1,3,4\}}}(\mathbf{p}^3) = 0.2 + 0.1 = \sum_{i\in \{3,4\}} (q_i - p^3_i), \\
& \{3,4\} = \{1,3\}\cup\{N(C)_3\} \setminus \{N^-_C(N(C)_3)\}, \\
\mathbf{p}^4 = (1, 3, 0, 0), \quad & F_{2, \mathbf{1}^{\{1,3,4\}}}(\mathbf{p}^4) = 0.2 + 0.1 = \sum_{i\in \{3,4\}} (q_i - p^4_i). 
\end{align*}
Intuitively, the arcs in $C$ govern the supports of $F_{|L(C)|, \mathbf{1}^{N(C)}}(\mathbf{p}^{k})$, which follows from our construction of $\mathbf{p}$ and $\boldsymbol{\delta}$ with respect to $C$.
\end{example}

With this intuition, we discuss the terms in the SEPI  associated with our proposed $\mathbf{p}$ and $\boldsymbol{\delta}$ for the epigraph of $F_{|L(C)|, \mathbf{1}^{N(C)}}$.

\begin{observation}
For any $i\in \{1,\dots,|N(C)|\}$, let $j = N_C^-(N(C)_i) = \delta(i)$. We represent $N(C)_i$ equivalently by $N_C^+(j)$. 
When $j = \delta(i)\in L(C)$, $x_j - p_j = x_j - (-1)$. Given that $h^j(p^{i-1}_j) = q_j + 1$, $j$ must be in the support of $F_{|L(C)|, \mathbf{1}^{N(C)}}(\mathbf{p}^{i-1})$. At $\mathbf{p}^i$, $h^j(p^{i}_j) = q_j < q_{N^+_C(j)} = h^{N^+_C(j)}(p^{i}_{N^+_C(j)})$. In fact, by our construction, the support of $F_{|L(C)|, \mathbf{1}^{N(C)}}(\mathbf{p}^{i})$ is exactly the support of $F_{|L(C)|, \mathbf{1}^{N(C)}}(\mathbf{p}^{i-1})$ with $j$ removed and with $N^+_C(j)$ added. Thus, 
$F_{|L(C)|, \mathbf{1}^{N(C)}}(\mathbf{p}^i) - F_{|L(C)|, \mathbf{1}^{N(C)}}(\mathbf{p}^{i-1}) = q_{N_C^+(j)} - (1+q_j)$.
Similarly, when $j\in N(C)\setminus L(C)$, $x_j - p_j = x_j - 0$. By construction, the support of $F_{|L(C)|, \mathbf{1}^{N(C)}}(\mathbf{p}^{i})$ is again the support of $F_{|L(C)|, \mathbf{1}^{N(C)}}(\mathbf{p}^{i-1})$ with $j$ removed and with $N^+_C(j)$ added. Therefore, $F_{|L(C)|, \mathbf{1}^{N(C)}}(\mathbf{p}^i) - F_{|L(C)|, \mathbf{1}^{N(C)}}(\mathbf{p}^{i-1}) = h^{N^+_C(j)}(p^{i}_{N^+_C(j)}) - h^j(p^{i-1}_j) = q_{N_C^+(j)} - q_j$. In summary,
\begin{equation*}
 \left[F_{|L(C)|, \mathbf{1}^{N(C)}}(\mathbf{p}^i) - F_{|L(C)|, \mathbf{1}^{N(C)}}(\mathbf{p}^{i-1})\right](x_{\delta(i)} - p_{\delta(i)})  = \begin{cases}
\left[q_{N_C^+(j)} - (1+q_j)\right](x_j + 1), & j \in L(C), \\
\left(q_{N_C^+(j)} - q_j\right)x_j, & j \in N(C)\setminus L(C).
\end{cases}
\end{equation*}
\end{observation}

\begin{observation}
For all $i\in \{|N(C)|+1, \dots, n\}$, again, let $j = N_C^-(N(C)_i) = \delta(i)$. First note that $h^j(p^{i-1}_j) = q_j - 2$, which is too low to be in the support of $F_{|L(C)|, \mathbf{1}^{N(C)}}(\mathbf{p}^{i-1})$. 
Second, $h^j(p^i_j) = h^j(p^{i-1}_j + 1) = q_j - 3$, which means that $j$ must not be added to the optimal support of $F_{|L(C)|, \mathbf{1}^{N(C)}}(\mathbf{p}^i)$. Thus, $F_{|L(C)|, \mathbf{1}^{N(C)}}(\mathbf{p}^i) = F_{|L(C)|, \mathbf{1}^{N(C)}}(\mathbf{p}^{i-1})$, and 
\[\left[F_{|L(C)|, \mathbf{1}^{N(C)}}(\mathbf{p}^i) - F_{|L(C)|, \mathbf{1}^{N(C)}}(\mathbf{p}^{i-1})\right](x_{\delta(i)} - p_{\delta(i)})  = 0. \]
\end{observation}

\begin{lemma}
\label{lemma:cycle_SEPI}
The SEPI  associated with our proposed $\mathbf{p}$ and $\boldsymbol{\delta}$ for the epigraph of $F_{|L(C)|, \mathbf{1}^{N(C)}}$ is 
\[\eta \geq \sum_{j\in L(C)} \left[ q_{N_C^+(j)} + \left(q_{N_C^+(j)} - q_j - 1\right)x_j\right] + \sum_{j\in N(C)\setminus L(C)} \left(q_{N_C^+(j)} - q_j\right)x_j. \]
\end{lemma}
\begin{proof}
Due to the two observations, the SEPI  is 
\begin{align*}
\eta & \geq F_{|L(C)|, \mathbf{1}^{N(C)}}(\mathbf{p}^0) + \sum_{i=1}^{|N(C)|} \left[F_{|L(C)|, \mathbf{1}^{N(C)}}(\mathbf{p}^i) - F_{|L(C)|, \mathbf{1}^{N(C)}}(\mathbf{p}^{i-1})\right](x_{\delta(i)} - p_{\delta(i)}) \\
& = \sum_{j\in L(C)} (1 + q_j) + \sum_{j\in L(C)} \left[q_{N_C^+(j)} - (1+q_j)\right](x_j + 1) + \sum_{j\in N(C)\setminus L(C)} \left(q_{N_C^+(j)} - q_j\right)x_j\\
& = \sum_{j\in L(C)} (1 + q_j + q_{N_C^+(j)} - (1+q_j)) +  \sum_{j\in L(C)} \left(q_{N_C^+(j)} - q_j - 1\right)x_j + \sum_{j\in N(C)\setminus L(C)} \left(q_{N_C^+(j)} - q_j\right)x_j\\
& = \sum_{j\in L(C)} \left[ q_{N_C^+(j)} + \left(q_{N_C^+(j)} - q_j - 1\right)x_j\right] + \sum_{j\in N(C)\setminus L(C)} \left(q_{N_C^+(j)} - q_j\right)x_j. 
\end{align*}
\end{proof}

\begin{proposition}
Any facet-defining cycle inequality associated with $C$ is the MISEPI associated with $(u_0, \mathbf{u})$, $\mathbf{p}$ and $\boldsymbol{\delta}$ following our construction. 
\end{proposition}
\begin{proof}
A cycle inequality is facet-defining if it corresponds to an elementary cycle $C$ with distinct $q_i$ for $i\in N(C)$ \citep{van2005continuous}. Observation \ref{obs:self_loop} has fully addressed the case where $C = \{(i,i)\}$ is any self-loop.  Now suppose $|N(C)| \geq 2$. 
The cycle inequality assumes the form 
\[|L(C)| w + \sum_{i\in N(C)} y_i \geq \sum_{j\in L(C), (j,k)\in C} [q_k + (q_k - q_j - 1)x_j] +  \sum_{j\in N(C)\setminus L(C), (j,k)\in C} (q_k - q_j)x_j. \]
By construction, we obtain the MISEPI 
\[|L(C)| w + \sum_{i\in N(C)} y_i \geq \sum_{j\in L(C)} \left[ q_{N_C^+(j)} + \left(q_{N_C^+(j)} - q_j - 1\right)x_j\right] + \sum_{j\in N(C)\setminus L(C)} \left(q_{N_C^+(j)} - q_j\right)x_j.\]
The right-hand side follows from Lemma \ref{lemma:cycle_SEPI}.
This coincides with the cycle inequality. 
\end{proof}

\begin{remark}
Recall that Proposition \ref{prop:finite_CV} provides a finite representation of $\conv{\mathcal{P}^{\text{CMIX}}}$. Here, in fact,
$\mathcal{U}'$ given by \eqref{eq:U_prime} can be refined to the following set
\begin{equation*}
\mathcal{U}' := \bigcup_{M\in 2^N\setminus \{\emptyset\}}\left\{\mathbf{u}\in\mathbb{R}^n_+ : \; 
\begin{aligned}
& \mathbf{u}^M \in \bigcup_{B\in GL(|M|,2)}\{B^{-1}\mathbf{1} : B_{i\cdot}\mathbf{1} = B_{i'\cdot}\mathbf{1}, \forall\, i,i'\in \{1,\dots, |M|\}, \mathbf{u}^{N\setminus M} = \mathbf{0} \\
\end{aligned}
\right\}.
\end{equation*}
That is, every nonsingular binary square matrix $B$ is further restricted to have equal row sums. 
\end{remark}

\subsubsection{Binary multi-capacity continuous mixing set}
\label{sect:eg_bicmix}

In this section, we consider another example of $\mathcal{P}^H$ whose convex hull follows from our discussion. Let arbitrary $\mathbf{q}\in\mathbb{R}^n$ and $\mathbf{c}\in\mathbb{R}^n_+$ be given. We examine $H:\mathcal{B}_\mathbf{0} \times \mathbb{R}^n_+ \rightarrow\mathbb{R}$ given by
\[H(\mathbf{x}, \mathbf{y}) := \max_{i\in N}\{q_i - c_i x_i - y_i\}.\] 
We denote its epigraph by $\mathcal{P}^{\text{MCMIX}}$. This set generalizes $\mathcal{P}^{\text{CMIX}}$ and more realistically captures the varying capacities of the elements involved in lot-sizing, capacitated facility location, or capacitated network design. Multiple studies have considered incorporating divisible weights $\mathbf{c}$, while omitting the continuous variables $\mathbf{y}$. For example,  \citet{constantino2010mixing} examine
\[\{(w,\mathbf{x})\in\mathbb{R}_+\times\mathbb{Z}^n : w + c_ix_i \geq q_i, \; \forall\, i\in N\},\] where $c_i \in \{L, C\}\subset \mathbb{Z}$ and $C | L\in\mathbb{Z}$.  \citet{zhao2008mixing} allow $c_i$, $i\in N$, to be positive real values that are divisible (i.e., $c_1 | \dots | c_n$). In other words, $c_i / c_j \in \mathbb{Z}$ for $i \geq j$. Studies that are more closely related to $\mathcal{P}^H$ in our case are \citet{bansal2015n} and \citet{zhao2008note}. The former proposes a class of valid inequalities called the $n$-step cycle inequalities and gives their facet conditions. The latter shows that when the weights $\mathbf{c}$ are divisible (i.e., $c_1 | \dots |c_n$), optimization over $\mathcal{P}^H$ can be performed in polynomial time. To the best of our knowledge, the description of $\conv{\mathcal{P}^H}$ and the complexity of optimizing over $\mathcal{P}^H$ without any restrictions on $\mathbf{c}$ (e.g., divisibility) remain unanswered. Thanks to the hidden L$^\natural$-convexity and our results in Sections \ref{sect:properties_PH}--\ref{sect:FD_MISEPI}, we now address this challenge. \\

Recall that $\mathbf{c} \geq \mathbf{0}$, so $h^i(x_i) = q_i - c_ix_i$ is univariate and monotonically decreasing for every $i\in N$. Assumption \hyperlink{A1}{1} holds by Remark \ref{remark:A1_eg}, and $F_{u_0, \mathbf{u}}$ are L-convex for all $(u_0, \mathbf{u})\in\mathcal{U}$ by Remark \ref{remark:LC_Fu0u}. 
The next corollary addresses the description of $\conv{\mathcal{P}^{\text{MCMIX}}}$, following from Proposition \ref{prop:finite_CV}.
\begin{corollary}
{(Following from Proposition \ref{prop:finite_CV}.)} The convex hull of $\mathcal{P}^{\text{MCMIX}}$ is completely described by the MISEPIs associated with $(1, \mathbf{u})$, for all $\mathbf{u}\in\mathcal{U}'$ \eqref{eq:U_prime}, with $\mathbf{p} = \mathbf{0}$, and with $\boldsymbol{\delta}\in\mathfrak{S}(N)$. {Here, $\underline{\mathcal{X}} = \{\mathbf{0}\}$ because $\mathcal{X} = \mathcal{B}_\mathbf{0}$, so $\mathbf{p} = \mathbf{0}$ is the only relevant choice.}
\end{corollary}

We include an example to illustrate what these MISEPIs look like. 
\begin{example}
When $\mathbf{q} = (2, 0.5, 4, 2.75)$ and $\mathbf{c} = (3, 1, 4, 2.5)$, the following inequality is facet-defining for $\conv{\mathcal{P}^{\text{MCMIX}}}$. 
\[w + (2/3, 1/3, 1/3, 1/3)^\top\mathbf{y} \geq 35/12 - 3x_1/2 - x_2/12 - 7x_3/6 - x_4/4 \]
Here, $\mathbf{u} = B^{-1}\mathbf{1}$ where 
\[ B = 
\begin{pmatrix}
1 & 1 & 0 & 0 \\
1 & 0 & 1 & 0 \\
1 & 0 & 0 & 1 \\
0 & 1 & 1 & 1
\end{pmatrix}, 
\]
and $\boldsymbol{\delta} = (4,3,2,1)$. 
\end{example}

In addition, we determine the complexity of the associated minimization problem following from Proposition \ref{prop:complexity}. 
\begin{corollary}
The relaxation of the multi-capacitated lot-sizing problem, given by optimizing a linear objective over $\mathcal{P}^{\text{MCMIX}}$, is polynomial-time solvable. 
\end{corollary}

\section{Conclusion}
\label{sect:conclusion}

L$^\natural$-convex functions form a broad class of nonlinear functions with immense practical utility. Our work conducts a comprehensive, in-depth study of the convexification of multiple classes of mixed-integer sets associated with L$^\natural$-convex functions. We establish useful properties of these and related functions, such as lattice and continuous submodular functions. We formalize the linear description for the convex hull of any L$^\natural$-convex function's epigraph with linear inequalities called SEPIs, which are polynomial-time separable. We reveal hidden L$^\natural$-convexity in the mixing set, recovering its existing convex hull description as a special case of our study. We also completely describe the joint convex hulls of multiple L$^\natural$-convex functions sharing common variables, as well as a variant with additional constraints. We show that, in this case, the intersection of individual convex hulls equals the convex hull of the intersection, generalizing analogous results for {submodular set function}s. Lastly, we extend the notion of L$^\natural$-convexity to functions defined over mixed-integer domains. We examine the convex hull description for the epigraph of any such mixed-integer function. Under certain conditions, we provide a complete description via MISEPIs, with specified facet conditions and polynomial separation complexity; the latter implies polynomial complexity of minimizing such functions over mixed-integer variables. We illustrate our results with the continuous mixing set and its multi-capacitated variant by uncovering the hidden L$^\natural$-convexity. A natural step in our future research is to extend our results for L$^\natural$-convexity to its superclass, lattice submodularity---more precisely, to obtain convexification of the epigraph of lattice submodular functions in the original decision space. Combining with our methods in this work, we will then be able to generalize the convex hull description for the multi-capacity continuous mixing set from $\mathcal{X} = \mathcal{B}_{\mathbf{0}}$ to $\mathcal{X} = \mathbb{Z}^n$. {Our convexification of $\mathcal{P}^H$ proceeds in three steps: (i) reformulating $\mathcal{P}^H$ as a disjunction of the sets $\mathcal{D}^j$, $j\in N$; (ii) convexifying each $\mathcal{D}^j$ through the simultaneous convexification result of Section \ref{sect:multi_LC}; and (iii) convexifying the disjunction through a polar characterization combined with an aggregation argument. All three steps exploit a common property---the convex envelopes of the functions $h^i$ and $g^j_i$ involved are determined by the same polyhedral subdivision, which is guaranteed under Assumption \hyperlink{A1}{1}. Formalizing an abstract set of conditions under which this three-step strategy applies beyond L$^\natural$-convexity, in the spirit of the inclusion certificates of \citet{tawarmalani2010inclusion}, is another promising direction that could further broaden the class of mixed-integer sets amenable to this analysis, including variants of mixing sets that do not satisfy Assumption \hyperlink{A1}{1}.}

\section*{Acknowledgments}
{The authors are grateful to the associate editor and the two anonymous referees for their careful reading of the manuscript and for their thoughtful and constructive suggestions, which have substantially improved both the positioning of our results within the discrete convex analysis and convexification literature and the clarity of the exposition. We are particularly indebted to one referee for pointing out the connection between our epigraph description and the known convex-extension results for L$^\natural$-convex functions, and for generously supplying elegant short proofs of Lemmas \ref{lemma:G_partialmin} and \ref{lem:Djx_alt}, which we have adopted in place of our original arguments.}
This work is supported by grants from the Office of Naval Research \#N00014-22-1-2602 and the Natural Sciences and Engineering Research Council of Canada RGPIN-2024-05059.

\section*{Appendix}

\begin{proof}[{Proof of Proposition \ref{prop:F_constant_LNC}}]
First, we show that $F$ is L-convex. For any $\mathbf{x}, \mathbf{y}\in\mathcal{X}$,
\begin{align*}
F(\mathbf{x} \vee \mathbf{y}) + F(\mathbf{x} \wedge \mathbf{y}) & = F(\mathbf{x} \vee \mathbf{y}) + [F(\mathbf{x}) \vee F(\mathbf{y})] \quad \text{(by (iii))} \\
& \leq [F(\mathbf{x}) \wedge F(\mathbf{y})] + [F(\mathbf{x}) \vee F(\mathbf{y})] \\
& \quad  \quad \text{(because $F$ is decreasing, $\mathbf{x} \vee \mathbf{y} \geq \mathbf{x}$, and $\mathbf{x} \vee \mathbf{y} \geq \mathbf{y}$)} \\
& = F(\mathbf{x}) + F(\mathbf{y}). 
\end{align*}
Thus, $F$ is lattice submodular. Due to (ii), $F$ is L-convex, and in particular, 
\[r = F(\mathbf{x} + \mathbf{1}) - F(\mathbf{x}) \leq 0.\]

Next, we prove that $F^c$ is L$^\natural$-convex by showing \eqref{eq:LC_def_sub}. That is, for any $\mathbf{x}, \mathbf{y}\in\mathcal{X}$, and for any $\alpha\in\mathbb{Z}_+$ with $(\mathbf{x} - \alpha \mathbf{1}) \vee \mathbf{y}$, $\mathbf{x} \wedge (\mathbf{y} + \alpha \mathbf{1}) \in\mathcal{X}$, we will show that 
\[F^c(\mathbf{x}) +F^c(\mathbf{y}) \geq F^c\left( (\mathbf{x} - \alpha \mathbf{1}) \vee \mathbf{y} \right) + F^c\left( \mathbf{x} \wedge (\mathbf{y} + \alpha \mathbf{1}) \right).\]

(Case 1) If $F\left( (\mathbf{x} - \alpha \mathbf{1}) \vee \mathbf{y} \right), F\left( \mathbf{x} \wedge (\mathbf{y} + \alpha \mathbf{1})\right) > c$, then 
\begin{align*}
F^c\left( (\mathbf{x} - \alpha \mathbf{1}) \vee \mathbf{y} \right) + F^c\left( \mathbf{x} \wedge (\mathbf{y} + \alpha \mathbf{1}) \right) & = F\left( (\mathbf{x} - \alpha \mathbf{1}) \vee \mathbf{y} \right) + F\left( \mathbf{x} \wedge (\mathbf{y} + \alpha \mathbf{1}) \right) \\
& \leq F(\mathbf{x}) + F(\mathbf{y}) \quad \text{(by L-convexity of $F$ and Theorem \ref{thm:LC_LNC})} \\
& \leq [F(\mathbf{x}) \vee c] + [F(\mathbf{y}) \vee c]  \\
& \leq F^c(\mathbf{x}) +F^c(\mathbf{y}). 
\end{align*}

(Case 2) If $F\left( (\mathbf{x} - \alpha \mathbf{1}) \vee \mathbf{y} \right), F\left( \mathbf{x} \wedge (\mathbf{y} + \alpha \mathbf{1})\right) \leq c$, then 
\begin{align*}
F^c\left( (\mathbf{x} - \alpha \mathbf{1}) \vee \mathbf{y} \right) + F^c\left( \mathbf{x} \wedge (\mathbf{y} + \alpha \mathbf{1}) \right) & = c + c  \leq [F(\mathbf{x}) \vee c] + [F(\mathbf{y}) \vee c]  = F^c(\mathbf{x}) +F^c(\mathbf{y}). 
\end{align*}

(Case 3) If $F\left( (\mathbf{x} - \alpha \mathbf{1}) \vee \mathbf{y} \right) > c$ and $F\left( \mathbf{x} \wedge (\mathbf{y} + \alpha \mathbf{1})\right) \leq c$, then 
\begin{align*}
F^c\left( (\mathbf{x} - \alpha \mathbf{1}) \vee \mathbf{y} \right) + F^c\left( \mathbf{x} \wedge (\mathbf{y} + \alpha \mathbf{1}) \right) & = F\left( (\mathbf{x} - \alpha \mathbf{1}) \vee \mathbf{y} \right) + c \\
& \leq F(\mathbf{y}) + c \quad \text{(because $F$ is decreasing)} \\
& \leq [F(\mathbf{y}) \vee c] + [F(\mathbf{x})\vee c] \\
& \leq F^c(\mathbf{y}) + F^c(\mathbf{x}). 
\end{align*}

(Case 4) Lastly, if $F\left( (\mathbf{x} - \alpha \mathbf{1}) \vee \mathbf{y} \right) \leq c$ and $F\left( \mathbf{x} \wedge (\mathbf{y} + \alpha \mathbf{1})\right) > c$, then 
\begin{align*}
F^c\left( (\mathbf{x} - \alpha \mathbf{1}) \vee \mathbf{y} \right) + F^c\left( \mathbf{x} \wedge (\mathbf{y} + \alpha \mathbf{1}) \right) & = c + F\left( \mathbf{x} \wedge (\mathbf{y} + \alpha \mathbf{1})\right)\\
& = c + [F( \mathbf{x}) \vee F(\mathbf{y} + \alpha \mathbf{1}) ]\quad \text{(by (iii))}\\
& = c + \{F( \mathbf{x}) \vee [F(\mathbf{y}) + \alpha r]\} \quad \text{(by (ii))}\\
& \leq (c \wedge c) + [F(\mathbf{x}) \vee F(\mathbf{y})] \quad \text{($\alpha \in \mathbb{Z}_+$, $r \leq 0$)}\\
& \leq \{[F(\mathbf{x}) \vee c] \wedge [F(\mathbf{y}) \vee c]\} + \{[F(\mathbf{x})\vee c] \vee [F(\mathbf{y})\vee c]\} \\
& = [F(\mathbf{x}) \vee c] + [F(\mathbf{y}) \vee c] \\
& = F^c(\mathbf{x}) +F^c(\mathbf{y}). 
\end{align*}
\end{proof}

\begin{proof}[{Proof of Proposition \ref{prop:max_aff_sub}}]
For any $\mathbf{x}, \mathbf{y}\in\mathcal{X}$, we will show that $F(\mathbf{x}) + F(\mathbf{y})\geq F(\mathbf{x}\vee \mathbf{y}) + F(\mathbf{x}\wedge \mathbf{y})$ by cases. We first make an observation that is useful in the following discussion. Under the given condition on $\mathbf{a}, \mathbf{b}$, there exists at most one positive entry and at most one negative entry in $\mathbf{a} - \mathbf{b}$. All other entries in $\mathbf{a} - \mathbf{b}$ are zeros. For the given arbitrary $\mathbf{x}, \mathbf{y}\in\mathcal{X}$, we let
\[X := \{i\in N: x_i > y_i\}, \quad Y:=\{i\in N: x_i < y_i\},\]
and $E = N\setminus (X\cup Y) =\{i\in N: x_i = y_i\}$.
Suppose $X = \emptyset$. Then $\mathbf{x}\leq \mathbf{y}$ component-wise. This means that $\mathbf{x} \vee \mathbf{y} = \mathbf{y}$ and $\mathbf{x} \wedge \mathbf{y} = \mathbf{x}$. It follows that $F(\mathbf{x}) + F(\mathbf{y}) = F(\mathbf{x}\vee \mathbf{y}) + F(\mathbf{x}\wedge \mathbf{y})$. The same argument applies when $Y = \emptyset$. Thus, it remains to consider the cases where $X\neq \emptyset$ and $Y\neq\emptyset$.
Observe that $X\cap Y = \emptyset$, so there always exists $A \in \{X, Y\}$ such that $a_i - b_i \geq 0$ for all $i\in A$. 
Similarly, there always exists $B \in \{X, Y\}$ such that $a_i - b_i \leq 0$ for all $i\in B$. Here, $A$ and $B$ can be identical. Note that $A, B \neq \emptyset$ because $X$ and $Y$ are both non-empty in this case. 

(Case 1) Suppose $F(\mathbf{x}\vee \mathbf{y}) = \mathbf{b}^\top (\mathbf{x}\vee \mathbf{y}) + b_0$ and $F(\mathbf{x}\wedge \mathbf{y}) = \mathbf{a}^\top (\mathbf{x}\wedge \mathbf{y}) + a_0$. We obtain 
\[F(\mathbf{x}\vee \mathbf{y}) + F(\mathbf{x}\wedge \mathbf{y})  = \sum_{i\in X} b_ix_i + \sum_{i\in Y} b_iy_i + \sum_{i\in E} b_ix_i + b_0 + \sum_{i\in X} a_iy_i + \sum_{i\in Y} a_ix_i + \sum_{i\in E} a_ix_i + a_0.\]
Here, $x_i, y_i$ for $i\in E$ are used interchangeably because $x_i = y_i$ for all $i\in E$.
We know that there exists $A \in \{X, Y\}$ such that $a_i - b_i \geq 0$ for all $i\in A$.
\begin{itemize}
\item Suppose $A = X$. We have 
\begin{align*}
& \hspace{20pt} \sum_{i\in X} b_ix_i + \sum_{i\in Y} b_iy_i + \sum_{i\in E} b_ix_i + b_0 + \sum_{i\in X} a_iy_i + \sum_{i\in Y} a_ix_i + \sum_{i\in E} a_ix_i + a_0 \\
& = \sum_{i\in X} b_ix_i + \sum_{i\in N} b_iy_i - \sum_{i\in X} b_iy_i - \sum_{i\in E} b_iy_i + \sum_{i\in E} b_ix_i + b_0 \\
& \hspace{20pt} + \sum_{i\in X} a_iy_i + \sum_{i\in N} a_ix_i - \sum_{i\in X} a_ix_i - \sum_{i\in E} a_ix_i+ \sum_{i\in E} a_ix_i + a_0 \\
& = \sum_{i\in X} b_i(x_i -y_i) + \mathbf{b}^\top\mathbf{y}+ b_0 + \sum_{i\in X} a_i(y_i-x_i) + \mathbf{a}^\top\mathbf{x} + a_0 \quad \text{(by definition, $x_i = y_i$ for all $i\in E$)}\\
& =  \mathbf{b}^\top\mathbf{y}+ b_0 + \mathbf{a}^\top\mathbf{x} + a_0  + \sum_{i\in X} (b_i - a_i)(x_i -y_i) \\
& \leq  \mathbf{b}^\top\mathbf{y}+ b_0 + \mathbf{a}^\top\mathbf{x} + a_0 \quad \text{($b_i \leq a_i$ and $x_i > y_i$ for $i\in X$)}\\
& \leq F(\mathbf{y}) + F(\mathbf{x}). 
\end{align*}

\item Suppose $A = Y$. Observe that 
\begin{align*}
& \hspace{20pt} \sum_{i\in X} b_ix_i + \sum_{i\in Y} b_iy_i + \sum_{i\in E} b_ix_i + b_0 + \sum_{i\in X} a_iy_i + \sum_{i\in Y} a_ix_i + \sum_{i\in E} a_ix_i + a_0 \\
& = \sum_{i\in N} b_ix_i - \sum_{i\in Y} b_ix_i - \sum_{i\in E} b_ix_i + \sum_{i\in Y} b_iy_i + \sum_{i\in E} b_ix_i + b_0 \\
& \hspace{20pt} + \sum_{i\in N} a_iy_i - \sum_{i\in Y} a_iy_i- \sum_{i\in E} a_iy_i+ \sum_{i\in Y} a_ix_i + \sum_{i\in E} a_ix_i + a_0 \\
& =  \mathbf{b}^\top\mathbf{x}+ b_0 + \sum_{i\in Y} b_i(y_i - x_i) + \mathbf{a}^\top\mathbf{y} + a_0 - \sum_{i\in Y} a_i (y_i - x_i)\\
& = \mathbf{b}^\top\mathbf{x}+ b_0 + \mathbf{a}^\top\mathbf{y} + a_0 + \sum_{i\in Y} (b_i - a_i )(y_i - x_i)\\
& \leq \mathbf{b}^\top\mathbf{x}+ b_0 + \mathbf{a}^\top\mathbf{y} + a_0 \quad \text{($b_i \leq a_i$ and $y_i > x_i$ for $i\in Y$)}\\
& \leq F(\mathbf{x}) + F(\mathbf{y}). 
\end{align*}
\end{itemize}

(Case 2) Suppose $F(\mathbf{x}\vee \mathbf{y}) = \mathbf{a}^\top (\mathbf{x}\vee \mathbf{y}) + a_0$ and $F(\mathbf{x}\wedge \mathbf{y}) = \mathbf{b}^\top (\mathbf{x}\wedge \mathbf{y}) + b_0$. In this case, 
\[
F(\mathbf{x}\vee \mathbf{y}) + F(\mathbf{x}\wedge \mathbf{y})  = \sum_{i\in X} a_i x_i + \sum_{i\in Y} a_i y_i + \sum_{i\in E} a_ix_i + a_0 + \sum_{i\in X}b_iy_i + \sum_{i\in Y} b_ix_i + \sum_{i\in E} b_i x_i + b_0. 
\]
Recall $B \in \{X, Y\}$ such that $a_i - b_i \leq 0$ for all $i\in B$. 
\begin{itemize}
\item Suppose $B = X$. Then 
\begin{align*}
& \hspace{20pt} \sum_{i\in X} a_i x_i + \sum_{i\in Y} a_i y_i + \sum_{i\in E} a_ix_i + a_0 + \sum_{i\in X}b_iy_i + \sum_{i\in Y} b_ix_i + \sum_{i\in E} b_i x_i + b_0 \\
& = \sum_{i\in X} a_i x_i + \sum_{i\in N} a_iy_i - \sum_{i\in X} a_i y_i - \sum_{i\in E} a_i y_i + \sum_{i\in E} a_ix_i + a_0 \\
& \hspace{20pt} + \sum_{i\in X} b_iy_i + \sum_{i\in N} b_i x_i - \sum_{i\in X} b_i x_i - \sum_{i\in E} b_i x_i + \sum_{i\in E} b_i x_i + b_0 \\
& = \sum_{i\in N} a_iy_i + \sum_{i\in X} a_i (x_i - y_i) + a_0 +  \sum_{i\in N} b_i x_i + \sum_{i\in X} b_i(y_i - x_i) + b_0 \quad \text{($x_i = y_i$ for $i\in E$)} \\
& = \mathbf{a}^\top\mathbf{y} + a_0 + \mathbf{b}^\top\mathbf{x} + b_0 + \sum_{i\in X} (a_i - b_i) (x_i - y_i) \\
& \leq \mathbf{a}^\top\mathbf{y} + a_0 + \mathbf{b}^\top\mathbf{x} + b_0 \quad \text{($a_i - b_i \leq 0$ and $x_i > y_i$ for all $i\in X$)} \\
& \leq F(\mathbf{x}) + F(\mathbf{y}). 
\end{align*}

\item Suppose $B = Y$. Then 
\begin{align*}
& \hspace{20pt} \sum_{i\in X} a_i x_i + \sum_{i\in Y} a_i y_i + \sum_{i\in E} a_ix_i + a_0 + \sum_{i\in X}b_iy_i + \sum_{i\in Y} b_ix_i + \sum_{i\in E} b_i x_i + b_0 \\
& = \sum_{i\in N} a_ix_i - \sum_{i\in Y} a_i x_i - \sum_{i\in E} a_i x_i + \sum_{i\in Y} a_i y_i + \sum_{i\in E} a_ix_i + a_0 \\
& \hspace{20pt} + \sum_{i\in N} b_iy_i - \sum_{i\in Y}b_iy_i -\sum_{i\in E}b_iy_i + \sum_{i\in Y} b_ix_i + \sum_{i\in E} b_i x_i + b_0 \\
& =  \sum_{i\in N} a_ix_i + \sum_{i\in Y} a_i (y_i - x_i) + a_0 + \sum_{i\in N} b_iy_i - \sum_{i\in Y} b_i(y_i -x_i)+ b_0 \\
& = \mathbf{a}^\top\mathbf{x} + a_0 +\mathbf{b}^\top\mathbf{y}+ b_0 + \sum_{i\in Y} (a_i - b_i) (y_i - x_i) \\
& \leq \mathbf{a}^\top\mathbf{x} + a_0 +\mathbf{b}^\top\mathbf{y}+ b_0 \quad \text{($a_i - b_i \leq 0$ and $y_i > x_i$ for all $i\in Y$)} \\
& \leq F(\mathbf{x}) + F(\mathbf{y}). 
\end{align*}
\end{itemize}

(Case 3) Suppose $F(\mathbf{x}\vee \mathbf{y}) = \mathbf{c}^\top (\mathbf{x}\vee \mathbf{y}) + c_0$ and $F(\mathbf{x}\wedge \mathbf{y}) = \mathbf{c}^\top (\mathbf{x}\wedge \mathbf{y}) + c_0$ for $(\mathbf{c}, c_0) \in \{(\mathbf{a}, a_0), (\mathbf{b}, b_0)\}$. Then 
\begin{align*}
F(\mathbf{x}\vee \mathbf{y}) + F(\mathbf{x}\wedge \mathbf{y}) & = \mathbf{c}^\top (\mathbf{x}\vee \mathbf{y}) + c_0 + \mathbf{c}^\top (\mathbf{x}\wedge \mathbf{y}) + c_0 \\
& = \mathbf{c}^\top (\mathbf{x} + \mathbf{y}) + 2c_0 \\
& = \mathbf{c}^\top \mathbf{x} + c_0 + \mathbf{c}^\top\mathbf{y} + c_0 \\
& \leq F(\mathbf{x}) + F(\mathbf{y}). 
\end{align*}
Hence, $F$ is lattice submodular. 
\end{proof}

\begin{proof}[{Proof of Proposition \ref{prop:valid_SEPI}}]
We show that for an arbitrary $\hat{\mathbf{x}}\in\mathcal{X}$,
\[f(\hat{\mathbf{x}}) \geq f(\mathbf{p}) + \sum_{i=1}^{n} \left[ f\left(\mathbf{p} + \sum_{j=1}^i \mathbf{1}^{\delta(j)} \right) - f\left( \mathbf{p} + \sum_{j=1}^{i-1} \mathbf{1}^{\delta(j)} \right) \right] \left(\hat{x}_{\delta(i)} - p_{\delta(i)}\right).\]

Suppose $\hat{\mathbf{x}} \in \mathcal{B}_\mathbf{p}$. Because $f$ is a submodular set function over the discrete unit hypercube $\mathcal{B}_\mathbf{p}$, SEPI is valid at $\hat{\mathbf{x}}$. On the other hand, suppose $\hat{\mathbf{x}} \notin \mathcal{B}_\mathbf{p}$. We know that $\overline{f}(\hat{\mathbf{x}})= f(\hat{\mathbf{x}})$ because $\hat{\mathbf{x}}\in\mathbb{Z}^n$. 
We first construct $\mathbf{p}^1 \in \mathbb{R}^n$ with 
\[p^1_{\delta(i)} = p_{\delta(i)} + \frac{n+1-i}{n+1}\] for all $i = 1,2,\dots, n$. By construction, $\mathbf{p}^1 \in \overline{\mathcal{B}}_\mathbf{p}$ and 
\[p^1_{\delta(1)} - p_{\delta(1)} > p^1_{\delta(2)} - p_{\delta(2)} > \dots > p^1_{\delta(n)} - p_{\delta(n)}.\] 
Thus, 
\[\overline{f}(\mathbf{p}^1) = f(\mathbf{p}) + \sum_{i=1}^{n} \left[ f\left(\mathbf{p} + \sum_{j=1}^i \mathbf{1}^{\delta(j)} \right) - f\left( \mathbf{p} + \sum_{j=1}^{i-1} \mathbf{1}^{\delta(j)} \right) \right] \left(p^1_{\delta(i)} - p_{\delta(i)}\right)\] 
by Remark \ref{re:int_conv_lovasz}. Next, we construct $\mathbf{p}^2 \in \mathbb{R}^n$ such that $\mathbf{p}^2 = \hat{\mathbf{x}} + (1+\alpha) (\mathbf{p}^1 - \hat{\mathbf{x}})$. Here, $\alpha \in \mathbb{R}_{++}$ satisfies $\alpha|\mathbf{p}^1 - \hat{\mathbf{x}}|_\infty \leq 1/(2n+2)$. Intuitively, $\mathbf{p}^2$ lies in a small neighborhood of $\{\mathbf{x}\in \mathbb{R}^n : p^1_i - 1/(2n+2) \leq x_i \leq p^1_i + 1/(2n+2), \forall\, i\in N\}$. We observe that $\mathbf{p}^2 \in \overline{\mathcal{B}_\mathbf{p}}$, and $p^2_{\delta(1)} - p_{\delta(1)} \geq p^2_{\delta(2)} - p_{\delta(2)} \geq \dots \geq p^2_{\delta(n)} - p_{\delta(n)}$. Therefore, 
\[\overline{f}(\mathbf{p}^2) = f(\mathbf{p}) + \sum_{i=1}^{n} \left[ f\left(\mathbf{p} + \sum_{j=1}^i \mathbf{1}^{\delta(j)} \right) - f\left( \mathbf{p} + \sum_{j=1}^{i-1} \mathbf{1}^{\delta(j)} \right) \right] \left(p^2_{\delta(i)} - p_{\delta(i)}\right).\] 
Given that $f$ is integrally convex, $\overline{f}$ is convex and 
\[ \overline{f}(\mathbf{p}^1) = \overline{f}\left(\frac{\alpha}{1+\alpha}\hat{\mathbf{x}} + \frac{1}{1+\alpha}\mathbf{p}^2\right)  \leq  \frac{\alpha}{1+\alpha} \overline{f}(\hat{\mathbf{x}}) + \frac{1}{1+\alpha}\overline{f}(\mathbf{p}^2). \]
In other words, 
\begin{align*}
\overline{f}(\hat{\mathbf{x}}) & \geq  \frac{1+\alpha}{\alpha} \overline{f}(\mathbf{p}^1) - \frac{1}{\alpha} \overline{f}(\mathbf{p}^2) \\
& = \frac{1+\alpha}{\alpha} \left\{f(\mathbf{p}) + \sum_{i=1}^{n} \left[ f\left(\mathbf{p} + \sum_{j=1}^i \mathbf{1}^{\delta(j)} \right) - f\left( \mathbf{p} + \sum_{j=1}^{i-1} \mathbf{1}^{\delta(j)} \right) \right] \left(p^1_{\delta(i)} - p_{\delta(i)}\right) \right\} \\
& \quad  - \frac{1}{\alpha} \left\{f(\mathbf{p}) + \sum_{i=1}^{n} \left[ f\left(\mathbf{p} + \sum_{j=1}^i \mathbf{1}^{\delta(j)} \right) - f\left( \mathbf{p} + \sum_{j=1}^{i-1} \mathbf{1}^{\delta(j)} \right) \right] \left(p^2_{\delta(i)} - p_{\delta(i)}\right)\right\} \\
& = f(\mathbf{p}) + \sum_{i=1}^{n} \left[ f\left(\mathbf{p} + \sum_{j=1}^i \mathbf{1}^{\delta(j)} \right) - f\left( \mathbf{p} + \sum_{j=1}^{i-1} \mathbf{1}^{\delta(j)} \right) \right] \left(\frac{1+\alpha}{\alpha} ( p^1_{\delta(i)} - p_{\delta(i)}) - \frac{1}{\alpha} (p^2_{\delta(i)} - p_{\delta(i)}) \right) \\
& = f(\mathbf{p}) + \sum_{i=1}^{n} \left[ f\left(\mathbf{p} + \sum_{j=1}^i \mathbf{1}^{\delta(j)} \right) - f\left( \mathbf{p} + \sum_{j=1}^{i-1} \mathbf{1}^{\delta(j)} \right) \right] \left(\frac{1+\alpha}{\alpha} p^1_{\delta(i)} - \frac{1}{\alpha} p^2_{\delta(i)} - p_{\delta(i)} \right) \\
& = f(\mathbf{p}) + \sum_{i=1}^{n} \left[ f\left(\mathbf{p} + \sum_{j=1}^i \mathbf{1}^{\delta(j)} \right) - f\left( \mathbf{p} + \sum_{j=1}^{i-1} \mathbf{1}^{\delta(j)} \right) \right] \\
& \quad \quad \left[\frac{1+\alpha}{\alpha} \left(\frac{\alpha}{1+\alpha}\hat{x}_{\delta(i)} + \frac{1}{1+\alpha} p^2_{\delta(i)}\right) - \frac{1}{\alpha} p^2_{\delta(i)} - p_{\delta(i)} \right] \\
& =  f(\mathbf{p}) + \sum_{i=1}^{n} \left[ f\left(\mathbf{p} + \sum_{j=1}^i \mathbf{1}^{\delta(j)} \right) - f\left( \mathbf{p} + \sum_{j=1}^{i-1} \mathbf{1}^{\delta(j)} \right) \right] \left(\hat{x}_{\delta(i)} - p_{\delta(i)} \right).
\end{align*}
This completes the proof for the validity of the SEPIs for $\mathcal{P}^f_\mathcal{X}$. 
\end{proof}

\begin{proof}[{Proof of Proposition \ref{prop:SEPI_mix_correspond}}]
The first case is trivial. We prove the second and the third cases. 
Suppose $p_{\min} = 0$. We index the elements in $K$ by $\boldsymbol{\sigma}(K)$ such that $q_{\sigma(1)} \geq \dots \geq q_{\sigma(|K|)}$. It must be that $f(\mathbf{p}^k) = q_{\sigma(1)}$ for $k = 0, \dots, \delta^{-1}(\sigma(1))-1$. Similarly, $f(\mathbf{p}^k) = q_{\sigma(j)}$ for $k = \delta^{-1}(\sigma(j-1)), \dots, \delta^{-1}(\sigma(j))-1$. Lastly, $f(\mathbf{p}^k) = 0$ for $k = \delta^{-1}(\sigma(|K|)), \dots, n$ because $\mathbf{p}^k \geq \mathbf{1}$ for all such $k$. Recall that $q_{\sigma(|K|+1)} = 0$. Therefore, the SEPI is 
\begin{align*}
w & \geq  q_{\sigma(1)} + \sum_{j=1}^{|K|} (q_{\sigma(j+1)} - q_{\sigma(j)})x_{\sigma(j)}  \\
& = \sum_{j=1}^{|K|} (q_{\sigma(j)} - q_{\sigma(j+1)}) + \sum_{j=1}^{|K|} (q_{\sigma(j)} - q_{\sigma(j+1)})(-x_{\sigma(j)}) \\
& = \sum_{j=1}^{|K|} (q_{\sigma(j)} - q_{\sigma(j+1)})(1-x_{\sigma(j)}),
\end{align*}
which matches \eqref{eq:mix_a} with respect to $K$. \\

Suppose $p_{\min} \leq 1$. Recall that $K = \{k_1, \dots, k_L\} \cup \{k'_1, \dots, k'_{L'}\}$, consistent with the notation in Algorithm \ref{alg:build_K}.  Notice that $k'_1 > \dots > k'_{L'} > k_1 > \dots > k_L$, and
\[q_{k'_1} \geq \dots \geq q_{k'_{L'}} \geq q_{k_1} \geq \dots \geq q_{k_L}\]
by construction. We will denote $q_{k_{L+1}} = 0$. We observe that $f(\mathbf{p}^k) = q_{k_1} - p_{\min}$ for $k = 0, \dots, \delta^{-1}(k_1)-1$. Similarly, $f(\mathbf{p}^k) = q_{k_\kappa} - p_{\min}$ for $k = \delta^{-1}(k_{\kappa -1}), \dots, \delta^{-1}(k_\kappa)-1$, where $\kappa\in\{2,\dots, L\}$. If $L' = 0$, then  $f(\mathbf{p}^k) = q_{k_1} - p_{\min} - 1$ for all $k = \delta^{-1}(k_L), \dots, n$. The SEPI is 
\begin{align*}
w & \geq q_{k_1} - p_{\min} + \sum_{\kappa = 1}^{L-1} [(q_{k_{\kappa+1}}- p_{\min}) - (q_{k_{\kappa}}- p_{\min})](x_{k_\kappa} - p_{\min}) + [(q_{k_1} - p_{\min} - 1) - (q_{k_L}- p_{\min})](x_{k_L} - p_{\min}) \\
& = q_{k_1} - p_{\min} + \sum_{\kappa = 1}^{L-1} (q_{k_{\kappa+1}} - q_{k_{\kappa}})(x_{k_\kappa} - p_{\min}) + (q_{k_1} - 1 - q_{k_L})(x_{k_L} - p_{\min}) \\
& = q_{k_1} - p_{\min} + \sum_{\kappa = 1}^{L} (q_{k_{\kappa+1}} - q_{k_{\kappa}})(x_{k_\kappa} - p_{\min}) + (q_{k_1} - 1)(x_{k_L} - p_{\min}) \\
& = q_{k_1} - p_{\min} + \sum_{\kappa = 1}^{L} (q_{k_{\kappa}} - q_{k_{\kappa+1}})(p_{\min} - 1 + 1 - x_{k_\kappa}) + (q_{k_1} - 1)x_{k_L} - (q_{k_1} - 1)p_{\min}\\
& =  q_{k_1} - p_{\min} + q_{k_1}(p_{\min} - 1) + \sum_{\kappa = 1}^{L} (q_{k_{\kappa}} - q_{k_{\kappa+1}})(1 - x_{k_\kappa}) + (q_{k_1} - 1)x_{k_L} - (q_{k_1} - 1)p_{\min} \\
& = - p_{\min} + q_{k_1}p_{\min} + \sum_{\kappa = 1}^{L} (q_{k_{\kappa}} - q_{k_{\kappa+1}})(1 - x_{k_\kappa}) + (q_{k_1} - 1)x_{k_L} - q_{k_1}p_{\min} + p_{\min} \\
& = \sum_{\kappa = 1}^{L} (q_{k_{\kappa}} - q_{k_{\kappa+1}})(1 - x_{k_\kappa}) - (1- q_{k_1})x_{k_L},
\end{align*}
which matches \eqref{eq:mix_b} with respect to $K = \{k_1, \dots, k_L\}$. \\

On the other hand, suppose $L' \geq 1$. Our discussion about the values of $f(\mathbf{p}^k)$ for $k = 0, \dots, \delta^{-1}(k_L)-1$ still holds. For $k = \delta^{-1}(k_L), \dots, \delta^{-1}(k'_1)-1$, $f(\mathbf{p}^k) = q_{k'_1} - p_{\min} - 1$. Similarly,  $f(\mathbf{p}^k) = q_{k'_{\kappa'}} - p_{\min} - 1$ for $k = \delta^{-1}(k'_{\kappa'-1}), \dots, \delta^{-1}(k'_{\kappa'})-1$ where $\kappa'\in\{2, \dots, L'\}$. Lastly, $f(\mathbf{p}^k) = q_{k_1} - p_{\min} - 1$ for all $k = \delta^{-1}(k'_{L'}), \dots, n$. Recall that $q_{k_{L+1}} = 0$. We construct the SEPI as follows. 
\begingroup
\allowdisplaybreaks
\begin{align*}
w & \geq q_{k_1} - p_{\min} + \sum_{\kappa = 1}^{L-1} [(q_{k_{\kappa+1}}- p_{\min}) - (q_{k_{\kappa}}- p_{\min})](x_{k_\kappa} - p_{\min}) \\
& \quad \quad + [(q_{k'_1} - p_{\min} - 1) - (q_{k_L}- p_{\min})](x_{k_L} - p_{\min}) \\
& \quad \quad + \sum_{\kappa' = 1}^{L'-1} [(q_{k'_{\kappa'+1}}- p_{\min} - 1) - (q_{k'_{\kappa'}}- p_{\min} - 1)](x_{k'_{\kappa'}} - p_{\min} - 1) \\
& \quad \quad + [(q_{k_1} - p_{\min} - 1) - (q_{k'_{L'}}- p_{\min}- 1)](x_{k'_{L'}} - p_{\min} - 1) \\
& = q_{k_1} - p_{\min} + \sum_{\kappa = 1}^{L-1} (q_{k_{\kappa+1}} - q_{k_{\kappa}})(x_{k_\kappa} - p_{\min}) \\
& \quad \quad + \sum_{\kappa' = 1}^{L'-1} (q_{k'_{\kappa'+1}} - q_{k'_{\kappa'}})(x_{k'_{\kappa'}} - p_{\min} - 1) \\
& \quad \quad + (q_{k_1} - q_{k'_{L'}})(x_{k'_{L'}} - p_{\min} - 1) \\
& \quad \quad + (q_{k'_1}  - 1)(x_{k_L} - p_{\min}) + (q_{k_{L+1}} - q_{k_L})(x_{k_L} - p_{\min}) \\
& = q_{k_1} - p_{\min} + \sum_{\kappa = 1}^{L} (q_{k_{\kappa+1}} - q_{k_{\kappa}})(x_{k_\kappa} - p_{\min}) \\
& \quad \quad + \sum_{\kappa' = 1}^{L'-1} (q_{k'_{\kappa'+1}} - q_{k'_{\kappa'}})(x_{k'_{\kappa'}} - p_{\min}) + (q_{k_1} - q_{k'_{L'}})(x_{k'_{L'}} - p_{\min}) \\
& \quad \quad + \sum_{\kappa' = 1}^{L'-1} (q_{k'_{\kappa'}} - q_{k'_{\kappa'+1}}) + (q_{k'_{L'}}- q_{k_1}) + (q_{k'_1}  - 1)(x_{k_L} - p_{\min}) \\
& = q_{k_1} - p_{\min} + \sum_{\kappa = 1}^{L} (q_{k_{\kappa}} - q_{k_{\kappa+1}})(p_{\min}-1+1 - x_{k_\kappa}) \\
& \quad \quad + \sum_{\kappa' = 1}^{L'-1} (q_{k'_{\kappa'}} - q_{k'_{\kappa'+1}})(p_{\min}-1+1 -x_{k'_{\kappa'}}) + (q_{k'_{L'}} - q_{k_1})(p_{\min}-1+1 -x_{k'_{L'}}) \\
& \quad \quad + (q_{k'_{1}}- q_{k_1}) + (q_{k'_1}  - 1)x_{k_L} - (q_{k'_1}  - 1)p_{\min} \\
& = q_{k_1} - p_{\min} + \sum_{\kappa = 1}^{L} (q_{k_{\kappa}} - q_{k_{\kappa+1}})(1 - x_{k_\kappa})  + \sum_{\kappa' = 1}^{L'-1} (q_{k'_{\kappa'}} - q_{k'_{\kappa'+1}})(1 -x_{k'_{\kappa'}}) + (q_{k'_{L'}} - q_{k_1})(1 -x_{k'_{L'}}) \\
& \quad \quad (q_{k_1} + q_{k'_1} - q_{k'_{L'}} + q_{k'_{L'}} - q_{k_1})(p_{\min}-1) + (q_{k'_{1}}- q_{k_1}) + (q_{k'_1}  - 1)x_{k_L} - (q_{k'_1}  - 1)p_{\min} \\
& = \sum_{\kappa = 1}^{L} (q_{k_{\kappa}} - q_{k_{\kappa+1}})(1 - x_{k_\kappa})  + \sum_{\kappa' = 1}^{L'-1} (q_{k'_{\kappa'}} - q_{k'_{\kappa'+1}})(1 -x_{k'_{\kappa'}}) + (q_{k'_{L'}} - q_{k_1})(1 -x_{k'_{L'}}) + (q_{k'_1}  - 1)x_{k_L}\\
& \quad \quad + q_{k_1} - p_{\min} + q_{k'_1}p_{\min}-q_{k'_1} + q_{k'_{1}}- q_{k_1} - q_{k'_1} p_{\min} + p_{\min} \\
& = \sum_{\kappa' = 1}^{L'-1} (q_{k'_{\kappa'}} - q_{k'_{\kappa'+1}})(1 -x_{k'_{\kappa'}}) + (q_{k'_{L'}} - q_{k_1})(1 -x_{k'_{L'}}) + \sum_{\kappa = 1}^{L} (q_{k_{\kappa}} - q_{k_{\kappa+1}})(1 - x_{k_\kappa})  + (q_{k'_1}  - 1)x_{k_L}. 
\end{align*}
\endgroup
The last inequality is exactly \eqref{eq:mix_b}. To see this, 
recall that $|K| = L' + L$ and  $q_{k'_1} \geq \dots \geq q_{k'_{L'}} \geq q_{k_1} \geq \dots \geq q_{k_L}$, which gives $\boldsymbol{\sigma}\in\mathfrak{S}(K)$ where $\sigma(j) = k'_j$ for $j = 1,\dots, L'$, and $\sigma(j) = k_j$ for $j = 1+L',\dots,L+L'$. With this alternative notation, the SEPI is 
\[w \geq \sum_{j=1}^{|K|} (q_{\sigma(j)} - q_{\sigma(j+1)})(1 -x_{\sigma(j)}) + (q_{\sigma(1)}  - 1)x_{\sigma(|K|)}. \]
\end{proof}

\begin{proof}[{Proof of Proposition \ref{prop:inter1}}]
Every $(\overline{\mathbf{w}}, \overline{\mathbf{x}}) \in \mathcal{Q}$ satisfies $(\overline{w}_i, \overline{\mathbf{x}})\in \conv{\mathcal{P}^{f_i}_{\mathcal{X}}}$ for all $i\in K$. Thus, $\conv{\mathcal{Q}} \subseteq \mathcal{CQ}$. Next, we show that $\conv{\mathcal{Q}} \supseteq \mathcal{CQ}$. Consider any $(\overline{\mathbf{w}}, \overline{\mathbf{x}}) \in \mathcal{CQ}$. If $\overline{\mathbf{x}}\in\mathcal{X}$, then $(\overline{w}_i, \overline{\mathbf{x}})\in \mathcal{P}^{f_i}_\mathcal{X}$ for all $i\in K$. Therefore, $(\overline{\mathbf{w}}, \overline{\mathbf{x}}) \in \mathcal{Q}\subset \conv{\mathcal{Q}}$. On the other hand, suppose $\overline{\mathbf{x}}\in\overline{\mathcal{X}}\setminus\mathcal{X}$. We apply Algorithm \ref{alg:frac_greedy} on $\overline{\mathbf{x}}$ to obtain $\boldsymbol{\delta}\in\mathfrak{S}(N)$ and $\mathbf{p}\in\underline{\mathcal{X}}$. Note that for all $f_i$, $i\in K$, we obtain the same $\boldsymbol{\delta}$ and $\mathbf{p}$. Let 
\[r_{\delta(j)} = \overline{x}_{\delta(j)} - p_{\delta(j)},\]
 and $\mathbf{p}^j = \mathbf{p} + \sum_{i=1}^j\mathbf{1}^{\delta(i)}$ for $j\in N\cup\{0\}$. In addition, let 
\[\lambda_j = \begin{cases} 
1 - r_{\delta(1)},  & j = 0, \\
r_{\delta(j)} - r_{\delta(j+1)}, & j = 1, \dots, n-1, \\
r_{\delta(n)}, & j = n.
\end{cases}\]
Notice that $\boldsymbol{\lambda}^\top\mathbf{1} = 1$ and $\boldsymbol{\lambda}\geq \mathbf{0}$ by the construction of $\boldsymbol{\delta}$. Thus, $\overline{\mathbf{x}}$ is given by the convex combination of $\{\mathbf{p}^{j}\}_{j=0}^n$ as the following: 
\[\overline{\mathbf{x}} = \mathbf{p}^0 + \sum_{j=1}^n r_{\delta(j)}\left(\mathbf{p}^{j} - \mathbf{p}^{j-1}\right) = \sum_{j=0}^n \lambda_j \mathbf{p}^{j}. \]
In addition, for all $i\in K$, by validity of the SEPI associated with $\boldsymbol{\delta}$ and $\mathbf{p}$ for $\mathcal{P}^{f_i}_\mathcal{X}$, 
\[w_i \geq f_i(\mathbf{p}^0) + \sum_{j=1}^n r_{\delta(j)}\left[f_i\left(\mathbf{p}^{j}\right) - f_i\left(\mathbf{p}^{j-1}\right)\right] = \sum_{j=0}^n \lambda_j f_i\left(\mathbf{p}^{j}\right), \]
for all $(w_i, \overline{\mathbf{x}}) \in \conv{\mathcal{P}^{f_i}_\mathcal{X}}$. Thus, for all $i\in K$, given that $(\overline{w}_i, \overline{\mathbf{x}})\in \conv{\mathcal{P}^{f_i}_\mathcal{X}}$, 
\[\overline{w}_i \geq \sum_{j=0}^n \lambda_j f_i\left(\mathbf{p}^{j}\right). \]
Consider $(\mathbf{w}^j, \mathbf{x}^j) \in \mathbb{R}^k \times \mathcal{X}$ for $j\in N\cup\{0\}$, where 
\[w^j_i = f_i\left(\mathbf{p}^{j}\right), \; \forall\, i\in K, \]
and 
\[\mathbf{x}^j = \mathbf{p}^{j}.\]
We observe that $(\mathbf{w}^j, \mathbf{x}^j) \in \mathcal{Q}$ for all $j\in N\cup\{0\}$, 
$\overline{\mathbf{x}} = \sum_{j=0}^n \lambda_j \mathbf{x}^j$,
and 
$\overline{\mathbf{w}} \geq \sum_{j=0}^n \lambda_j \mathbf{w}^j$. 
This implies that $(\overline{\mathbf{w}}, \overline{\mathbf{x}}) \in \conv{\mathcal{Q}}$. Hence, $\conv{\mathcal{Q}}\supseteq\mathcal{CQ}$, and we conclude that $\conv{\mathcal{Q}} = \mathcal{CQ}$. 
\end{proof}

\begin{proof}[{Proof of Lemma \ref{lem:fju_prop}}]
Recall the definition of $\mathcal{CD}^j$ given by \eqref{eq:CDj_def}. Given L$^\natural$-convexity of $g^j_i$ and $g^j_i + h^j$ for all $i\in N$, there exist $\{\lambda_k, \mathbf{x}^k\}_{k\in K}$ for some $K\subset \mathbb{Z}_{++}$ with $K \neq \emptyset$, such that $\boldsymbol{\lambda} \geq \mathbf{0}$, $\boldsymbol{\lambda}^\top\mathbf{1} = 1$, $\mathbf{x}^k\in\mathcal{X}$ for $k\in K$, 
\[\mathbf{x} = \sum_{k\in K} \lambda_k \mathbf{x}^k, \]
\[w+y_i \geq \sum_{k\in K}\lambda_k \left[g^j_i(\mathbf{x}^k) + h^j(\mathbf{x}^k)\right], \quad \forall\, i\in N, \]
and 
\[y_i \geq \sum_{k\in K}\lambda_k g^j_i(\mathbf{x}^k), \quad \forall\, i\in N.\]
Take any $(\boldsymbol{\alpha}, \boldsymbol{\beta})\in\mathbb{R}_+^{2n}$ with $\sum_{i\in N} \alpha_i = u_0$ and $\alpha_i + \beta_i = u_i$ for $i\in N$. Such $(\boldsymbol{\alpha}, \boldsymbol{\beta})$ exists---for example, let 
\[\alpha_i = u_0 \frac{u_i}{\sum_{i\in N} u_i}, \quad \beta_i = u_i - u_0 \frac{u_i}{\sum_{i\in N} u_i}, \quad \forall\, i\in N. \]
Here, $\beta_i  =  u_i (\sum_{i\in N} u_i - u_0) / (\sum_{i\in N} u_i) \geq 0$ because $(u_0, \mathbf{u}) \in \mathcal{U}$. 
Then 
\[u_0 w + \sum_{i\in N} u_i y_i = w  \sum_{i\in N} \alpha_i + \sum_{i\in N} (\alpha_i + \beta_i) y_i = \sum_{i\in N}\alpha_i (w + y_i ) + \sum_{i\in N} \beta_i y_i. \]
Moreover, 
\begin{align*}
\sum_{i\in N}\alpha_i (w + y_i ) + \sum_{i\in N} \beta_i y_i & \geq \sum_{i\in N}\alpha_i \sum_{k\in K}\lambda_k \left[g^j_i(\mathbf{x}^k) + h^j(\mathbf{x}^k)\right] + \sum_{i\in N} \beta_i \sum_{k\in K}\lambda_k g^j_i(\mathbf{x}^k)  \\
& = \sum_{k\in K}\lambda_k \sum_{i\in N}\alpha_i \left[g^j_i(\mathbf{x}^k) + h^j(\mathbf{x}^k)\right] + \sum_{k\in K}\lambda_k \sum_{i\in N} \beta_i g^j_i(\mathbf{x}^k)  \\
& = \sum_{k\in K}\lambda_k \left\{\sum_{i\in N}\alpha_i h^j(\mathbf{x}^k) + \sum_{i\in N} (\alpha_i + \beta_i ) g^j_i(\mathbf{x}^k) \right\} \\
& = \sum_{k\in K}\lambda_k \left\{u_0 h^j(\mathbf{x}^k) + \sum_{i\in N} u_i g^j_i(\mathbf{x}^k) \right\} \\
& = \sum_{k\in K}\lambda_k f^j_{u_0, \mathbf{u}}(\mathbf{x}^k).
\end{align*}
Hence, $\left(u_0 w + \sum_{i\in N} u_i y_i, \mathbf{x}\right) \in \conv{\mathcal{P}^{f^j_{u_0, \mathbf{u}}}}$. 
\end{proof}

\begin{proof}[Proof of Lemma \ref{lemma:monotone_composite_sub}.]
Given any $i\in N$, 
\[h^i(v) \vee h^i(w) = h^i(v\wedge w), \quad h^i(v) \wedge h^i(w) = h^i(v\vee w),\]
for all $v, w \in\mathcal{X}_i$ by monotonicity of $h^i$. Consider any $\mathbf{x}, \mathbf{y}\in\mathcal{X}$. We note that 
\begin{align*}
G(\mathbf{x}) + G(\mathbf{y}) & = F\left(h^1(x_1), h^2(x_2), \dots, h^n(x_n)\right) + F\left(h^1(y_1), h^2(y_2), \dots, h^n(y_n)\right) \\
& \geq F\left(h^1(x_1) \vee h^1(y_1), h^2(x_2)\vee h^2(y_2), \dots, h^n(x_n)\vee h^n(y_n)\right) \\
& \quad \quad + F\left(h^1(x_1) \wedge h^1(y_1), h^2(x_2)\wedge h^2(y_2), \dots, h^n(x_n)\wedge h^n(y_n)\right)  \quad \text{($F$ is submodular \eqref{eq:submodular})} \\
& = F\left(h^1(x_1\wedge y_1), h^2(x_2\wedge y_2), \dots, h^n(x_n\wedge y_n)\right) \\
& \quad \quad + F\left(h^1(x_1\vee y_1), h^2(x_2\vee y_2), \dots, h^n(x_n\vee y_n)\right) \quad \text{($h^i$ is monotone for $i\in N$)}\\
& = G(\mathbf{x}\wedge \mathbf{y}) + G(\mathbf{x}\vee \mathbf{y}).
\end{align*}
\end{proof}

\begin{proof}[Proof of Lemma \ref{lemma:continuous_max}]
Consider any $\mathbf{x}, \mathbf{y}\in\overline{\mathcal{X}}$. Suppose $i', j'\in N$ satisfy 
\[g(\mathbf{x}\vee \mathbf{y}) = x_{i'}\vee y_{i'}, \quad  g(\mathbf{x} \wedge\mathbf{y}) = x_{j'}\wedge y_{j'},\]
respectively. If $x_{i'} \geq y_{i'}$, then 
\[g(\mathbf{x}\vee \mathbf{y}) + g(\mathbf{x} \wedge\mathbf{y}) = x_{i'} + (x_{j'}\wedge y_{j'}) \leq x_{i'} +  y_{j'} \leq g(\mathbf{x}) + g(\mathbf{y}).\]
On the other hand, if $x_{i'} < y_{i'}$, then 
\[g(\mathbf{x}\vee \mathbf{y}) + g(\mathbf{x} \wedge\mathbf{y}) = y_{i'} + (x_{j'}\wedge y_{j'}) \geq y_{i'} + x_{j'} \leq g(\mathbf{y}) + g(\mathbf{x}).\]
\end{proof}

\begin{proof}[Proof of Lemma \ref{lemma:uni_conv_sub}]
We show that \eqref{eq:sub_alt_def} holds. Fix any $\mathbf{x}\in\mathbb{R}^2$ and any $a, b\in\mathbb{R}_+$. 
If $i = 1$ and $j= 2$, then 
\begin{align*}
g(\mathbf{x} + a\mathbf{1}^i) - g(\mathbf{x}) & = f(x_1 - x_2 + a) - f(x_1 - x_2) \\
& \geq f(x_1 - x_2 - b + a) - f(x_1 - x_2 - b) \\
& = g(\mathbf{x} + a\mathbf{1}^i + b\mathbf{1}^j) - g(\mathbf{x}+ b\mathbf{1}^j). 
\end{align*} The inequality follows from the convexity of $f$. Otherwise, $i = 2$ and $j = 1$, then similarly, 
\begin{align*}
g(\mathbf{x} + a\mathbf{1}^i) - g(\mathbf{x}) & = f(x_1 - x_2 - a) - f(x_1 - x_2) \\
& \geq f(x_1 - x_2 + b - a) - f(x_1 - x_2 + b) \\
& =g(\mathbf{x} + a\mathbf{1}^i + b\mathbf{1}^j) - g(\mathbf{x}+ b\mathbf{1}^j). 
\end{align*}
\end{proof}

\bibliography{L_mixing_ref}{}
\bibliographystyle{apalike}
\end{document}